\documentclass[11pt, reqno]{amsart}
\usepackage[dvipsnames]{xcolor}
\usepackage{geometry, amsmath, amsfonts, amssymb, tikz-cd, mathtools,lscape,multirow}
\usepackage{graphicx, rotating}
\usepackage[bordercolor=gray!20,backgroundcolor=blue!10,linecolor=blue!50,textsize=footnotesize,textwidth=0.8in]{todonotes}
\usepackage{hyperref}

\newtheorem{thm}{Theorem}[section]

\newtheorem{cor}[thm]{Corollary}
\newtheorem{corollary}[thm]{Corollary}
\newtheorem{lem}[thm]{Lemma}
\newtheorem{lemma}[thm]{Lemma}
\newtheorem{prop}[thm]{Proposition}
\newtheorem{proposition}[thm]{Proposition}
\newtheorem{claim}[thm]{Claim}
\newtheorem{theorem}[thm]{Theorem}

\theoremstyle{definition}

\newtheorem{defn}[thm]{Definition}
\newtheorem{definition}[thm]{Definition}
\newtheorem{ex}[thm]{Example}
\newtheorem{example}[thm]{Example}

\newtheorem{remark}[thm]{Remark}

\numberwithin{equation}{section}

\newcommand{\x}{{\tt x}}
\newcommand{\y}{{\tt y}}
\newcommand{\z}{{\tt z}}
\newcommand{\A}{{\tt A}}
\newcommand{\B}{{\tt B}}
\newcommand{\C}{{\tt C}}
\newcommand{\D}{{\tt D}}

\newcommand{\R}{\mathbb{R}}
\newcommand{\N}{\mathbb{N}}
\newcommand{\Q}{\mathbb{Q}}
\newcommand{\Z}{\mathbb{Z}}
\newcommand{\M}{\mathrm{M}}
\newcommand{\W}{\mathrm{W}}

\newcommand{\tr}{{\rm{tr}}}

\newcommand{\sgn}{{\tt{sgn}}}
\newcommand{\MCG}{{\mathcal{MCG}}}
\newcommand{\s}{\sigma}

\newcommand{\T}{\mathcal{T}}
\newcommand{\Type}{{\tt{Type}}}
\newcommand{\msp}{\rightharpoonup} 

\begin{document}

\title{Agol cycles of pseudo-Anosov 3-braids}

\author{Elaina Aceves} 
\address {Department of Mathematics, University of Utah, Salt Lake City, UT 84115}
\email{u6044706@umail.utah.edu}

\author{Keiko Kawamuro}
\address {Department of Mathematics, University of Iowa, Iowa City, IA 52242}
\email{keiko-kawamuro@uiowa.edu}

\date{\today}

\maketitle

\begin{abstract}
An Agol cycle is a complete invariant of the conjugacy class of a pseudo-Anosov mapping class. 
We study necessary and sufficient conditions for equivalent Agol cycles of pseudo-Anosov 3-braids. 
\end{abstract} 

\tableofcontents

\section{Introduction}

Nielsen-Thurston's classification states that every surface homeomorphism is isotopic to either a periodic, reducible or pseudo-Anosov map \cite{Thurston}. 
If a map $\phi$ is pseudo-Anosov, there are two transverse, singular measured foliations on the surface and a dilatation $\lambda >1$ such that $\phi$ stretches along one foliation by $\lambda$ and the other by $1/\lambda$. 

When a surface is an $n$-punctured disk $D_n$, the mapping class group $\MCG(D_n)$ is isomorphic to the 
$n$ stranded braid group $B_n$ \cite{BIRMAN}. 
Suppose that a braid $\beta \in B_n \simeq \MCG(D_n)$ is a pseudo-Anosov mapping class. 


By definition, the dilatation is an invariant of the conjugacy class of a pseudo-Anosov map. 
The dilatation can be the first tool for the conjugacy problem since it is often easy to compute. 
However, the downside of the dilatation is that it is not always an effective conjugacy class invariant as demonstrated in Theorem~\ref{thm:dilatation}: \\

\noindent
{\bf Theorem~\ref{thm:dilatation}. } 
{\em 
There are infinitely many integers 
$x, y$ and $z$, such that 
the braids $\beta = \sigma_1^x \sigma_2^{-1} \sigma_1^{y} \sigma_2^z$ and 
$\beta'=\sigma_1^x \sigma_2^z \sigma_1^y \sigma_2^{-1}$ 
belong to distinct conjugacy classes but 
have the same dilatation  
\[\lambda = \frac{1}{2}(\gamma + \sqrt{\gamma^2-4})\]
where 
\begin{equation}\label{eq:trace}
\gamma=\gamma(x,y,z)=\sgn(xyz)(-2 -x-y+xz+yz +xyz).
\end{equation}
}

Fortunately, there exists a stronger invariant.
The Agol cycle \cite{Agol} (see Definition~\ref{def of Agol cycle}) is another conjugacy class invariant of a pseudo-Anosov mapping class on a surface $S$, and more importantly, it is a complete conjugacy invariant up to the center of $\MCG(S)$.  

An Agol cycle is a sequence of measured train tracks generated by maximal splitting.

A measured train track is a combinatorial object encoding the transverse measured foliation of a pseudo-Anosov map. 
(It is also called an invariant train track since the train track is invariant under action of the pseudo-Anosov map, and often we call it a train track without the adjective measured or invariant for simplicity.)  
It is a graph and each edge is labeled by a positive number called its weight (or measure). 
As depicted in Figure~\ref{SwitchFig}, at each vertex three edges meet tangentially and the weights satisfy the {\em switch condition} $a=b+c$. 
For a detailed definition of a measured train track,
see Section~2 of Agol's paper \cite{Agol} or Chapter 15 of Farb and Margalit's book \cite{Primer}. 
\begin{figure}[h]
\includegraphics[height=2.0cm]{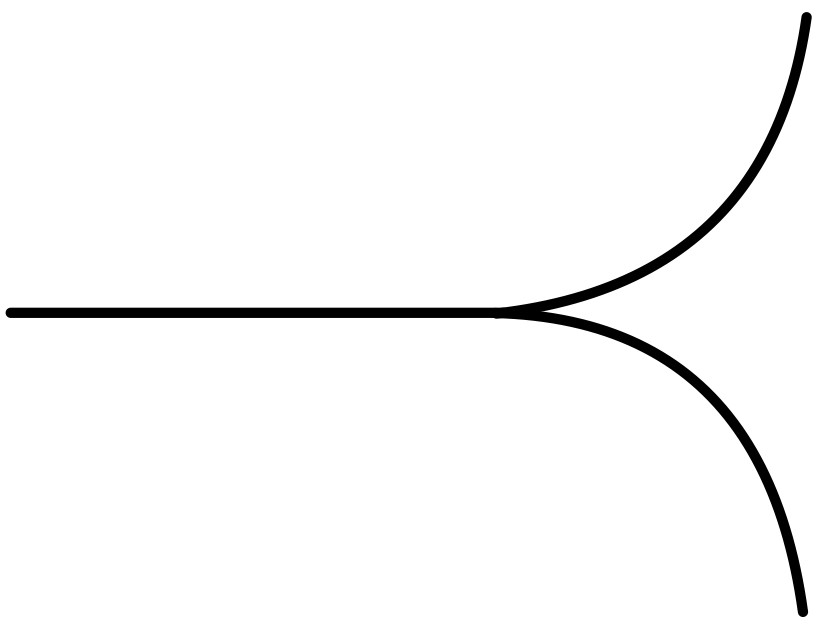}
\put(0, 45){\fontsize{11}{15}$c$}
\put(0, 5){\fontsize{11}{15}$b$}
\put(-60, 32){\fontsize{11}{15}$a$}
\caption{A vertex of a train track where the weights satisfy $a=b+c$.}
\label{SwitchFig}
\end{figure}

Maximal splitting is defined as follows:
 
\begin{definition}\label{def:max-split}
{\em Maximal splitting} ($\rightharpoonup$) on a train track is an operation along all edges with the largest weight simultaneously as depicted in Figure \ref{fig6}. 
Note that maximal splitting preserves the numbers of edges and vertices of the train track. 
\end{definition}
\begin{figure}[h]
\includegraphics[height=2.0cm]{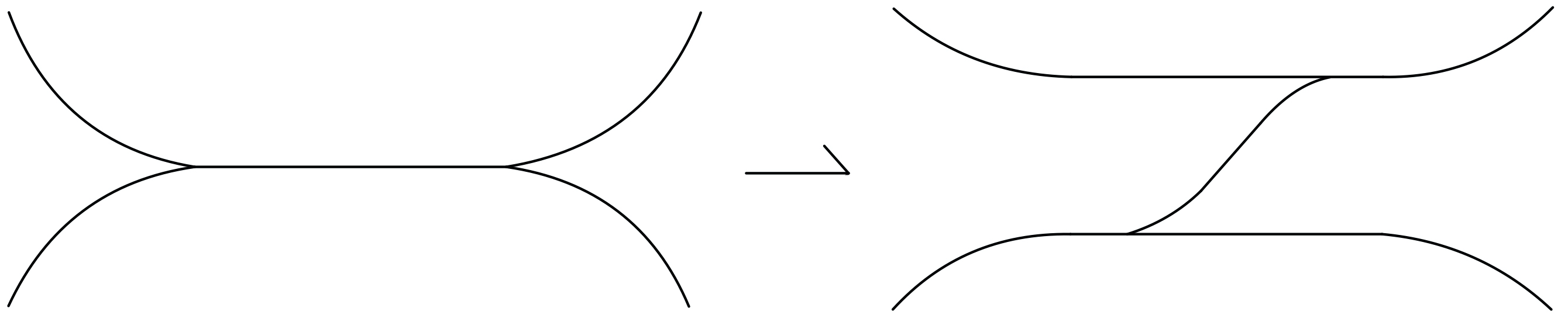}
\put(0, 60){\fontsize{11}{15}$c$}
\put(0, -5){\fontsize{11}{15}$d$}
\put(-130, -5){\fontsize{11}{15}$b$}
\put(-130, 60){\fontsize{11}{15}$a$}
\put(-55, 25){\fontsize{11}{15}$c-a=b-d$}
\put(-290, 60){\fontsize{11}{15}$a$}
\put(-155, 60){\fontsize{11}{15}$c$}
\put(-290, -5){\fontsize{11}{15}$b$}
\put(-155, -5){\fontsize{11}{15}$d$}
\put(-250, 35){\fontsize{11}{15}$a+b=c+d$}
\caption{Maximal splitting when $a<c$. The horizontal edge has the largest weight $a+b=c+d$. After the splitting, the largest weight will be either $b$ or $c$. }
\label{fig6}
\end{figure}
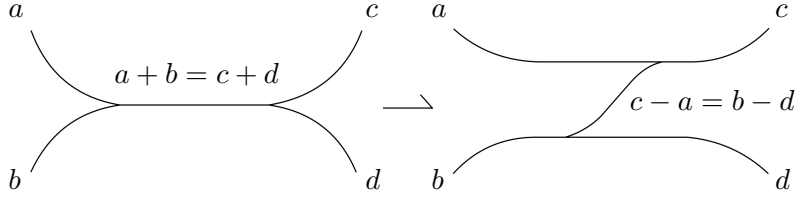

Given a pseudo-Anosov $\beta \in \MCG(S)$ with an invariant train track $\tau_0$, we can create an infinite maximal splitting sequence of train tracks $\tau_0 \rightharpoonup \tau_1\rightharpoonup \tau_2\rightharpoonup \cdots$. 

\begin{definition}\cite{Agol}\label{def of Agol cycle}
An {\em Agol cycle} $$\tau_p\rightharpoonup \cdots \rightharpoonup \tau_{q-1} \rightharpoonup \tau_q= \lambda^{-1}\beta(\tau_p)$$ for $\beta$ is formed when applying $\beta^{-1}$ to some $\tau_{q}$ returns to $\tau_{p}$ where $0\leq p<q$.

The head $\tau_0 \rightharpoonup \tau_1\rightharpoonup \cdots \rightharpoonup \tau_{p-1}$ of the sequence is called the {\em pre-periodic part}. The {\em length of the Agol cycle} is $q-p$ and the {\em length of the pre-periodic part} is $p$. 
\end{definition}

\begin{definition}
We say that two Agol cycles are {\em equivalent} if they are identical up to homeomorphism and scaling. 
Similarly, two Agol cycles are {\em mirror equivalent} if they are equivalent after taking mirror image. 
\end{definition}

\begin{thm}\label{thm:Agol}
\cite[Theorem 3.5 and Section 7]{Agol} 
Every pseudo-Anosov mapping class has an Agol cycle. 
Furthermore, two pseudo-Anosov maps are in the same conjugacy class up to center of $\MCG(S)$ if and only if they have equivalent Agol cycles. 
\end{thm}

A known example of an Agol cycle for a pseudo-Anosov 4-braid can be found in Agol's paper \cite{Agol}. 
Margalit's talk slides \cite{slide} contain an Agol cycle for the pseudo-Anosov 3-braid $\sigma_1 \sigma_2^{-1}$.
In general, it is not easy to compute Agol cycles by hand.
Therefore, necessary or sufficient conditions for equivalent Agol cycles will be helpful to solve the conjugacy problem. 
 
In Theorem~\ref{thm:TFAE}, we give three necessary conditions for equivalent Agol cycles. 
\\

\noindent
{\bf Theorem~\ref{thm:TFAE}} {\em 
Let $\beta$ and $\beta'$ be pseudo-Anosov 3-braids.
The following conditions satisfy: 
$(1) \Leftrightarrow (2) \Rightarrow (3) \Rightarrow (4) \Rightarrow (5).$  
\begin{enumerate}
\item
$\beta$ and $\beta'$ are conjugate in $B_3$ up to a center element. 
\item
$\beta$ and $\beta'$ have the same Agol cycle up to homeomorphism and scaling.
\item
{\em (Topological condition)}
There exist $l$ and $m \in \N$ such that $\sgn(\T_l)=\sgn(\T'_{m})$ and the triple-weight train track sequences 
$
\T_l \msp \T_{l+1} \msp \T_{l+2} \msp \cdots$ (for $\beta$)
and
$
\T'_{m} \msp \T'_{m+1} \msp \T'_{m+2} \msp \cdots $ (for $\beta'$)
give the same periodic I-II-I'-II'-sequence (cf. Definition~\ref{def:I-II-I'-II'-sequence}).
\item
{\em (Number theoretic condition)}
There exist $l$ and $m \in \N$ such that 
$\sgn(\T_l)=\sgn(\T'_{m})$ and the nested Farey interval sequences 
$J_l \supset J_{l+1} \supset J_{l+2} \supset \cdots$ for $\beta$ and 
$J'_{m} \supset J'_{m+1} \supset J'_{m+2} \supset \cdots$ for $\beta'$ 
give the same periodic ${\bf LR}$-sequences (cf. Definition~\ref{def:LandR}). 
\item
{\em (Numerical condition)}
There exist $l$ and $m \in \N$ such that the 4-ratios (cf. Definition~\ref{def of 4-tuple and ratio})
of $\T_l$ and $\T'_m$ are the same. \\
\end{enumerate} 
}

The equivalence $(1) \Leftrightarrow (2)$ is due to Agol.  See \cite[Section 7]{Agol} and also Margalit's talk slides \cite{slide}. 
Our new conditions (3), (4) and (5) have different characteristics as follows:

An Agol cycle contains a large amount of information and each train track in the cycle can be very complicated with many twists. After untwisting (which is a homeomorphism operation), we will show that there are only four types of train tracks that appear in a cycle (Type I, II, I', II' defined in Figure~\ref{figure-121'2'}). 
In Condition (3), we focus on the topological types (Type I, II, I', II') and forget the numerical data of the weights of the train track edges.  

In Condition (4), we estimate a particular algebraic number $\alpha$ (the MP-ratio defined in Proposition~\ref{prop:unique-alpha}) associated to the pseudo-Anosov braid in the Farey tessellation. 
Since the number $\alpha$ is irrational, there exists an infinite sequence of nested intervals ($J_l \supset J_{l+1} \supset J_{l+2} \supset \cdots$) in the Farey tessellation that converges to $\alpha$. 
For each pair of adjacent intervals $J_t \supset J_{t+1}$, we may forget their exact location in the Farey tessellation, and instead focus on the relative location of the sub-interval $J_{t+1}$ with respect to $J_t$. In the Farey tessellation, $J_t$ splits into two sub-intervals. We record whether $J_{t+1}$ is in the left-subinterval ({\bf L}) or the right-subinterval ({\bf R}) of $J_t$. This is how we generate an ${\bf LR}$-sequence.

In Condition (5), we focus on just one train track in the infinite sequence. 
In contrast to (3), we forget the topological side and focus on the numerical data of the edge weights of train tracks.  
Here is a useful consequence of (5): If the algebraic numbers (MP-ratios) $\alpha$ and $\alpha'$ do not satisfy $\A+\B \alpha+\C \alpha '+\D \alpha \alpha'=0$ for any $\A, \B, \C, \D \in \Z$, then the 3-braids $\beta$ and $\beta'$ are not conjugate.

A concrete example is presented in Example~\ref{examples beta and beta'}.

In Theorem~\ref{BigThm}, we explore the converse direction of (2) $\Rightarrow$ (5) in Theorem~\ref{thm:TFAE}. 
\\

\noindent
{\bf Weak version of Theorem~\ref{BigThm}} {\em 
Let $\beta$ and $\beta'$ be pseudo-Anosov 3-braids. 
Suppose that there exist $l$ and $m$ such that triple-weight train tracks $\mathcal{T}_l$ and $\mathcal{T}_m'$ have the same 4-ratio. Then the Agol cycles of $\beta$ and $\beta'$ are equivalent or mirror equivalent.\\
}

In particular, we see that condition (5) of Theorem~\ref{thm:TFAE} cannot be a sufficient condition for equivalent Agol cycles. 
Examples for Theorem~\ref{BigThm} are given in Example~\ref{ex:beta''}.

\subsection*{Acknowledgements}
The authors would like to thank Dan Margalit for useful conversation and talk slides, Ian Agol for interesting conversation, and Manuel Albrizzio for the instructive initial draft of the Matlab code to predict Agol cycles.  
EA was partially supported by the Ford Foundation and the Alfred P. Sloan Foundation and 
KK was partially supported by NSF DMS2005450.

\section{Stability of maximal splitting sequence}

In this section our goal is to study the stability property of a maximal splitting sequence. 
In Theorem~\ref{ThmTTT}, we prove that a maximal splitting sequence arrives at a train track with only three distinct weights and any subsequent train track also has triple-weight type. 

\begin{definition}
We define two types of measured train tracks called Types M and W as in Figure \ref{fig4}. 
Each type has exactly six edges and four vertices.
With the weights $a, b >0$, the left measured train track is denoted by M$(a,b)$ and the right by W$(a, b)$. 
The weights are considered up to scaling. 
\end{definition}

\begin{figure}[h]
\includegraphics[height=2.0cm]{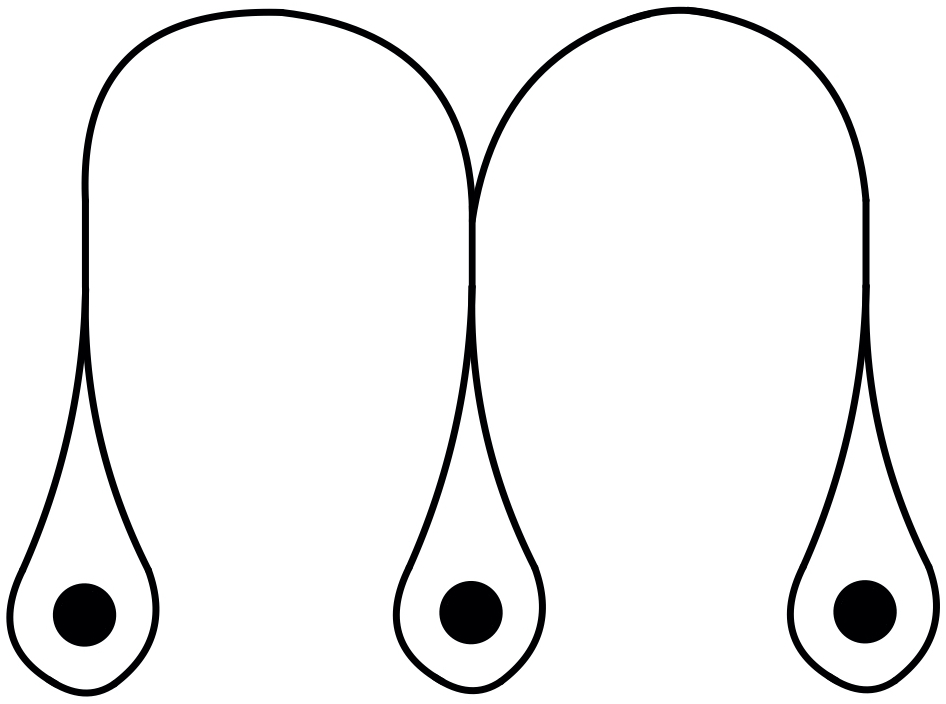} \quad 
\includegraphics[height=2.0cm]{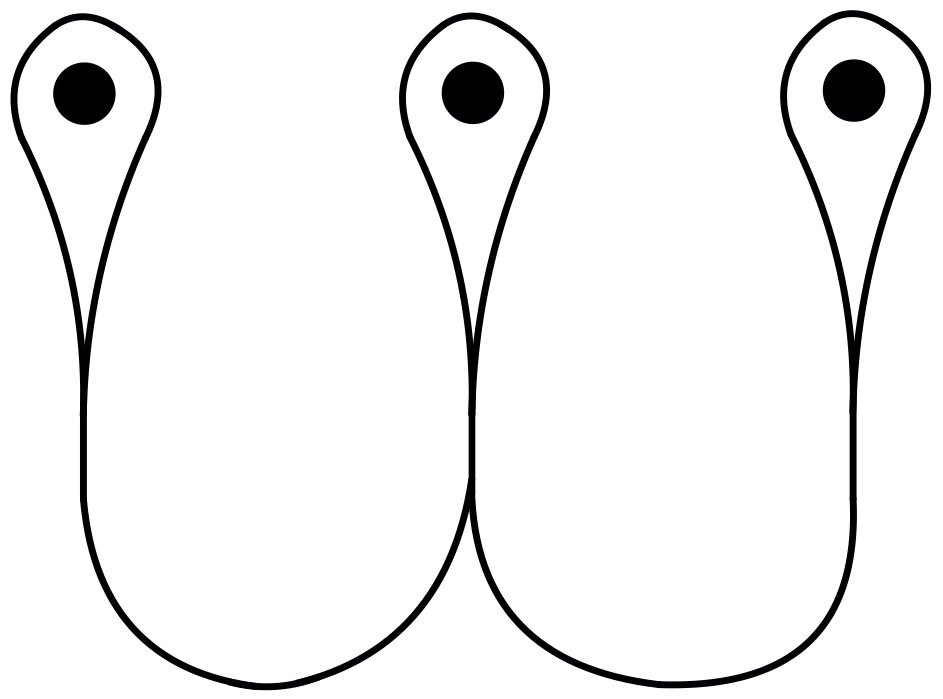}
\put(-57,5){$a$}
\put(-25,5){$b$}
\put(-150,45){$a$}
\put(-115,45){$b$}
\caption{Train tracks M$(a, b)$ and W$(a, b)$}
\label{fig4}
\end{figure}

Proposition~\ref{prop:unique-alpha} follows from the fact that the projective measured foliation space for a 4-punctured sphere is homeomorphic to $\R P^1= S^1$, see Figure 15.6 in \cite{Primer}.

\begin{proposition}\label{prop:unique-alpha}
For every pseudo-Anosov braid $\beta$, there exists a unique irrational number $\alpha\in (0,1)$ such that exactly one of 
$$
\M_\alpha:=\M(1, \alpha), \ \M'_\alpha:=\M(\alpha, 1), \ \W_\alpha:=\W(\alpha, 1), \ \W'_\alpha:=\W(1, \alpha)
$$ 
yields a train track of $\beta$. 
Since $\alpha$ describes the width-height ratio of the Markov partition rectangle, we call $\alpha$ the {\em MP-ratio} for $\beta$.
\end{proposition} 

In the proof of Lemma~\ref{Matrices}, we show how to compute a transition matrix and determine the type (Type $\M$ or $\W$) for certain pseudo-Anosov 3-braids.
It is straightforward to generalize the construction to general pseudo-Anosov 3-braids. 
With a transition matrix in hand, the dilatation $\lambda$ is its largest eigenvalue.
The eigenvector for $\lambda$ is of the form $(\alpha, 1)$ or $(1, \alpha)$, where $\alpha$ is the MP-ratio.
Since the dilatation $\lambda$ is irrational, the MP-ratio is also an irrational number.

Note that $\M_\alpha$ and $\W_\alpha$ are related to each other by a 180$^\circ$-rotation, and so are $\M'_\alpha$ and $\W'_\alpha$. 
We further noter that $\M_\alpha$ and $\M'_\alpha$ are related by a mirror reflection across a vertical line, and so are $\W_\alpha$ and $\W'_\alpha$.

Based on Proposition~\ref{prop:unique-alpha}, we present the Triple-Weight Train Track Theorem:

\begin{theorem}\label{ThmTTT}
Let $\alpha\in (0,1)$ be an irrational number and $n=\lfloor \frac{1}{\alpha} \rfloor$. Thus, $\frac{1}{n+1} < \alpha <\frac{1}{n}$. 
Consider the maximal splitting sequence starting from $\W(\alpha, 1)$:
$$\W_\alpha=\tau_0 \rightharpoonup \tau_1 \rightharpoonup \tau_2 \rightharpoonup \cdots$$ 
The train track $\tau_{n+4}$ in the sequence is shown in Figure~\ref{figure-T_n}. 
We call it a {\em triple-weight train track} as it has only three different weights.
This is the first triple-weight train track in the sequence and 
all the train tracks after $\tau_{n+4}$ have triple-weight type.

\begin{figure}[ht]\label{figure-T_n}
\hspace{-2cm}
\raisebox{60pt}{$ \tau_{n+4}= $ }
\includegraphics[height=4.0cm]{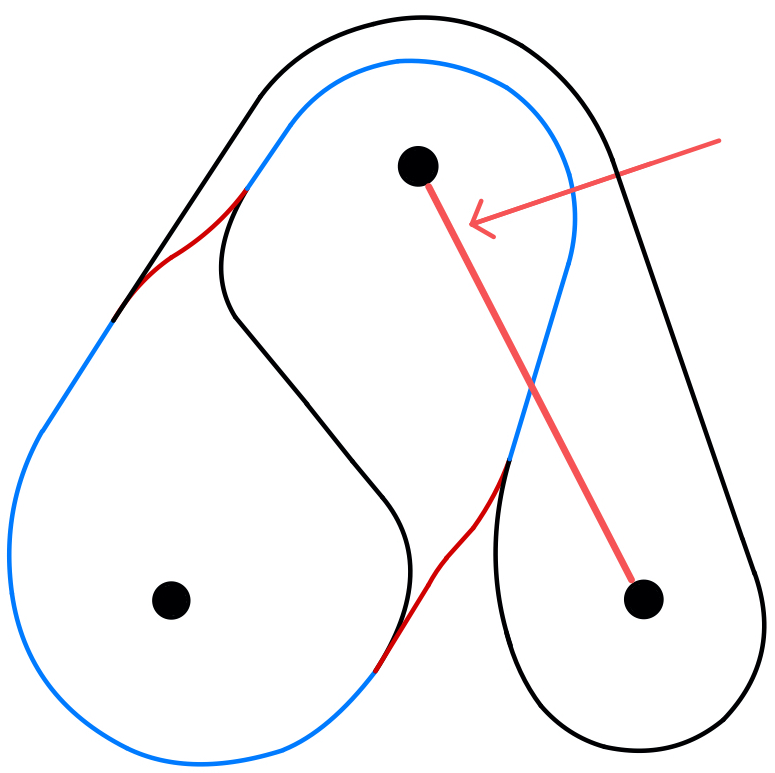}  
\put(-5, 100){\fontsize{11}{11}{\color{WildStrawberry}$n$ negative}}
\put(-5, 90){\fontsize{11}{11}{\color{WildStrawberry}half twists}}
\put(-40, -25){\fontsize{13}{11}{\color{red}$\frac{-1+(n+1)\alpha}{2}$}}
\put(40, -25){\fontsize{13}{11}$\frac{1-n\alpha}{2}$}
\put(80, -25){\fontsize{13}{11}{\color{blue}$\frac{\alpha}{2}$}}
\raisebox{60pt}{$\xleftrightarrow{\text{Isotopic }}\ $} 
\includegraphics[height=4.0cm]{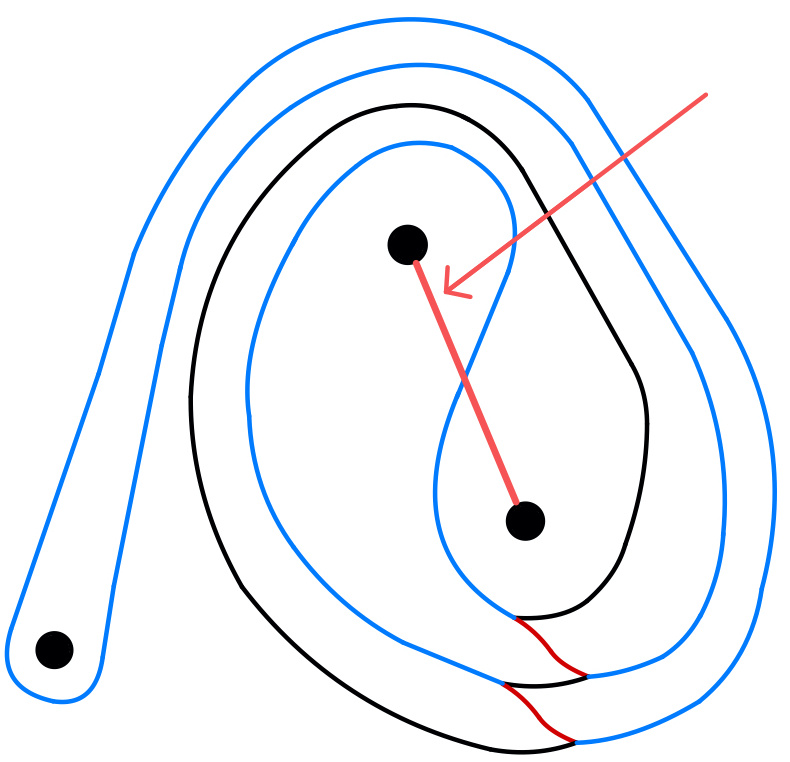} 
\put(-10, 105){\fontsize{11}{11}{\color{WildStrawberry}$(n-2)$ negative}}
\put(-10, 93){\fontsize{11}{11}{\color{WildStrawberry}half twists}}
\caption{The train track $\tau_{n+4}$. The weight of a red edge is $\frac{-1+(n+1)\alpha}{2}$ and of a black edge is $\frac{1-n\alpha}{2}$. 
The blue edges have the largest weight which is $\frac{\alpha}{2}=\frac{-1+(n+1)\alpha}{2}+\frac{1-n\alpha}{2}$. 
}
\label{figure-T_n}
\end{figure}
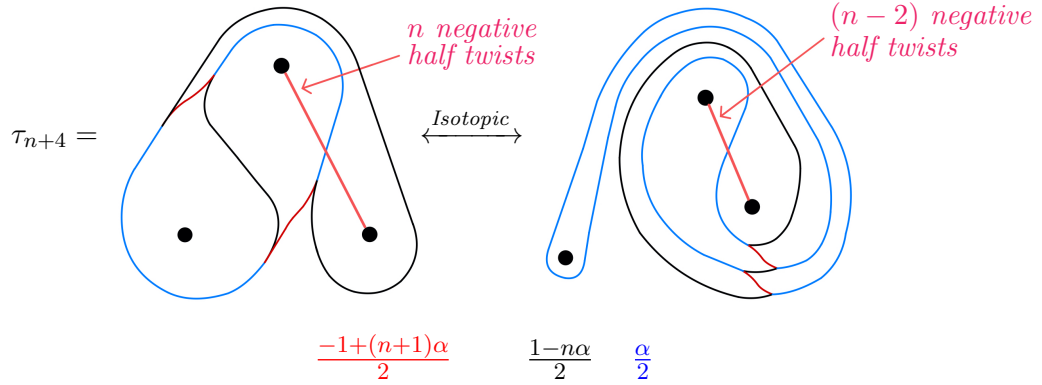

If $\tau_0=\W'_\alpha$, $\M'_\alpha$, or $\M_\alpha$, then the same result holds with the reflection of the figure about a vertical axis, the reflection about a horizontal axis, or a $180^\circ$-rotation, respectively. 
\end{theorem}

An immediate consequence is: 
\begin{cor}
The length $l$ of the pre-periodic part must satisfy $l \geq n+4$. 
\end{cor}

Here is a lemma that is needed to prove Theorem~\ref{ThmTTT}. 

\begin{lem} \label{LemTT}
Suppose that $\tau_0=W(\alpha, 1)$. If $n = \lfloor \frac{1}{\alpha} \rfloor \geq 3$, then 
$\tau_{n+1}=D_n$ as in Figure~\ref{fig:traintrackDn}. 
\end{lem} 

\begin{figure}[ht]
\hspace{-3cm}
\raisebox{70pt}{\fontsize{11}{15}$D_n=$} \qquad 
\includegraphics[height=4.5cm]{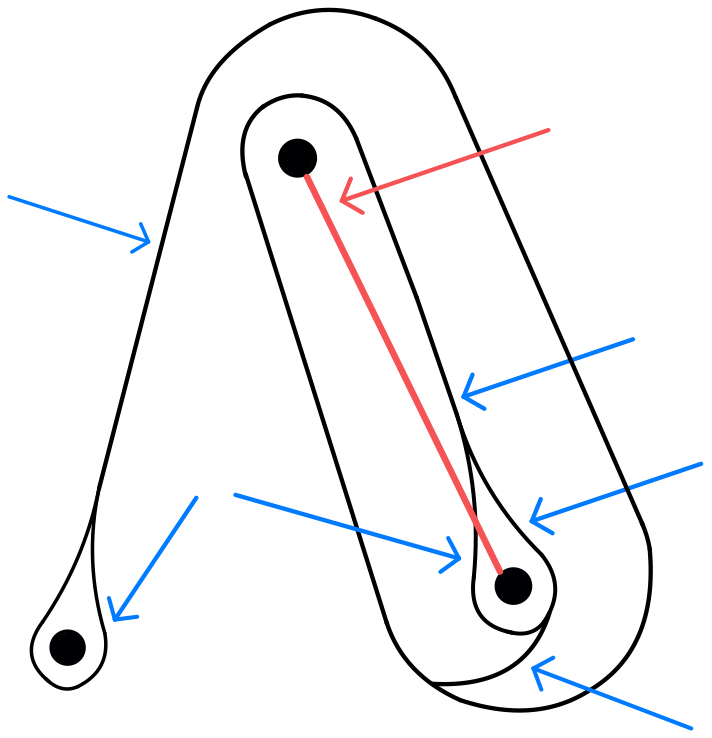}
\put(-135, 95){\fontsize{11}{15}$\alpha$}
\put(-90, 50){\fontsize{11}{15}$\frac{\alpha}{2}$}
\put(-25, 105){\fontsize{11}{15}{\color{WildStrawberry}$(n-3)$ negative half twists}}
\put(-10, 75){\fontsize{11}{15}$\frac{1-(n-2)\alpha}{2}$}
\put(0, 45){\fontsize{11}{15}$\frac{1-(n-1)\alpha}{2}$}
\put(0, 0){\fontsize{11}{15}$\frac{1-n\alpha}{2}$}
\caption{Train track $D_n$}
\label{fig:traintrackDn}
\end{figure}
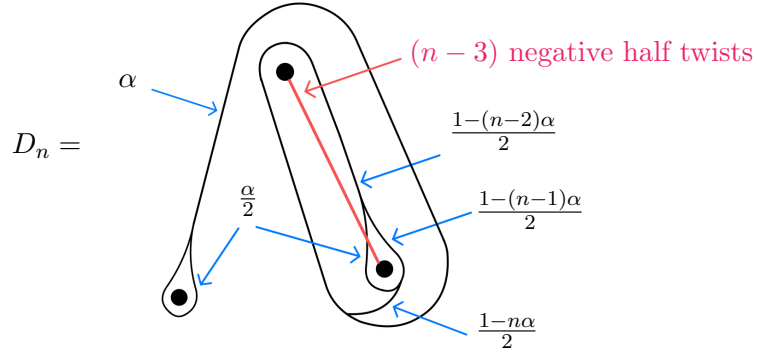

\begin{proof}
This is a proof by induction. 
We begin with the base case of $n=3$ and detail the splittings below with the labeled edges. 

\begin{eqnarray*}
\raisebox{40pt}{$\tau_0=$ } &&
\includegraphics[height=3.0cm]{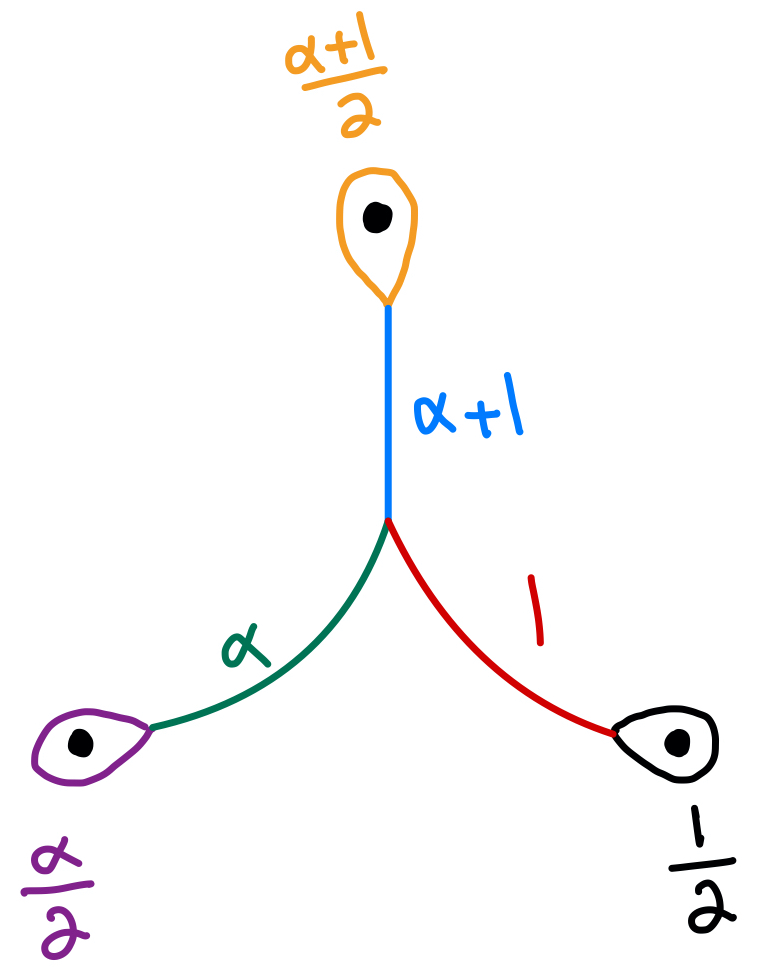}
\raisebox{40pt}{$\xrightharpoonup{\text{Splitting } \alpha+1} \ \tau_1 = $} 
\includegraphics[height=3.0cm]{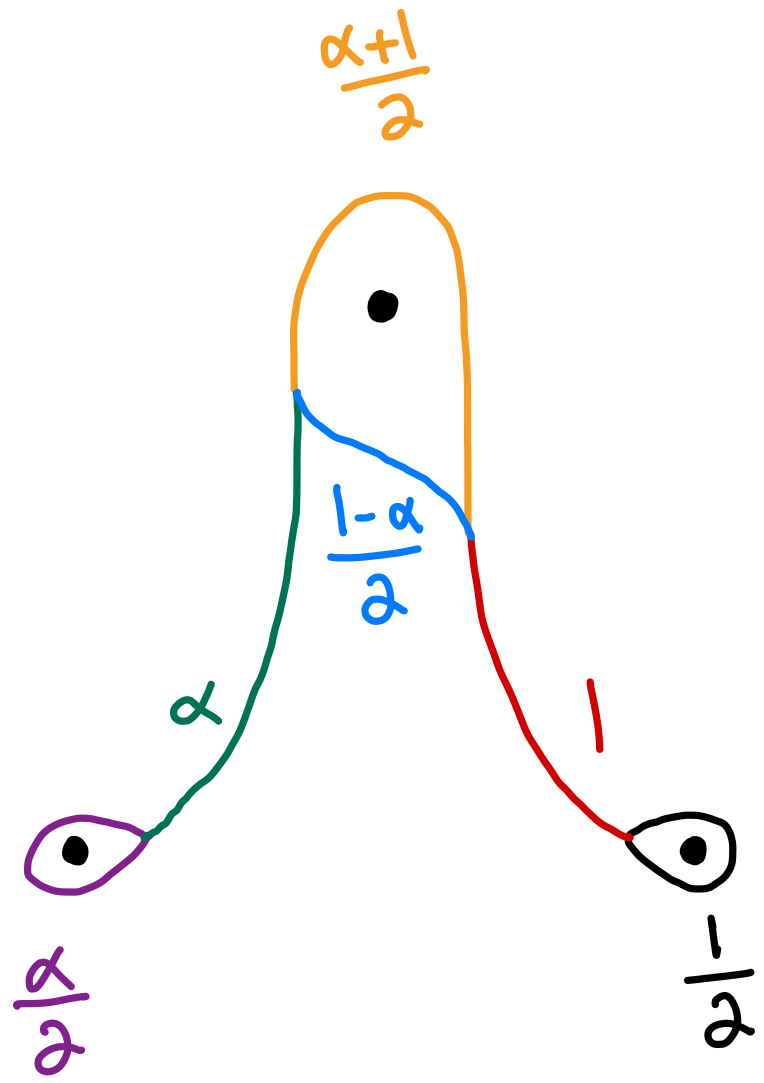}
\raisebox{40pt}{$\xrightharpoonup{\text{Splitting } 1} \ \tau_2=$} 
\includegraphics[height=3.0cm]{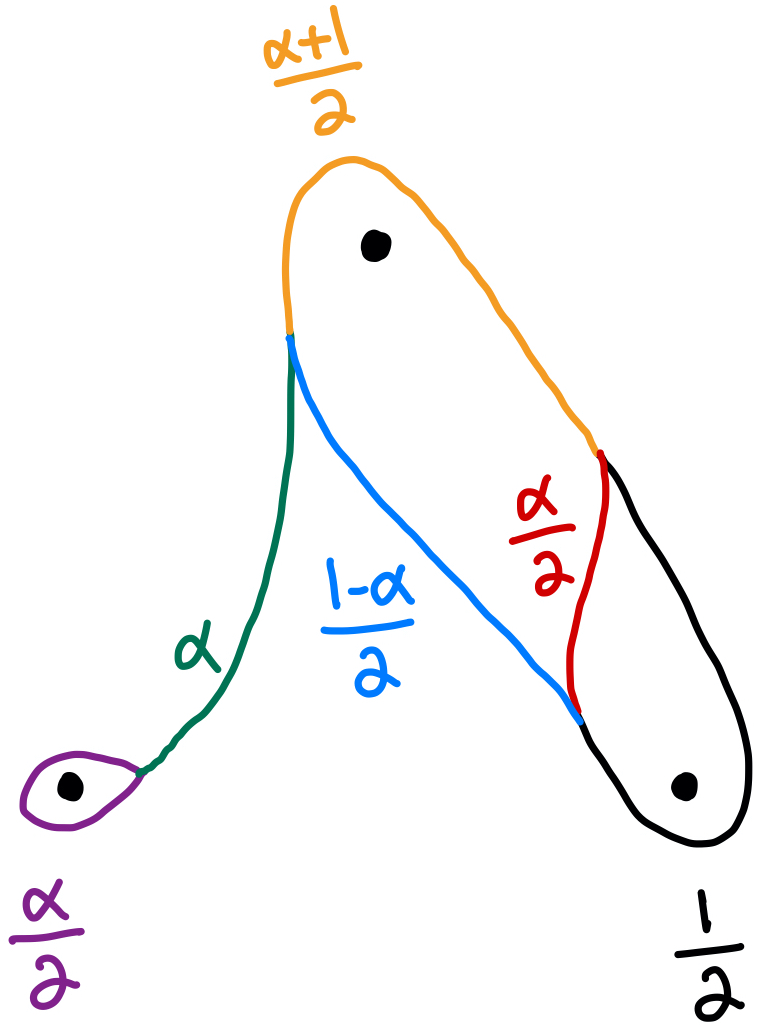} \\
&& \raisebox{40pt}{$\xrightarrow{\text{Isotopy}}$} 
\includegraphics[height=3.0cm]{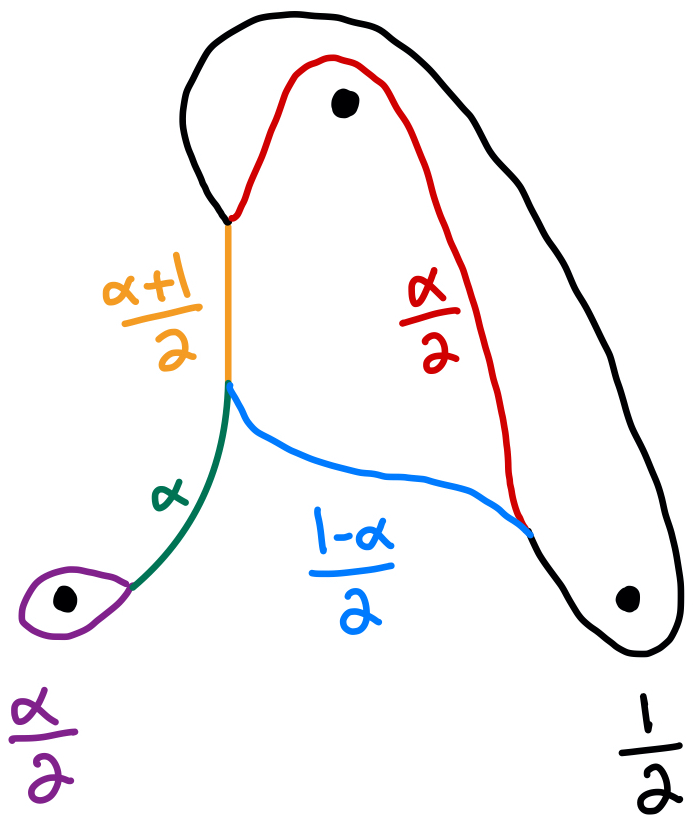}   
\raisebox{40pt}{$\xrightharpoonup{\text{Splitting } \frac{\alpha+1}{2}} \ \tau_3=$} 
\includegraphics[height=3.0cm]{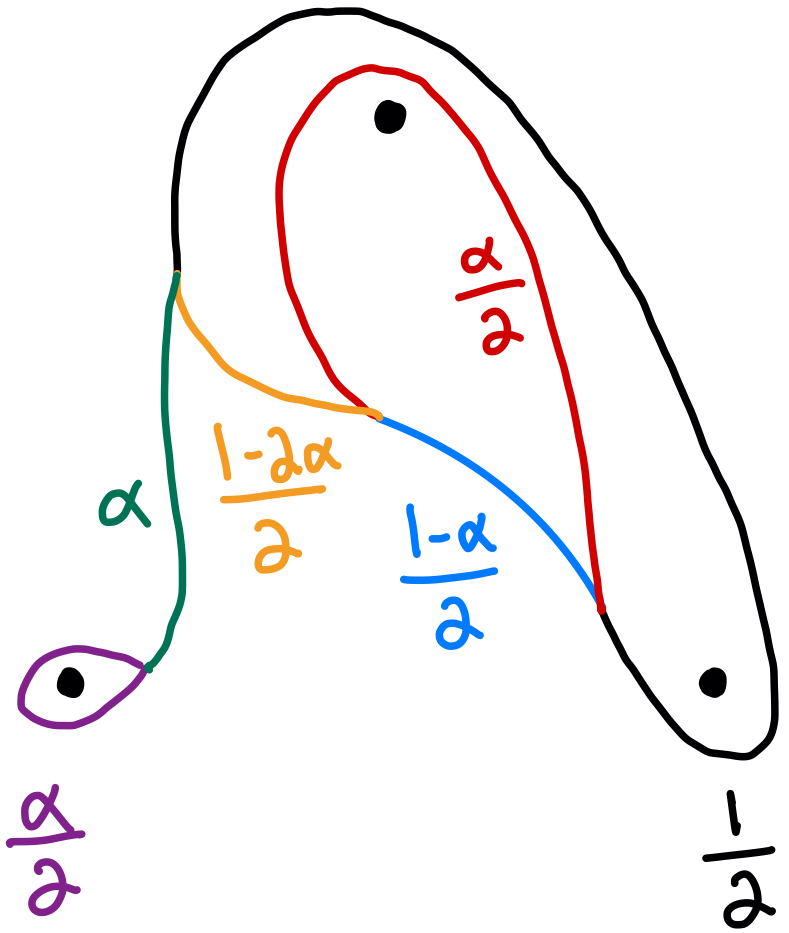}   
\raisebox{40pt}{$\xrightarrow{\text{Isotopy}}$} \\
&& \includegraphics[height=3.0cm]{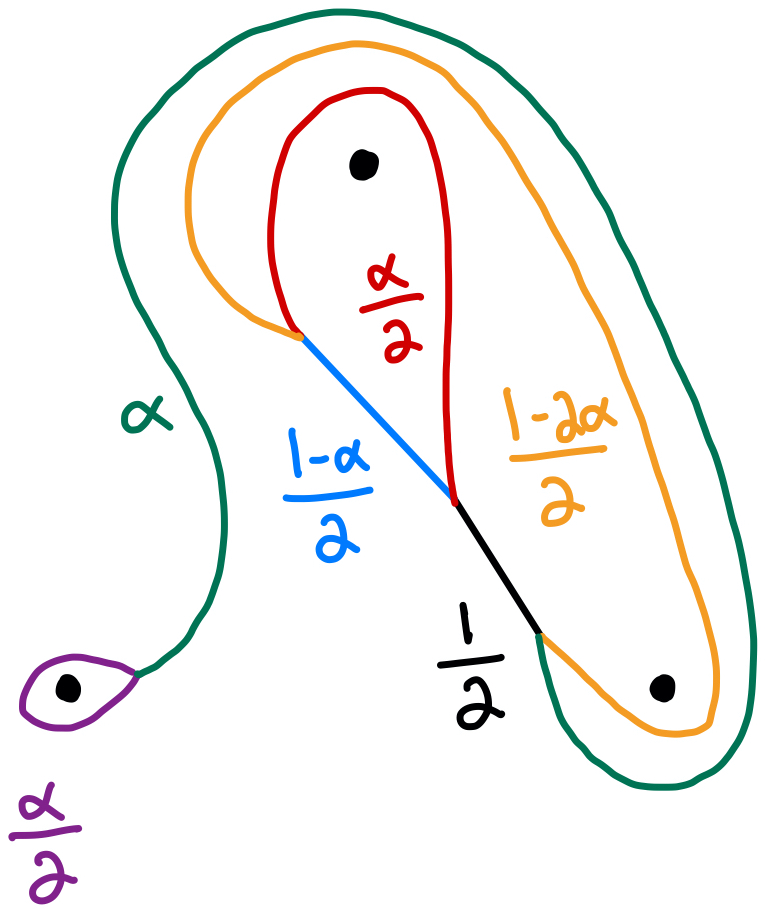}   
\raisebox{40pt}{$\xrightharpoonup{\text{Splitting } \frac{1}{2}}\ \tau_4=$} 
\includegraphics[height=3.0cm]{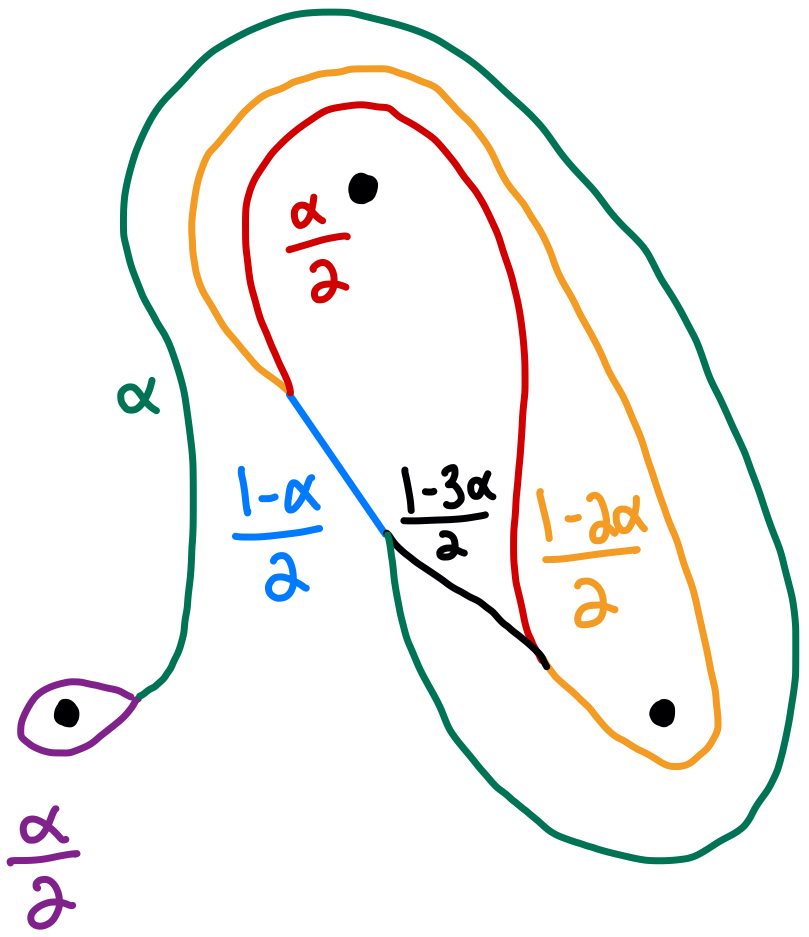}   
\raisebox{40pt}{$\xrightarrow{\text{Isotopy}}$} 
\includegraphics[height=3.0cm]{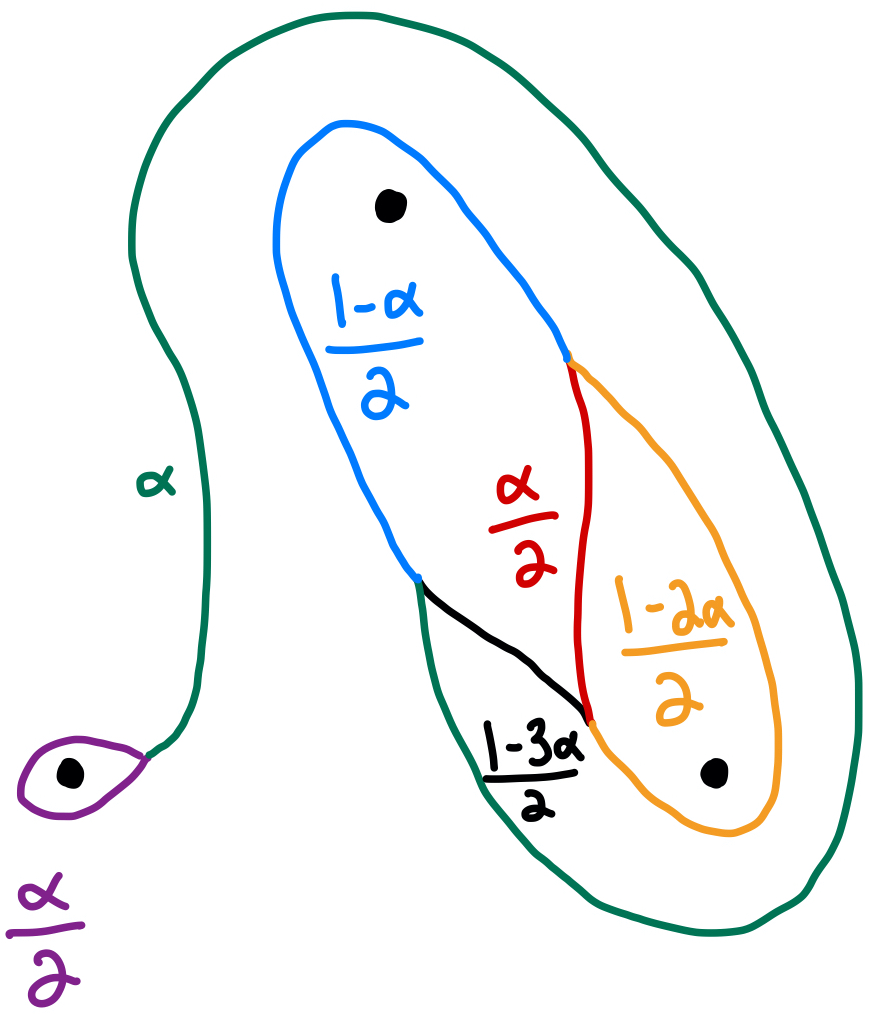} \raisebox{40pt}{$=D_3$ }
\end{eqnarray*}

Inductive Hypothesis: Suppose that if $\alpha<1/n$ and $n\geq 3$, then after $(n+1)$ maximal splittings, the train track $\tau_0$ becomes $D_n$.

Inductive Step: Assume that $\alpha<\frac{1}{n+1}$. 
Since $\frac{1}{n+1}<\frac{1}{n}$, $\alpha<1/n$ and we can apply the inductive hypothesis. 
Thus, after $(n+1)$ maximal splittings, $\tau_0$ becomes $D_n$. 
We perform another maximal splitting to the edge $\frac{1-(n-2)\alpha}{2}$ to obtain $D_{n+1}$. 
\vspace{-.03cm}
\begin{equation*}
\raisebox{40pt}{$D_n=$ } 
\includegraphics[height=3.5cm]{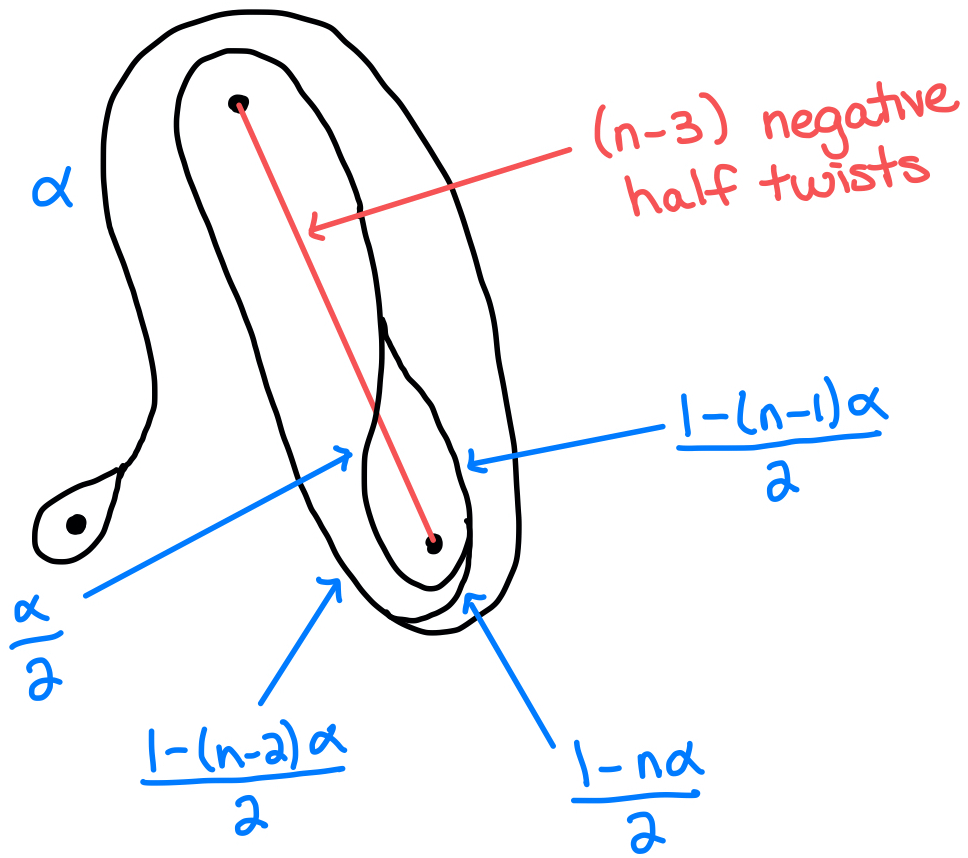}
\raisebox{40pt}{$\rightharpoonup$} 
\includegraphics[height=3.5cm]{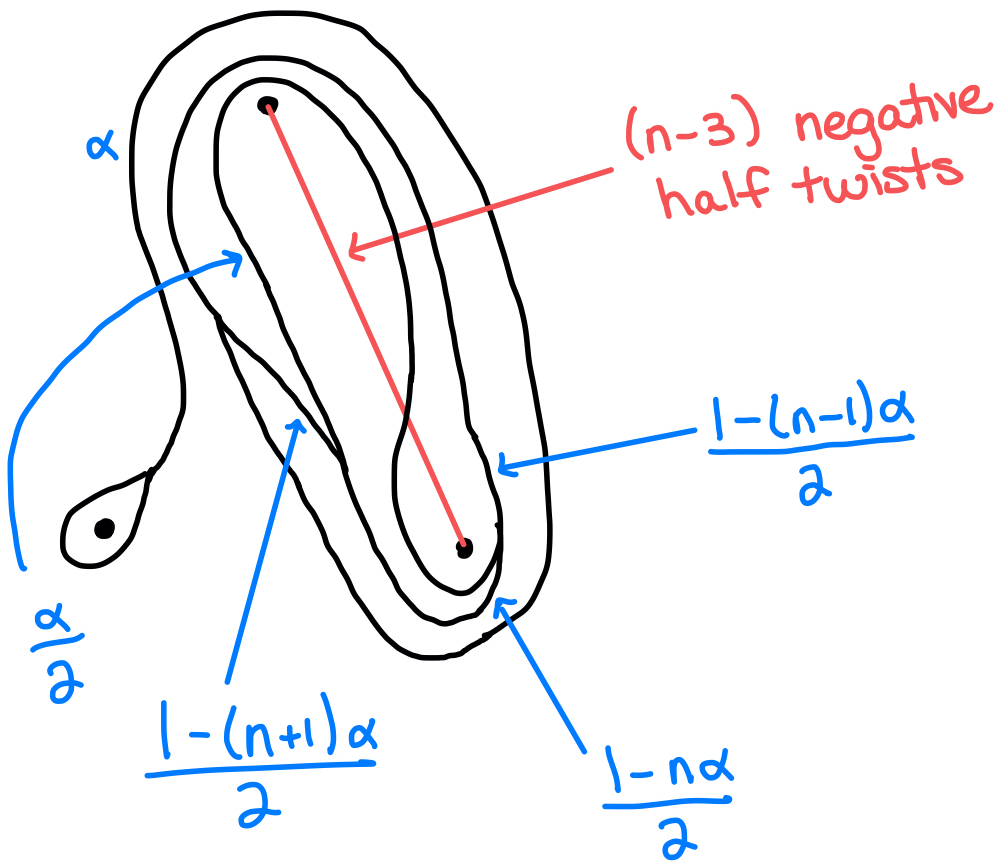}
\raisebox{40pt}{$\xrightarrow{\text{(+)twist}}$} 
\includegraphics[height=3.5cm]{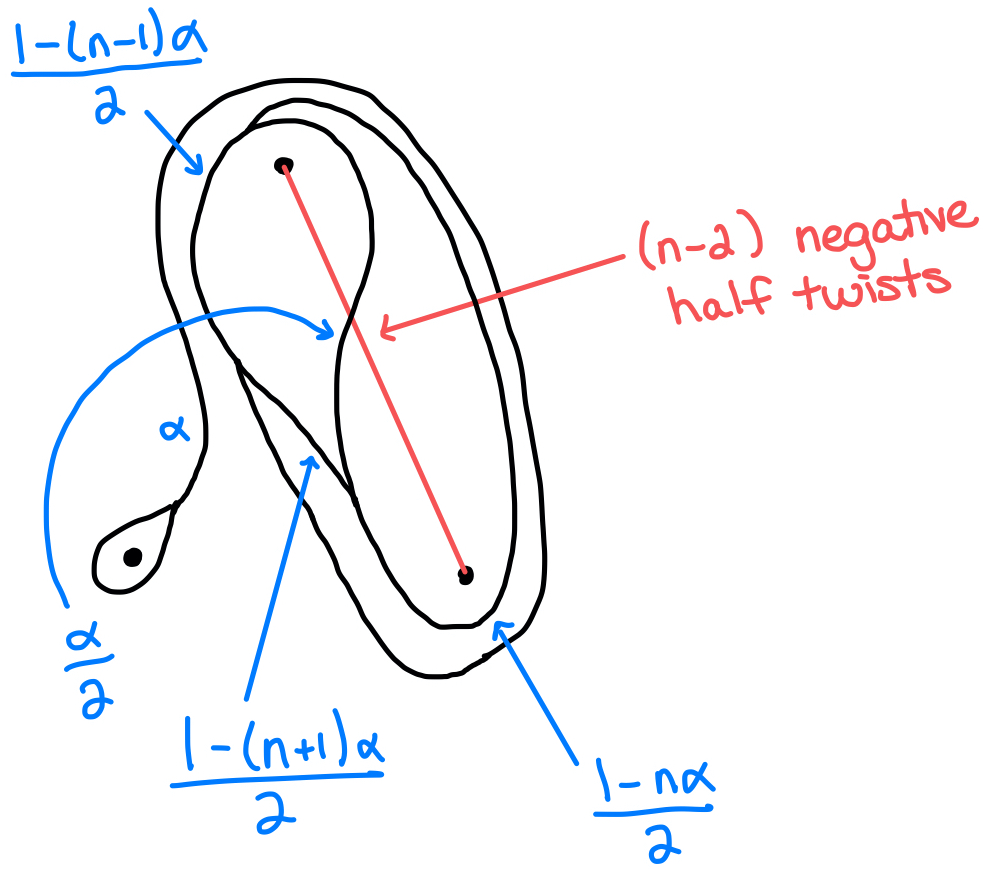}
\raisebox{40pt}{$=D_{n+1}$ }
\end{equation*}
\end{proof}

Now we are ready to prove the Triple-Weight Train Track Theorem. 

\begin{proof}[Proof of Theorem \ref{ThmTTT}]
We prove that the train track $\tau_{n+3}$ is the train track in Figure~\ref{figure-T_n}. 
Once this is done, it is easy to see that $\tau_i$ has triple-weight type for all $i\geq n+3$ since $\tau_{n+3}$ has six edges and the maximal splitting preserves the numbers of edges and vertices. 

We first provide the explicit details for $n=1$ and $n=2$ and then apply Lemma~\ref{LemTT} for $n\geq 3$.\\

\noindent ($n=1$): Note that $\frac{1}{2} < \alpha < 1$. 
\begin{equation*}
\raisebox{40pt}{$\tau_0=$ } 
\includegraphics[height=3.5cm]{LemA.jpeg}
\raisebox{40pt}{$\xrightharpoonup{\text{Splitting } \alpha+1}\ \tau_1=$} 
\includegraphics[height=3.5cm]{LemB.jpeg}
\raisebox{40pt}{$\xrightharpoonup{\text{Splitting } 1}\ \tau_2=$} 
\includegraphics[height=3.5cm]{LemC.jpeg}
\end{equation*}
\begin{equation*}
\raisebox{40pt}{$\xrightharpoonup{\text{Splitting } \frac{\alpha+1}{2}}\ \tau_3=$} 
\includegraphics[height=3.5cm]{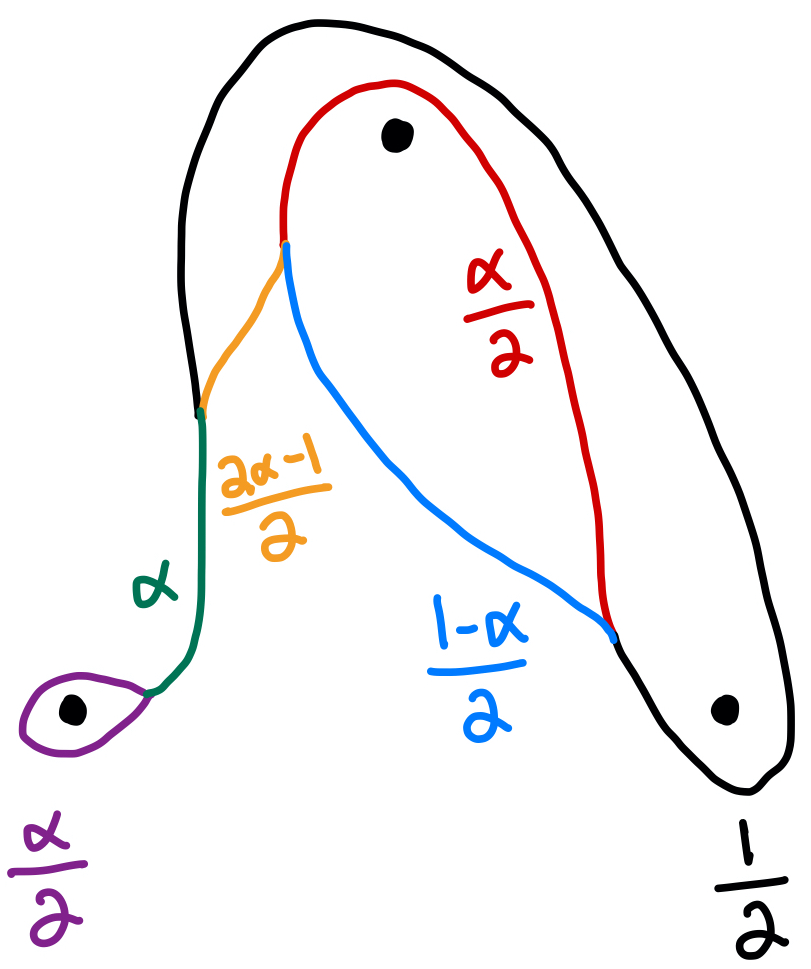}   
\raisebox{40pt}{$\xrightharpoonup{\text{Splitting } \alpha }\ \tau_4=$} 
\includegraphics[height=3.5cm]{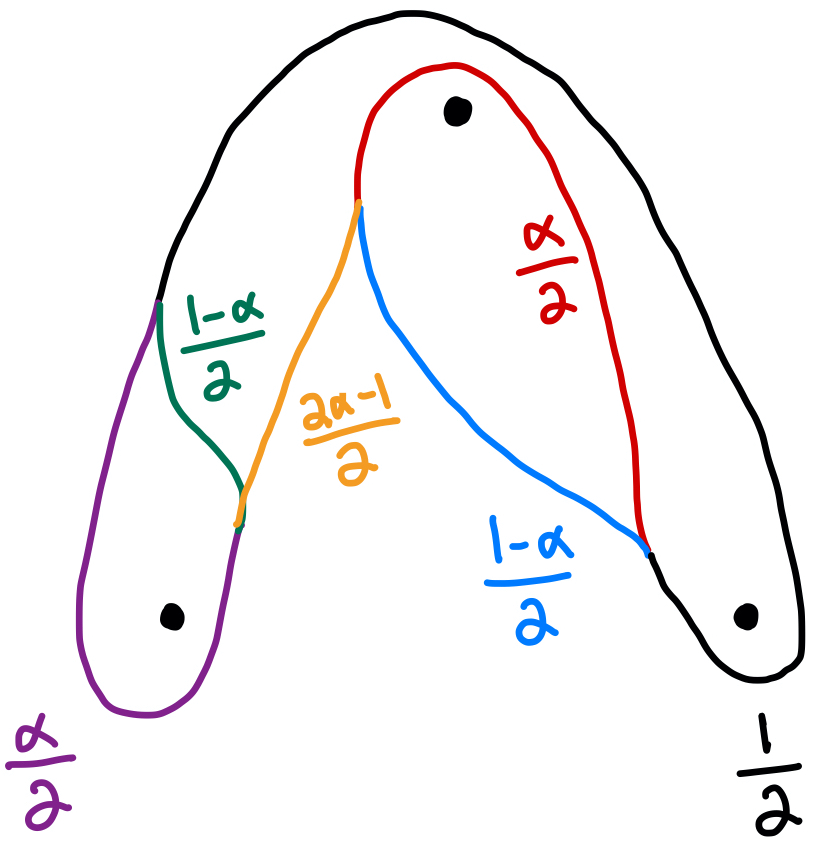} 
\end{equation*}
\begin{equation*}
\raisebox{40pt}{$\xrightharpoonup{\text{Splitting } \frac{1}{2} }\tau_5= $}
\includegraphics[height=3.5cm]{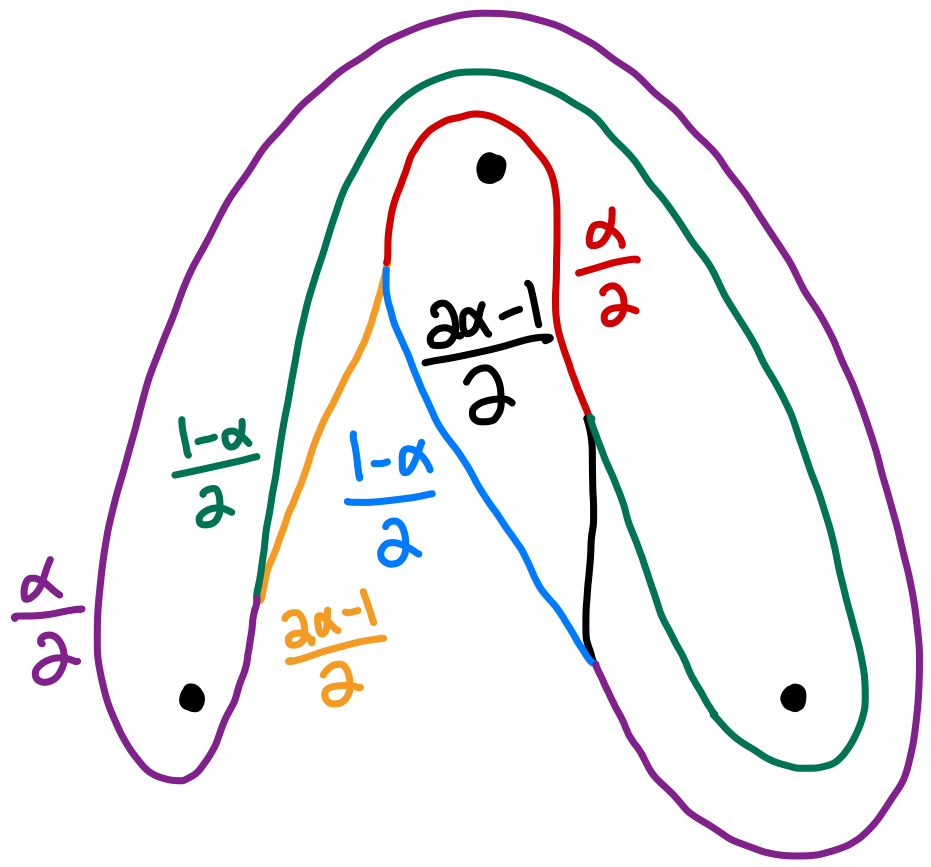}  
\raisebox{40pt}{$\xrightarrow{\text{Isotopy}}$}  
\includegraphics[height=3.5cm]{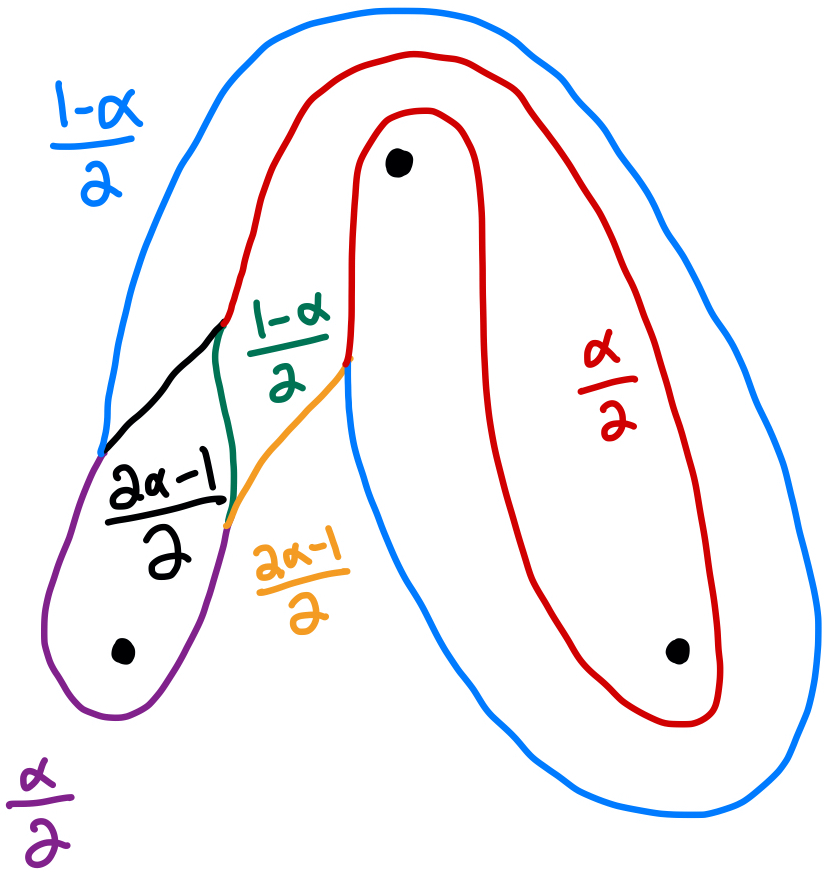}  \raisebox{40pt}{$=T_1$ } 
\end{equation*}

\noindent ($n=2$): Note that $\frac{1}{3} <\alpha < \frac{1}{2}$. Thus the initial part of the sequence $\tau_0 \rightharpoonup \tau_1 \rightharpoonup \tau_2$ is exactly the same as the $n=1$ case. The difference between $n=1$ and $n=2$ cases occur at $\tau_3$. 
\begin{equation*}
\raisebox{40pt}{$\tau_0=$ } 
\includegraphics[height=3.5cm]{LemA.jpeg}
\raisebox{40pt}{$\xrightharpoonup{\text{Splitting } \alpha+1} \tau_1 \xrightharpoonup{\text{Splitting } 1} \tau_2$ } 
\raisebox{40pt}{$\xrightharpoonup{\text{Splitting } \frac{\alpha+1}{2}}\ \tau_3=$} 
\includegraphics[height=3.5cm]{LemE.jpeg}   
\end{equation*}
\begin{equation*}
\raisebox{40pt}{$\xrightharpoonup{\text{Splitting } \frac{1}{2}} \ \tau_4=$} 
\includegraphics[height=3.5cm]{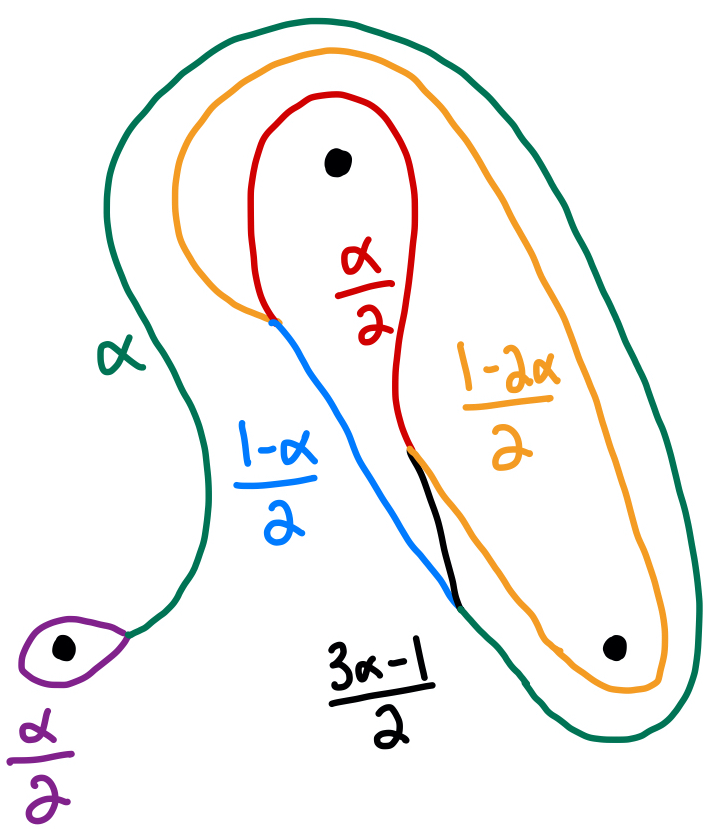}   
\raisebox{40pt}{$\xrightharpoonup{\text{Splitting } \alpha }\ \tau_5=\ $}
\includegraphics[height=3.5cm]{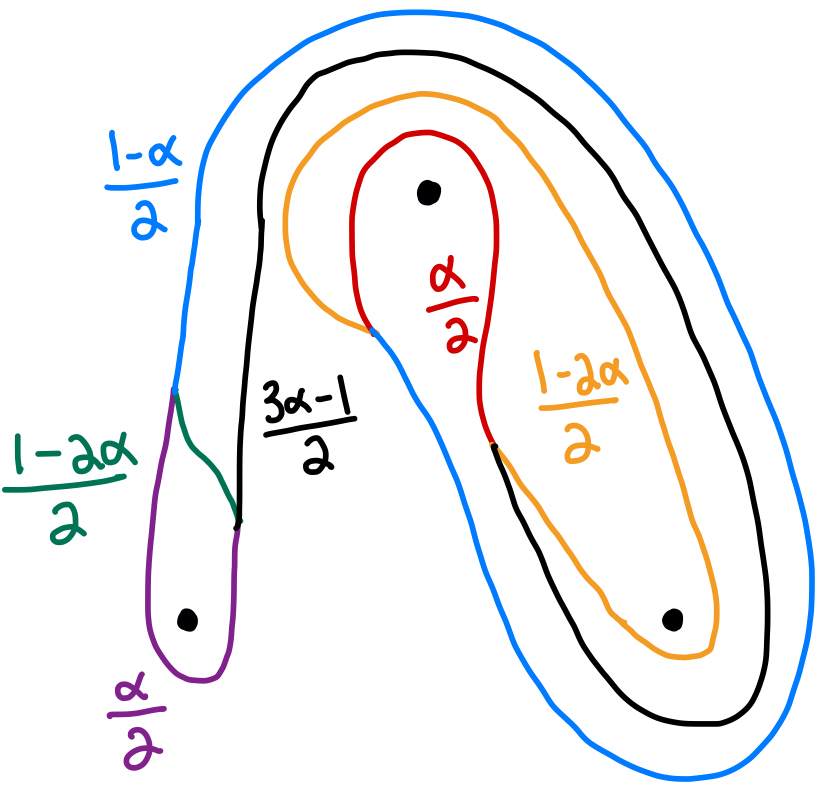}
\end{equation*}
\begin{equation*}
\raisebox{40pt}{$\xrightharpoonup{\text{Splitting } \frac{1-\alpha}{2} }\ \tau_6=$} 
\includegraphics[height=3.5cm]{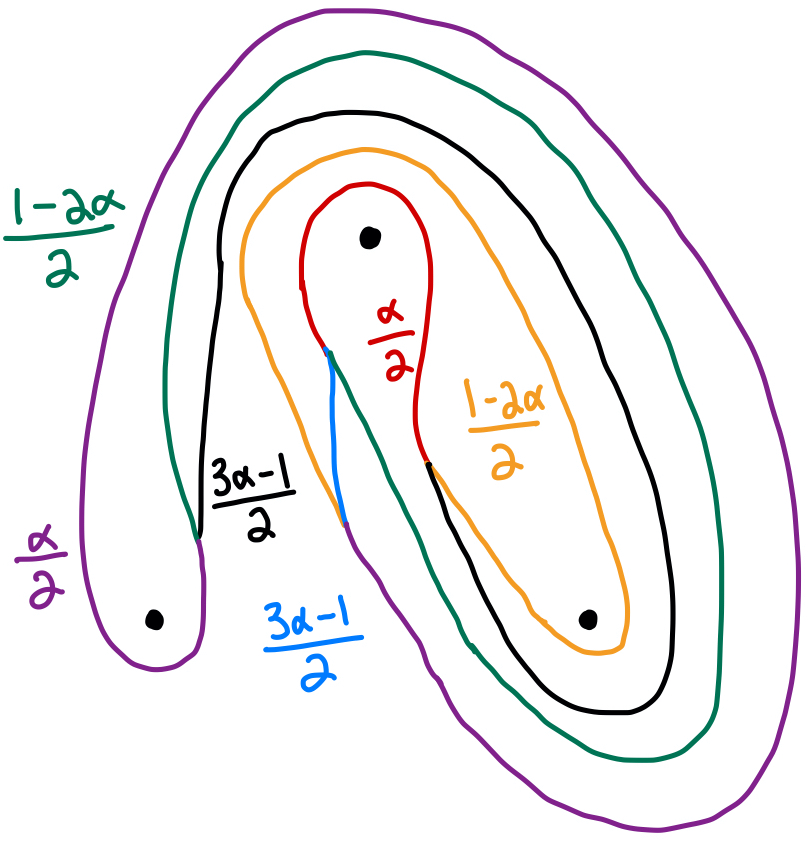}
\raisebox{40pt}{$\xrightarrow{\text{Isotopy}}$} 
\includegraphics[height=3.5cm]{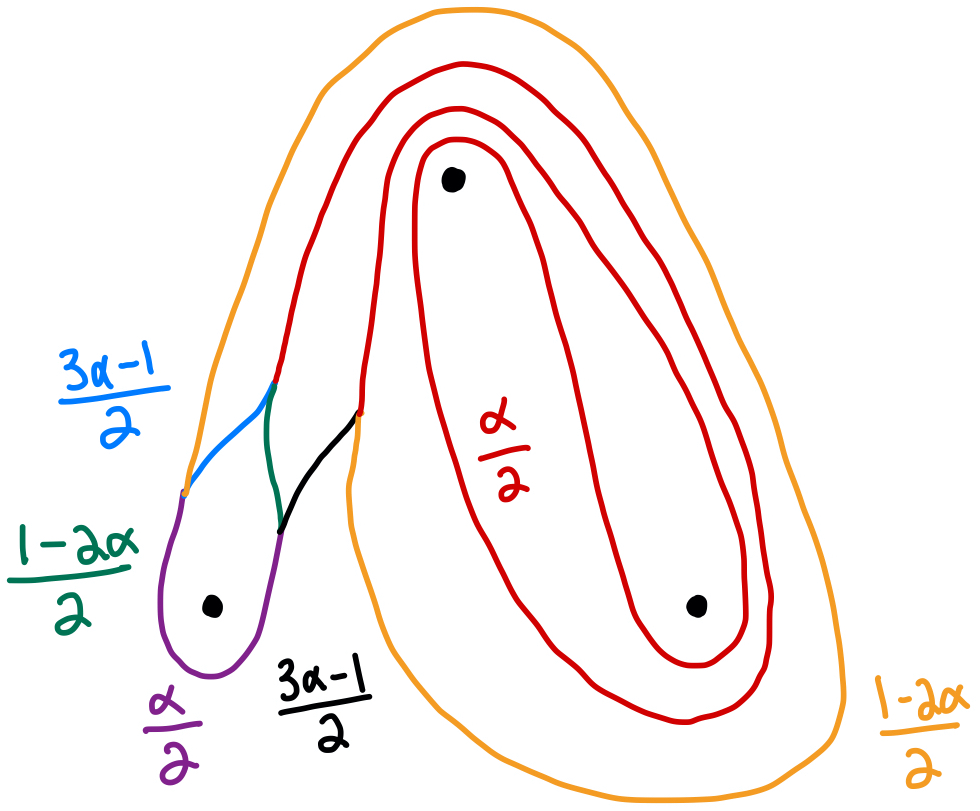}\raisebox{40pt}{$=T_2$ } 
\end{equation*}

Assume $n \geq 3$ and $\frac{1}{n+1}<\alpha<\frac{1}{n}$. 
By Lemma \ref{LemTT}, after $(n+1)$ maximal splittings, we arrive at train track $D_n$. 
We now apply three more maximal splittings and obtain $T_n$.

\begin{equation*}
\raisebox{40pt}{$\tau_{n+1}=D_n=$ } 
\includegraphics[height=3.5cm]{Dn.jpeg}
\raisebox{40pt}{$\xrightharpoonup{\text{Splitting } \frac{1-(n-2)\alpha}{2} }\ \tau_{n+2}=$} 
\includegraphics[height=3.5cm]{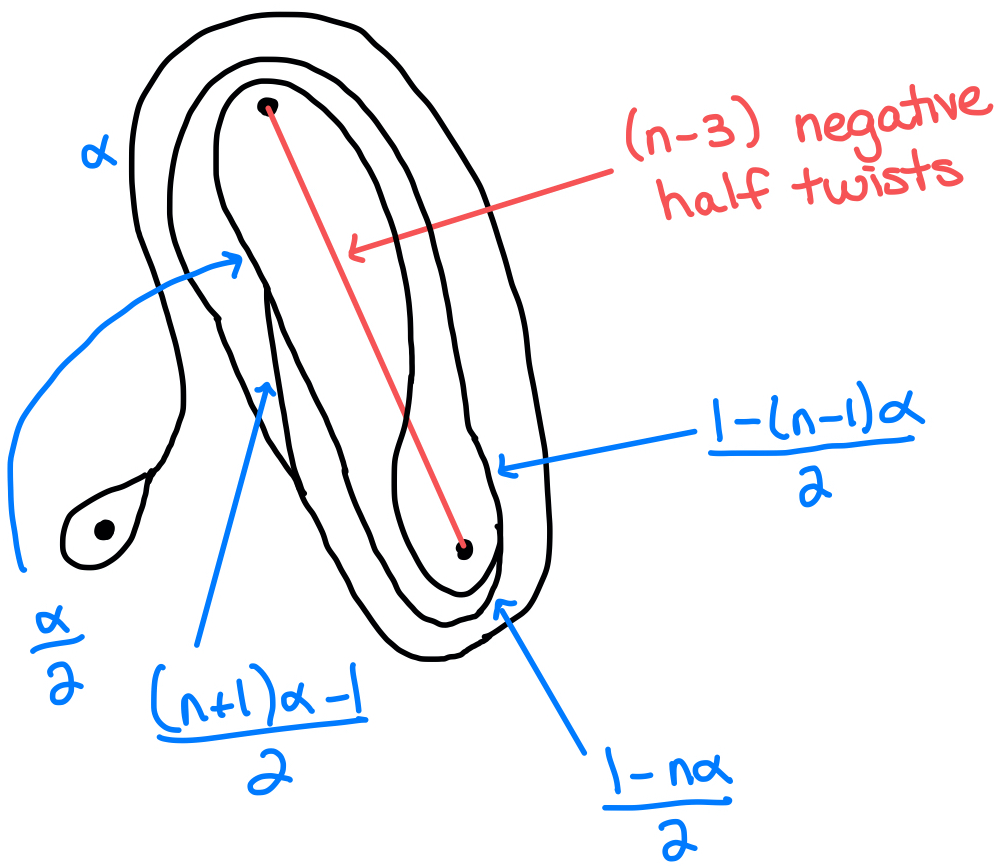} 
\end{equation*}
\begin{equation*}
\raisebox{40pt}{$\xrightarrow{\text{Absorb Half Twist} }$} 
\includegraphics[height=3.5cm]{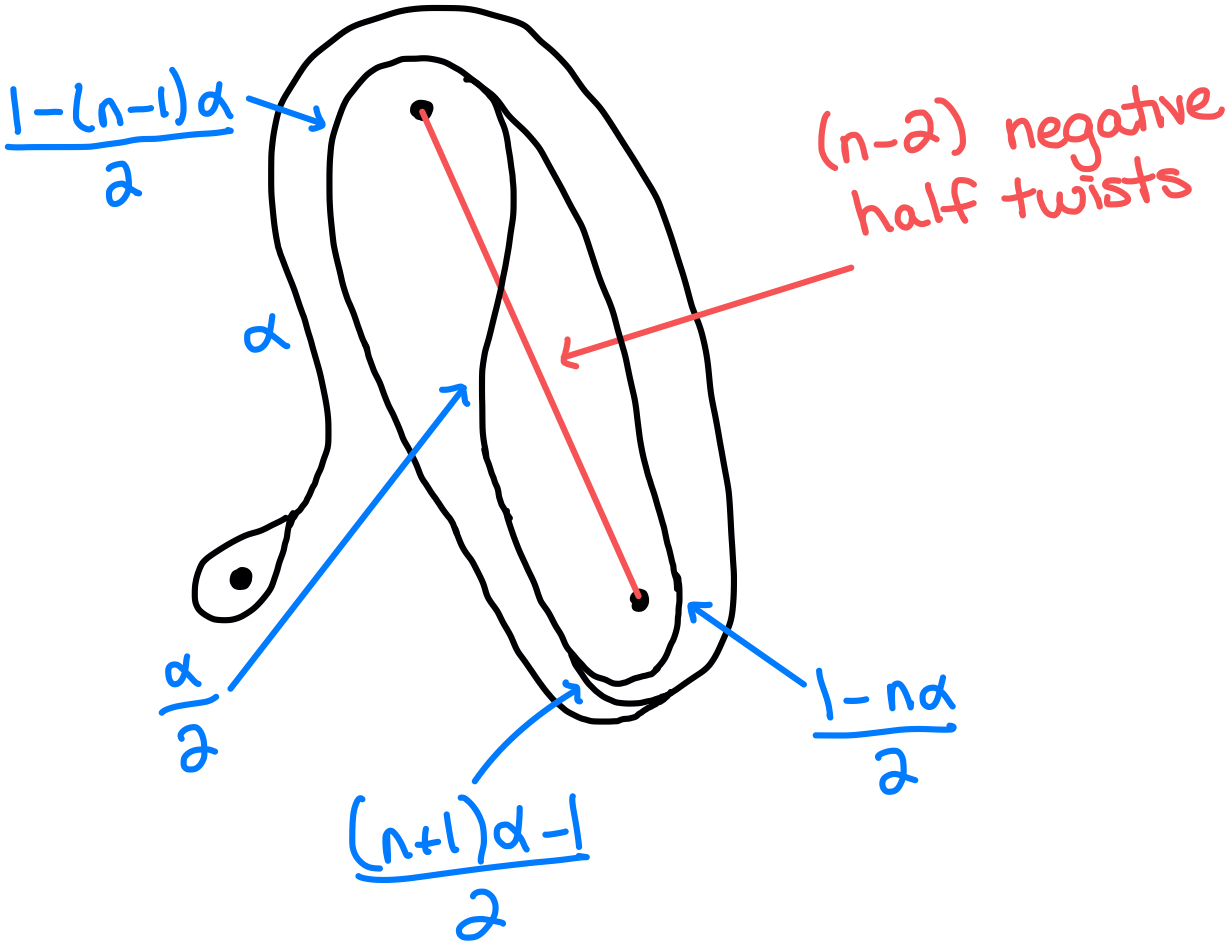} 
\raisebox{40pt}{$\xrightharpoonup{\text{Splitting } \alpha }\ \tau_{n+3}=$} \includegraphics[height=3.5cm]{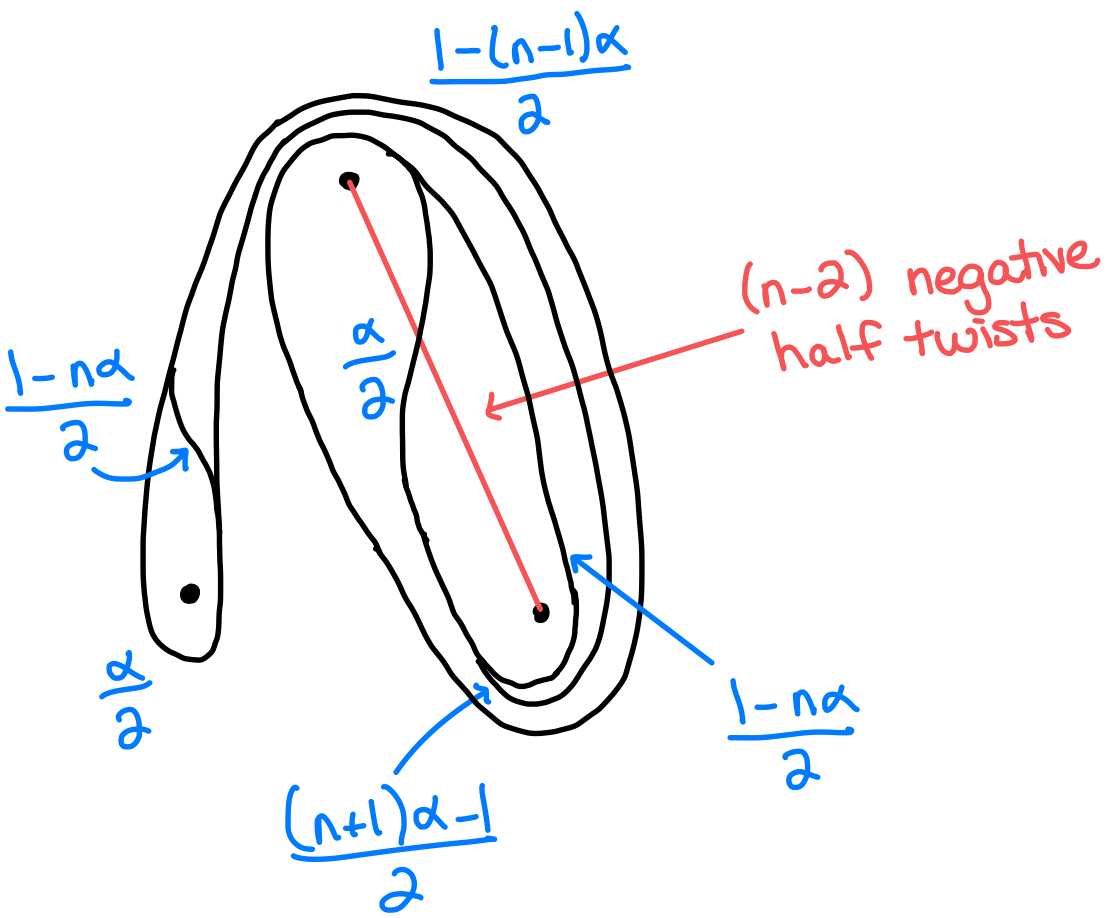}
\raisebox{40pt}{$\xrightarrow{\text{Isotopy}}$}
\end{equation*}
\begin{equation*}
\includegraphics[height=3.5cm]{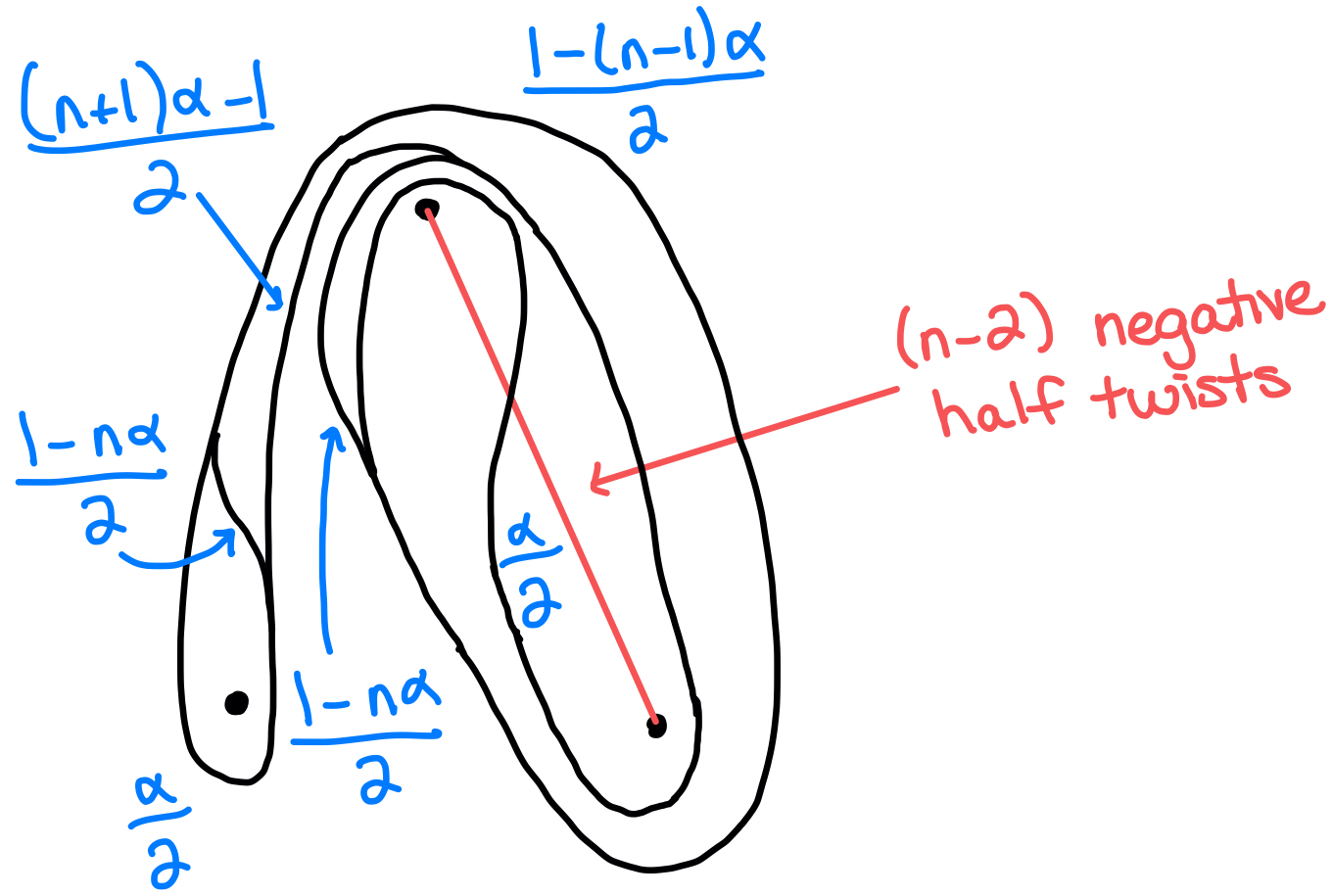} 
\raisebox{40pt}{$\xrightarrow{\text{Absorb Half Twist} }$}
\includegraphics[height=3.5cm]{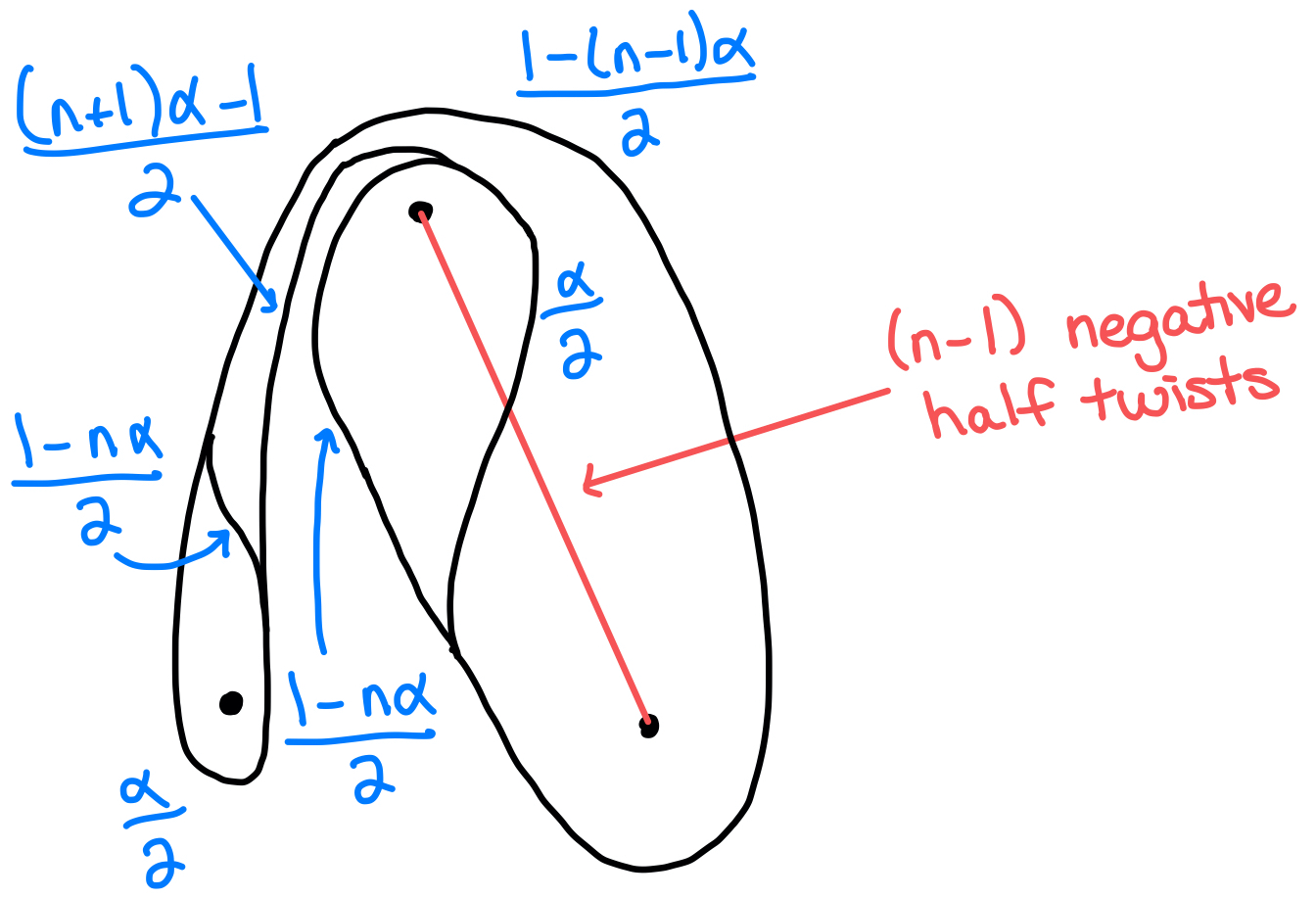}
\raisebox{40pt}{$\xrightharpoonup{\text{Splitting } \frac{1-(n-1)\alpha}{2}}$} 
\end{equation*}
\begin{equation*}
\raisebox{40pt}{$\tau_{n+4}=$}
\includegraphics[height=3.5cm]{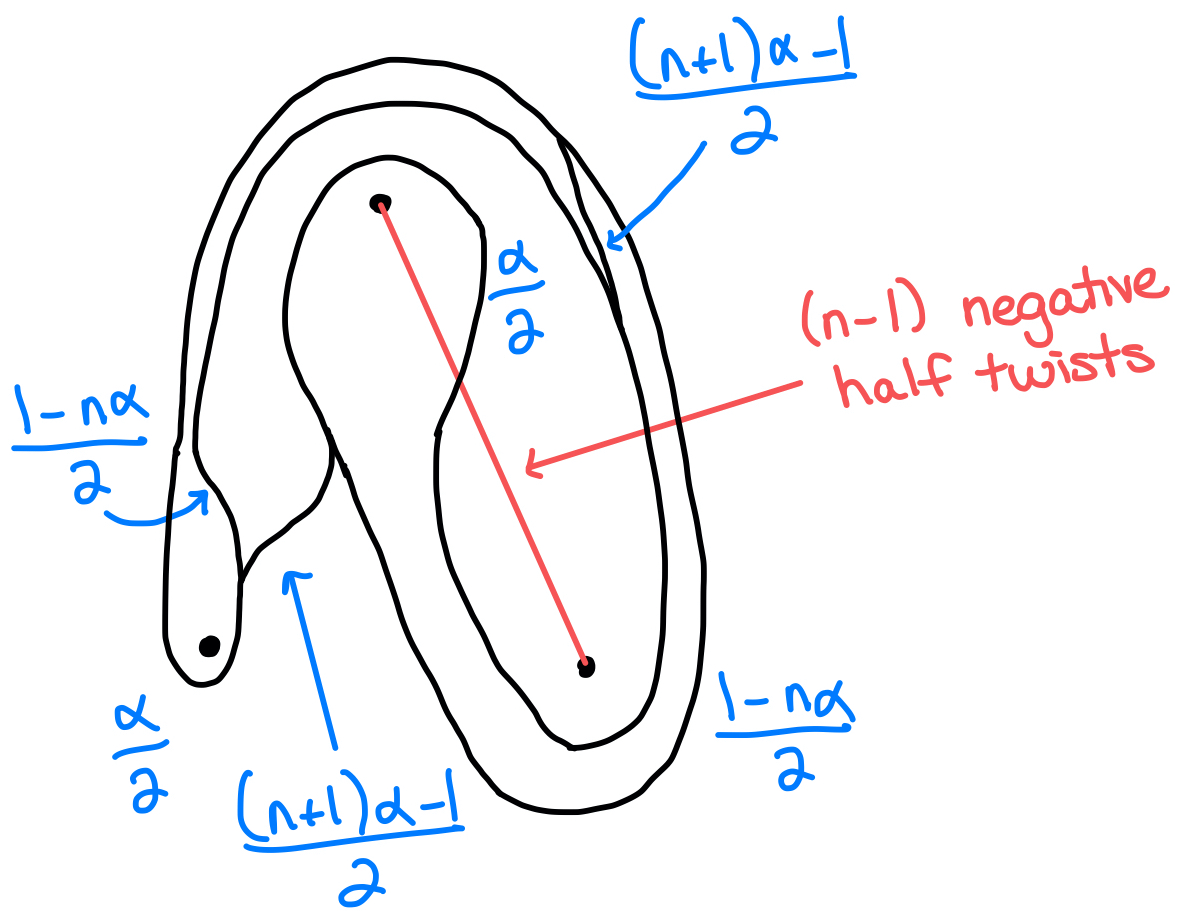}
\raisebox{40pt}{$\xrightarrow{\text{Absorb Half Twist} }$} \includegraphics[height=3.5cm]{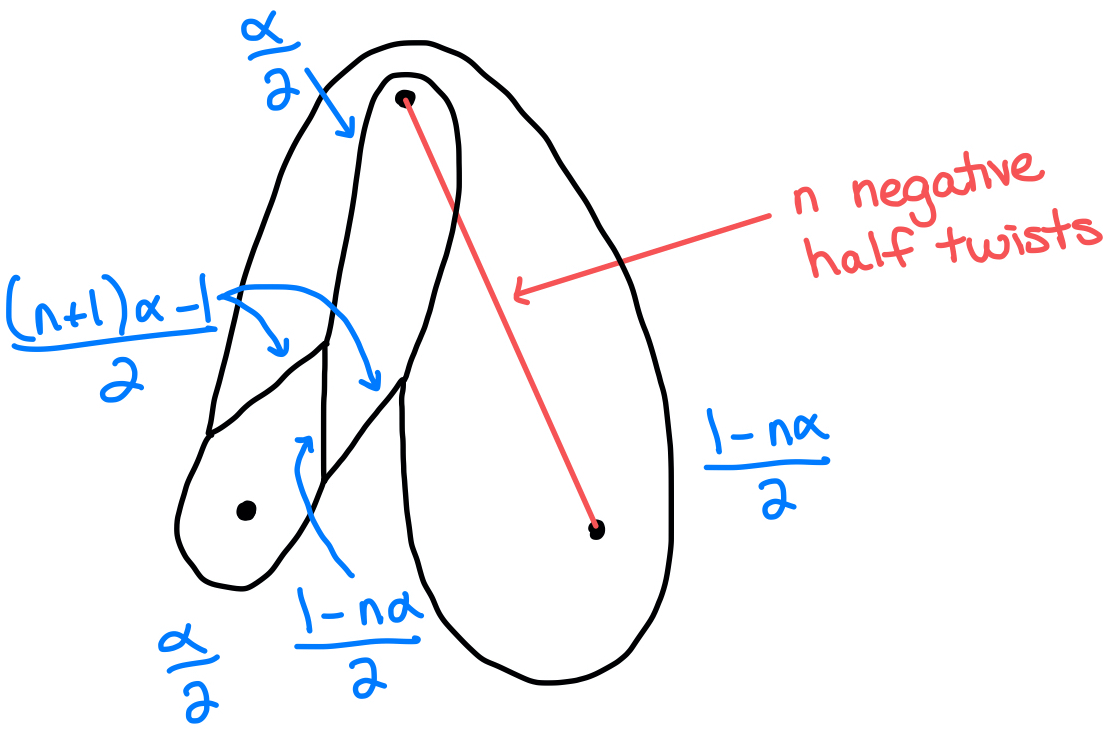}\raisebox{40pt}{$=T_n$ } 
\end{equation*}
\end{proof}

\section{System of triple-weight train tracks} 

The goal of this section is to study a closed system of triple-weight train tracks and prove Theorem~\ref{thm:closed-system}. 

\begin{theorem}\label{thm:closed-system}
Up to homeomorphism, in the maximal splitting sequence there are only four 
types of triple-weight train tracks for pseudo-Anosov 3-braids as depicted in Figure~\ref{figure-121'2'}, which we name Types I, I', II, II'. 

Moreover, they form a closed system under the maximal splitting operation (green arrows in the figure).  
\end{theorem}


\begin{figure}[ht]
\includegraphics[height=10cm]{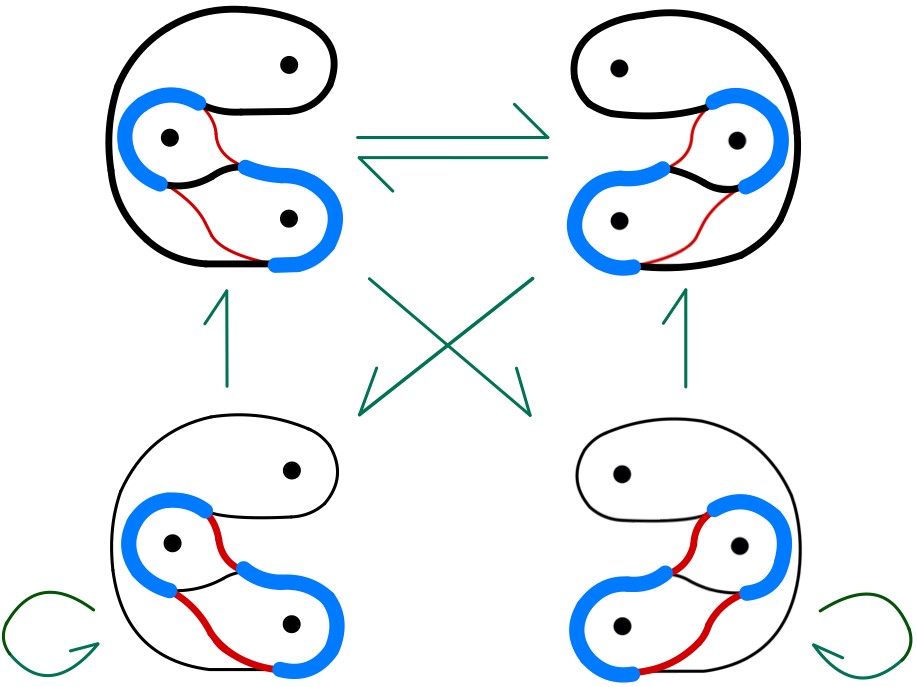}
\put(-40, 220){\fontsize{11}{15}Type I'}
\put(-40, 205){($\mp$)}
\put(-40, 80){\fontsize{11}{15}Type II'}
\put(-40, 65){($\pm$)}
\put(-380, 80){\fontsize{11}{15}Type II}
\put(-380, 65){($\mp$)}
\put(-380, 220){\fontsize{11}{15}Type I}
\put(-380, 205){($\pm$)}
\put(-282, 140){($\ast$)}
\put(-205, 240){($\ast$)} 
\put(-205, 205){($\ast$)} 
\put(-92, 140){($\ast$)} 
\put(-392, -5){($\ast\ast$)} 
\put(-10, -5){($\ast\ast$)} 
\put(-240, 150){($\ast\ast$)}
\put(-170, 150){($\ast\ast$)}
\caption{
(Theorem~\ref{thm:closed-system}) 
Type I, I', II, and II' triple-weight train tracks. 
Type $N$' is the mirror image of Type $N$.
The thickness of edges reflect the weights. 
Edges with the same color have the same weight. 
The blue edges have the highest weight for all the types.  
For Types I and I', the weight increases in the order of red $<$ black $<$ blue. 
For Types II and II', black $<$ red $<$ blue. 
The symbol ($\pm$) represents the sign (Definition~\ref{def of 4-tuple and ratio}) of the train track and is explained in Corollary~\ref{cor:Type and interval}. 
The arrow $\stackrel{(\ast)}{\rightharpoonup}$ (resp.~$\stackrel{(\ast\ast)}{\rightharpoonup}$) means the maximal splitting is Type~1 (resp.~Type~2) as defined in Proposition~\ref{prop:algorithm of 4-tuple}.}
\label{figure-121'2'}
\end{figure}
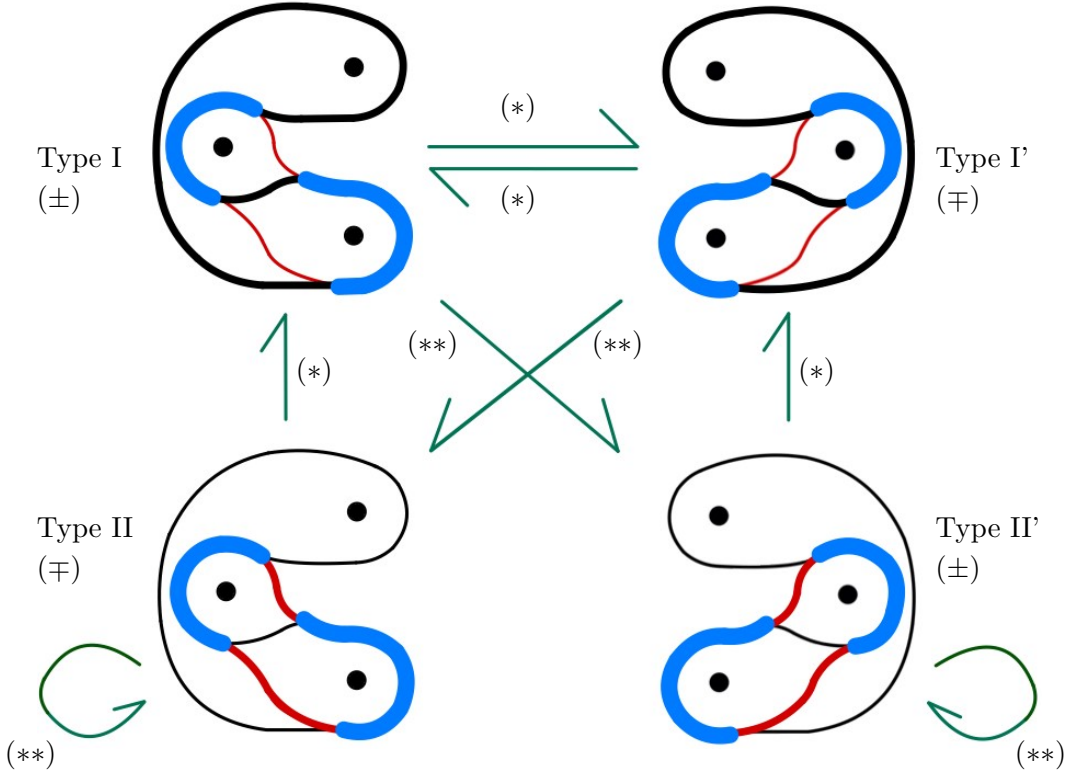

As a consequence of Theorem~\ref{ThmTTT},
the first triple-weight train track $\tau_{n+4}$ (Figure \ref{figure-T_n}) has Type I or II (resp, Type I' or II') if $\tau_0=\M_\alpha$ or $\W_\alpha$ (resp. $\tau_0=\M'_\alpha$ or $\W'_\alpha$). 
This is because Types I, $\dots,$ II' are considered up to homeomorphism, and $\M_\alpha$ is a $180^\circ$-rotation of $\W_\alpha$.

\begin{definition}\label{def:I-II-I'-II'-sequence}
It is convenient to define $\mathcal{T}_1:=\tau_{n+4}$ as it is the first triple-weight train track in our maximal splitting sequence and further define $\mathcal{T}_k:=\tau_{n+4+k-1}$ for $k \geq 1$. 
Thus, given a braid $\beta$, we obtain a maximal splitting sequence of triple-weight train tracks: 
$$
\T_1 \msp \T_2 \msp \T_3 \msp \cdots
$$
Looking at the types of the train tracks in $\mathcal T_1 \rightharpoonup \mathcal T_2 \rightharpoonup \mathcal T_3 \rightharpoonup \cdots$, we get a sequence in $\{$I, II, I', II'$\}$. We call it the {\em I-II-I'-II'-sequence} associated to the sequence $\mathcal T_1 \rightharpoonup \mathcal T_2 \rightharpoonup \mathcal T_3 \rightharpoonup \cdots$. 
\end{definition}

Examples of I-II-I'-II'-sequences are given in Example~\ref{examples beta and beta'}.

Before proving Theorem \ref{thm:closed-system}, we introduce the 4-tuple of a triple-weight train track that plays an important role in this paper. 
Then we proceed to prove Theorem \ref{thm:closed-system}. In the proof, we find a nice relation between the types of 4-tuples and the types of triple-weight train tracks, which is stated as Corollary~\ref{cor:rel_of_types}.

We have shown in Theorem~\ref{ThmTTT} that the three weights of the train track $\mathcal T_1$ are 
\begin{equation}\label{eq:3weights}
\frac{1-n\alpha}{2},\  \frac{-1+(n+1)\alpha}{2}, \ \frac{\alpha}{2}. 
\end{equation}
Therefore, by the nature of maximal splitting, every weight of a triple-weight train track $\mathcal T_k$ is in $\frac{1}{2}\Z + \frac{1}{2}\alpha \Z$.

\begin{definition}\label{def of 4-tuple and ratio}
Suppose that the smallest weight of a triple-weight train track $\T$ is given by $\frac{1}{2}x+\frac{1}{2}y\alpha$ and the second smallest weight is $\frac{1}{2}z+\frac{1}{2}w\alpha$. Thus $x+y \alpha < z + w\alpha.$
We call 
\begin{itemize}
\item
$(x, y\ ;\ z, w)$ the {\em 4-tuple} of the triple train track $\T$, 
\item
$\sgn(x)\in\{-1, 1\}$ the {\em sign} of the triple-weight train track $\T$ (we will show that $x\neq 0$), and we denote $\sgn(\T)=\sgn(x)$, 
\item
$\frac{x+y\alpha}{z+w\alpha}$ the {\em 4-tuple ratio} (or {\em 4-ratio} for short) of the triple-weight train track $\T$. 
\end{itemize}
\end{definition}

Due to the switch condition of train tracks, the largest weight of the triple-weight train track is the sum of the other two, that is $\frac{1}{2}(x+z) + \frac{1}{2}(y+w)\alpha$. 
Thus the 4-tuple $(x, y\ ;\ z, w)$ carries all the weight data of the triple-weight train track. 

By the definition of the 4-tuple, the following is immediate.

\begin{proposition}\label{prop:new}
If a triple-weight train track with 4-tuple $(x, y\ ;\ z, w)$ is
\begin{itemize}
\item
Type I or I', then the red-colored (thinnest) edges (in Figure~\ref{figure-121'2'}) have weight $\frac{1}{2}(x+y\alpha)$ and the black-colored (second thickest) edges have $\frac{1}{2}(z+w\alpha)$,
\item
Type II or II', then the black-colored (thinnest) edges have weight $\frac{1}{2}(x+y\alpha)$ and the red-colored (second thickest) edges have $\frac{1}{2}(z+w\alpha)$. 
\end{itemize}
\end{proposition}

We introduce a 4-tuple sequence $\{(x_k, y_k\ ; \ z_k, w_k) \mid k=0, 1, \cdots \}$ that plays an important role to prove the main results.  

\begin{definition}\label{def:four-tuple}
Set 
\begin{equation}\label{eq:initial 4-tuple}
(x_0, y_0\ ;\ z_0, w_0):=(1, -n\ ;\  0,1).
\end{equation} 
For $k=1, 2, \dots$, let $(x_k, y_k\ ; \ z_k, w_k)$ denote the 4-tuple of $\mathcal{T}_k$. 
\end{definition}

Among the three weights of $\mathcal T_1$ in (\ref{eq:3weights}), 
$\frac{\alpha}{2}$ is the largest.  
Depending on the value of $\alpha$ we have either
$$\frac{1}{2}(-1+(n+1)\alpha) < \frac{1}{2}(1-n\alpha) \  
\mbox{ or } \ 
\frac{1}{2}(1-n\alpha)< \frac{1}{2}(-1+(n+1)\alpha).$$
In the former (resp. latter) case, 
$$(x_1, y_1\ ;\ z_1, w_1)=(-1, n+1 \ ; \ 1, -n) \ \mbox{ 
(resp. } (1, -n \ ;\  -1, n+1))
$$ 
and we say that the 4-tuple $(x_1, y_1\ ;\ z_1, w_1)$ is Type 1 (resp.~Type 2).  
Combine this with the fact $\frac{1}{n+1} < \alpha < \frac{1}{n}$, and we obtain that the 4-tuple of $\mathcal T_1$ is 
$$
(x_1, y_1\ ;\ z_1, w_1)=
\left\{
\begin{array}{llc}
(-1, n+1 \ ; \ 1, -n) & \mbox{ if } \frac{1}{n+1}<\alpha<\frac{1}{n+\frac{1}{2}} & (\text{Type 1})\\
(1, -n \ ;\  -1, n+1) & \mbox{ if } \frac{1}{n+\frac{1}{2}}<\alpha<\frac{1}{n} & (\text{Type 2}).
\end{array}
\right.
$$

\begin{prop}\label{prop:algorithm of 4-tuple}
In general, the 4-tuple $(x_{k+1}, y_{k+1} \ ; \ z_{k+1}, w_{k+1})$ can be computed  from $(x_k, y_k\ ;\ z_k, w_k)$ as follows. 
\begin{eqnarray*}
&&(x_{k+1}, y_{k+1} \ ; \ z_{k+1}, w_{k+1})\\ 
&&:=
\begin{cases} (z_k-x_k, w_k-y_k\ ; \ x_k, y_k) & \text{ if } (2y_k-w_k) \alpha > z_k -2 x_k \ (\text{Type 1}) \\ 
(x_k,y_k\ ; \ z_k-x_k, w_k-y_k) & \text{ if }  (2y_k-w_k) \alpha < z_k -2 x_k  \ (\text{Type 2})  
\end{cases}
\end{eqnarray*} 
For the first (resp.~second) case, we say that the 4-tuple $(x_{k+1}, y_{k+1} \ ; \ z_{k+1}, w_{k+1})$ has {\em Type~1} (resp. {\em Type 2}) and the maximal splitting $\mathcal T_k \rightharpoonup \mathcal T_{k+1}$ has Type 1 (resp.~Type 2), which we denote $\mathcal T_k \stackrel{(\ast)}{\rightharpoonup} \mathcal T_{k+1}$ (resp.~$\mathcal T_k \stackrel{(\ast\ast)}{\rightharpoonup} \mathcal T_{k+1}$). 
\end{prop}

\begin{proof}
Start with a triple-weight train track with $(x, y\ ;\ z, w)$.
See Figure~\ref{fig:max_split}.
\begin{figure}[h]
\includegraphics[height=2.5cm]{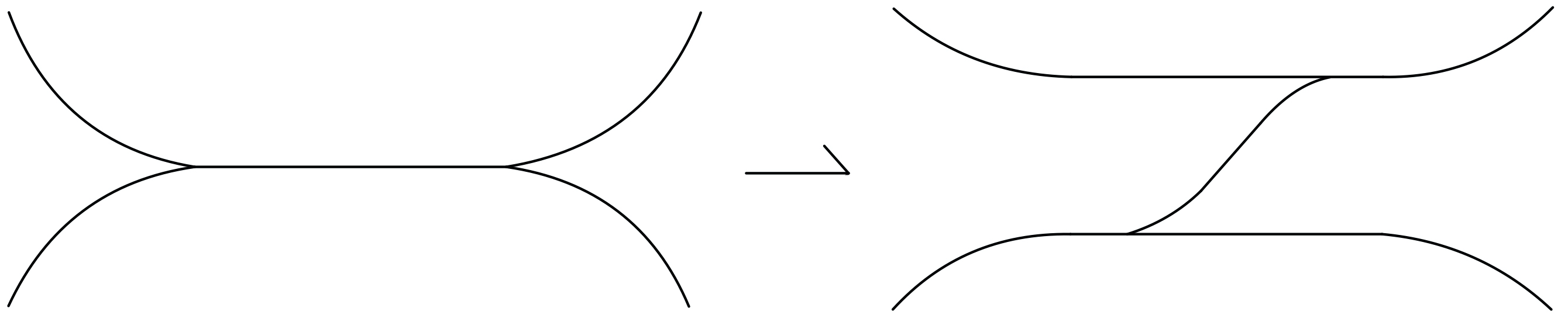}
\put(-20, 75){$z+w\alpha$}
\put(-20, -10){$x+y\alpha$}
\put(-160, -10){$z+w\alpha$}
\put(-160, 75){$x+y\alpha$}
\put(-65, 33){$(z-x)+(w-y)\alpha$}
\put(-370, 75){$x+y\alpha$}
\put(-235, 75){$z+w\alpha$}
\put(-370, -10){$z+w\alpha$}
\put(-235, -10){$x+y\alpha$}
\caption{Maximal splitting of triple-weight train track}
\label{fig:max_split}
\end{figure}
After a maximal splitting, the 4-tuple becomes $(x, y\ ; \ z-x, w-y)$ if 
$x+y\alpha < (z-x)+(w-y)\alpha$, equivalently 
$(2y-w)\alpha<z-2x$. 
We call it Type 1. 
Otherwise, we will have $(z-x, w-y\ ; \ x, y)$. 
We call it Type 2. 
\end{proof}


Now we prove Theorem~\ref{thm:closed-system}.

\begin{proof}[Proof of Theorem~\ref{thm:closed-system}]

Suppose that we have a Type I train track $\mathcal T$ with 4-tuple $(x, y\ ;\ z, w)$. 
The computation in Figure~\ref{fig:ABpart2} shows that the resulting train track $\mathcal T^{spl}$ of the maximal splitting has either Type I' or II'. 
\begin{figure}[h]
\hspace{-.5cm}
\includegraphics[height=3cm]{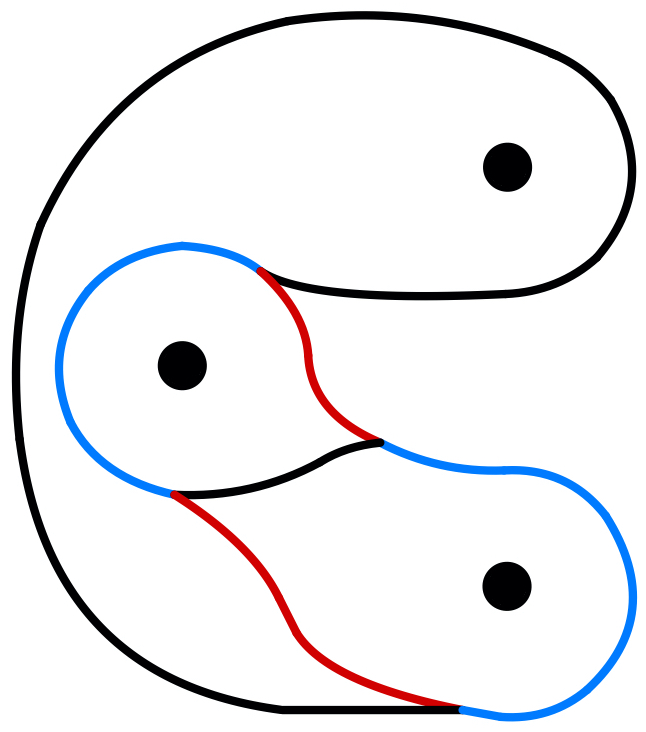}  
\raisebox{40pt}{$\xrightharpoonup[\text{Split}]{\text{Max}}$}  
\includegraphics[height=3cm]{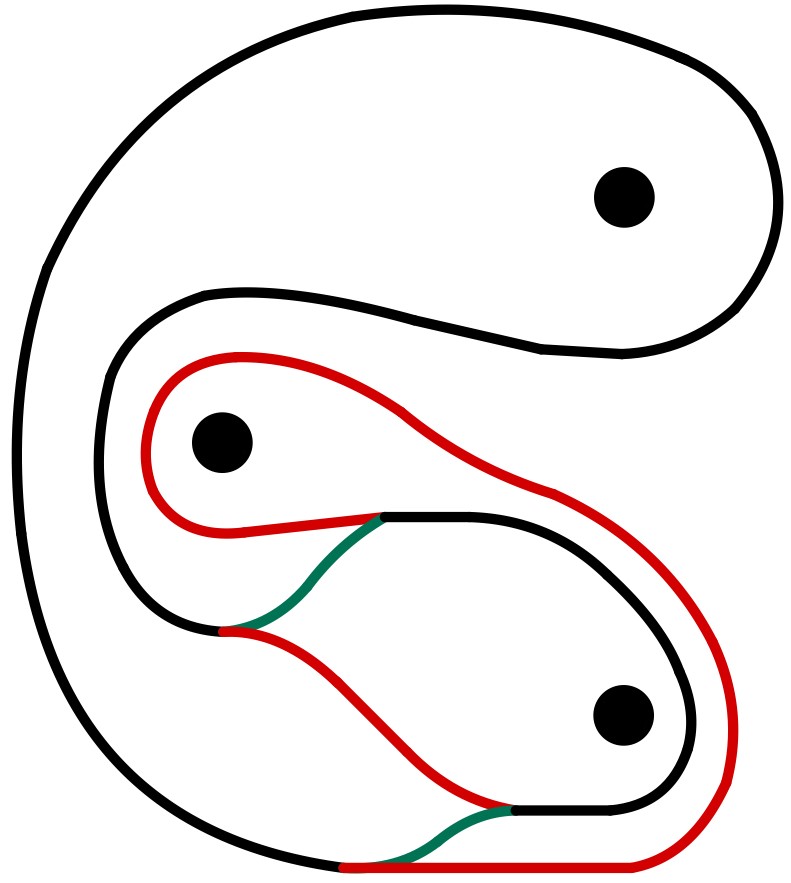}
\raisebox{40pt}{$\xrightarrow[\text{Isotopy}]{\text{Shift}}$}  
\scalebox{-1}[1]{\includegraphics[height=3cm]{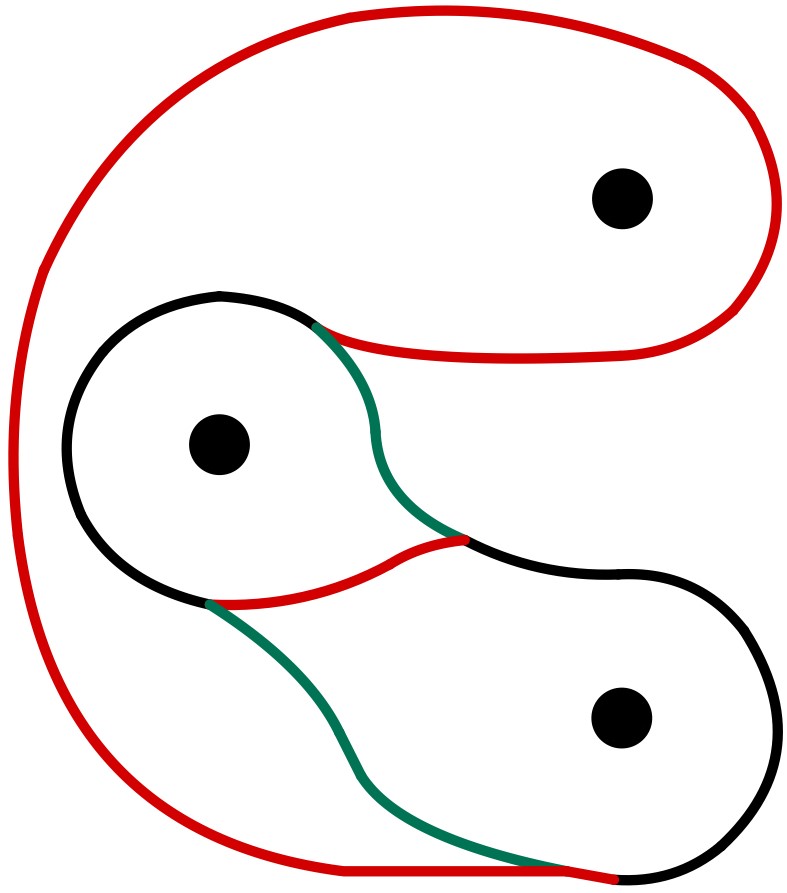}}
\put(-330, 15){Type I} 
\put(-320, 50){$\mathcal T$}
\put(-295, -15){\color{red}$x+y\alpha$}
\put(-255, -15){$< z+w\alpha$}
\put(-2, 15){Type I' or II'}
\put(8, 50){$\mathcal T^{spl}$}
\put(-75, -15){\color{red}$x+y\alpha$}
\put(-25, -15){\color{OliveGreen}$(z-x)+(w-y)\alpha$}
\caption{The Maximal splitting of Type I train track yields Type I' or II'.}
\label{fig:ABpart2}
\end{figure}
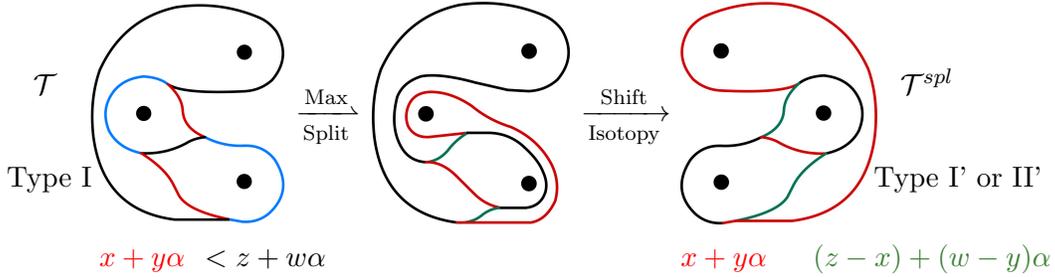

Looking into the weights carefully with Propositions~\ref{prop:algorithm of 4-tuple} and \ref{prop:new},  
we see that the 4-tuple of $\mathcal T^{spl}$ is Type 1 (resp. Type 2) if and only if    
$(z-x) + (w-y)\alpha < x +y\alpha$ (resp. $>$) 
if and only if the homeomorphism type of $\mathcal T^{spl}$ is Type I' (resp. Type II'). 
This yields the green arrow $\stackrel{(\ast)}{\rightharpoonup}$ (resp.~$\stackrel{(\ast\ast)}{\rightharpoonup}$) in Figure~\ref{figure-121'2'} coming out of Type I.

Suppose next that we have a Type II train track with 4-tuple $(x, y\ ;\ z, w)$. 
See Figure~\ref{fig:ABpart1}. 
After the maximal splitting we obtain a Type I (resp. Type II) train track if and only if 
$(z-x) + (w-y)\alpha < x +y\alpha$ (resp. $>$) 
if and only if 
the new 4-tuple (equivalently, the splitting type) is Type 1 (resp. Type 2). 
This yields the green arrow $\stackrel{(\ast)}{\rightharpoonup}$ (resp.~$\stackrel{(\ast\ast)}{\rightharpoonup}$) in Figure~\ref{figure-121'2'} coming out of Type II (resp.~coming back to itself). 
\begin{figure}[h]
\hspace{-.5cm}
\includegraphics[height=3cm]{A1.jpeg}  
\raisebox{40pt}{$\xrightharpoonup[\text{Split}]{\text{Max}}$}  
\includegraphics[height=3cm]{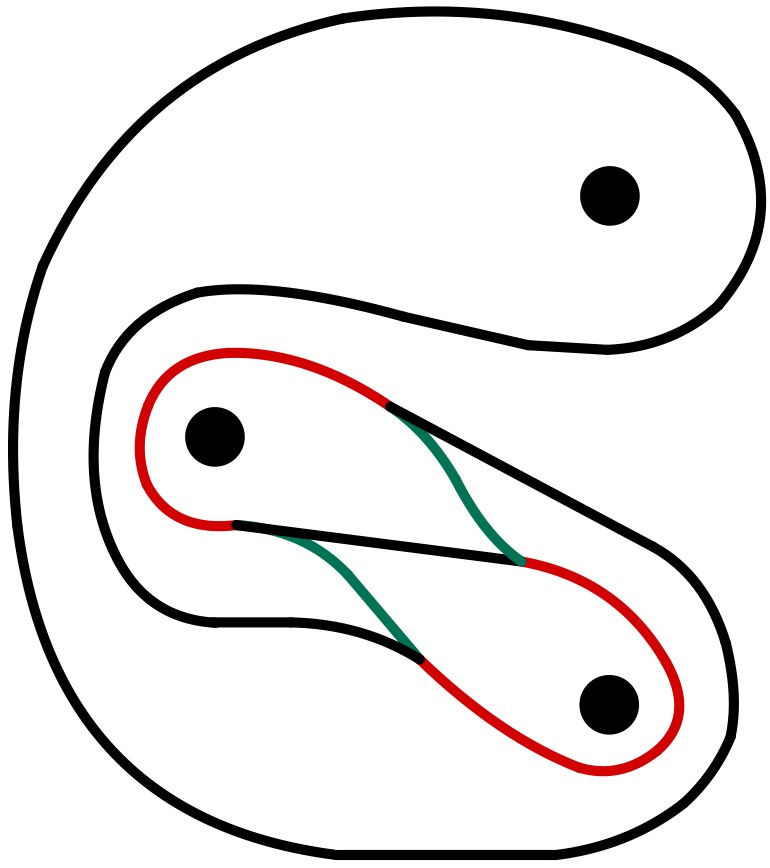}
\raisebox{40pt}{$\xrightarrow[\text{Isotopy}]{\text{Shift}}$}  
\includegraphics[height=3cm, angle=180,origin=c]{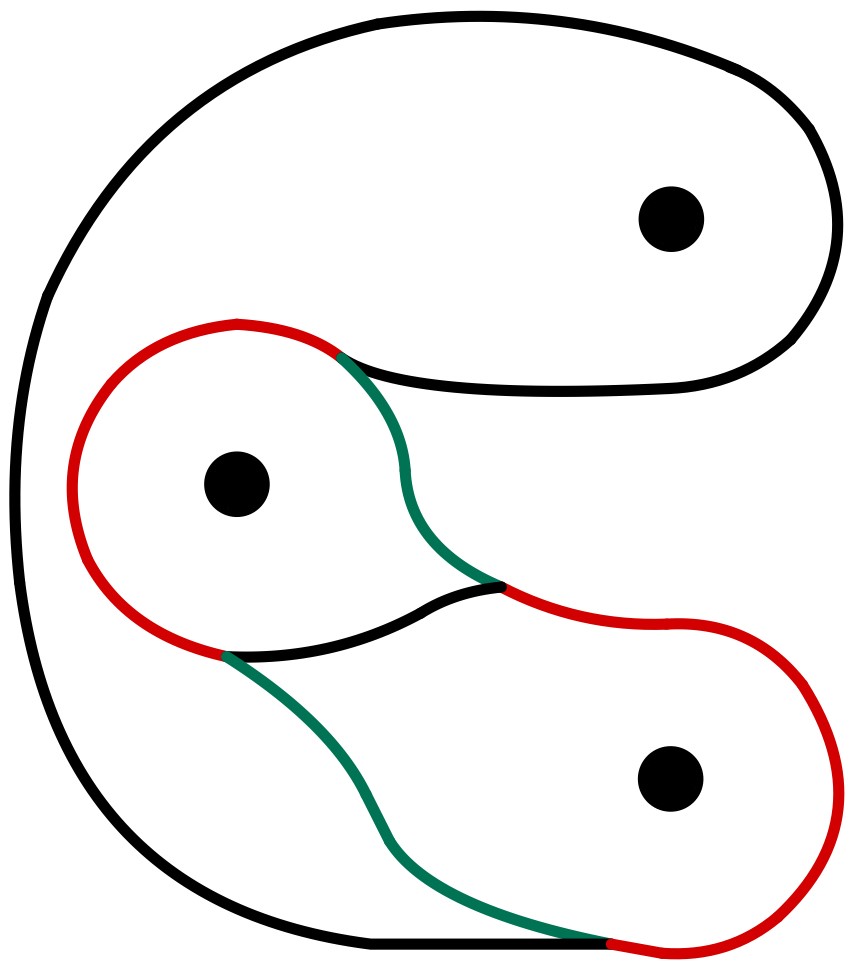}
\raisebox{40pt}{$\xrightarrow[\text{Isotopy}]{\text{Rotate}}$}  
\includegraphics[height=3cm]{A2Green.jpg}
\put(-390, -15){Type II}
\put(-405, -30){$x+y\alpha <$}
\put(-355, -30){\color{red}$z+w\alpha$}
\put(-90, -15){Type I or II \ \ $x+y\alpha$}
\put(-85, -30){\color{OliveGreen}$(z-x)+(w-y)\alpha$}
\caption{The Maximal splitting of Type II train track yields Type I or II.}
\label{fig:ABpart1}
\end{figure}

Recall Type I and Type I' are mirror to each other, and so are Type II and Type II'. This yields the rest of the maximal splitting arrows and completes the maximal splitting diagram in Figure~\ref{figure-121'2'}. 
\end{proof}

The above proof implies the following relation between topological types of triple-weight train tracks and types of 4-tuples. 

\begin{cor}\label{cor:rel_of_types}
Type I and Type I' triple train tracks are results of a Type 1 maximal splitting. 
Likewise, Type II and Type II' triple train tracks are results of a Type 2 maximal splitting. In other words, a triple-weight train track is 
\begin{itemize}
\item
Type I or I' if and only if its 4-tuple is Type 1, 
\item
Type II or II' if and only if its 4-tuple is Type 2. 
\end{itemize}
\end{cor}

\section{Triple-weight train track and Farey sequence}

In this section, we establish a relationship between triple-weight train tracks and Farey sequences that we use to prove our main results Theorem~\ref{thm:TFAE}.

Let $\alpha\in (\frac{1}{n+1}, \frac{1}{n}) \setminus \Q$ be the MP-ratio of some pseudo-Anosov 3-braid. 
We will define an infinite set of intervals and then construct a map $T$ which associates to each interval a 4-tuple of numbers. 
There is a unique sequence such that 
$I_{1,i_1} \supset I_{2,i_2} \supset I_{3,i_3} \supset \cdots \supset \bigcap_{k=1}^\infty I_{k,i_k}=\{\alpha\}$. 
Our goal of the section is to prove Theorem~\ref{thm:mapT}, where we relate the interval $I_{k,i_k}$ and the 4-tuple $(x_k, y_k \ ; \ z_k, w_k)$ under the map $T$. 

\subsection{Farey sequence}
First, we recall the {\em Farey sum}, $\boxplus$. 
Given two fractions, $\frac{s}{t}$ and $\frac{s'}{t'}$, we define:
\[
\frac{s}{t} \boxplus \frac{s'}{t'} := \frac{s+s'}{t+t'}.
\]
Notice that if $\frac{s}{t}<\frac{s'}{t'}$, then $\frac{s}{t} < \frac{s}{t} \boxplus \frac{s'}{t'}<\frac{s'}{t'}$.

For $k=0, 1, 2, \dots$, we inductively define an ordered set $L_k$ of fractions with cardinality $2^k+1$.  
The set $L_k$ is called the {\em Farey sequence} of order $k+1$.  
Define $$L_0=\{ a_{0,1}, a_{0,2}\}=\left\{ \frac{0}{1}, \frac{1}{1} \right\}$$ where we are deliberately not simplifying the fractions. 
Assume $L_k=\{a_{k,1}, a_{k,2}, \cdots, a_{k,2^k+1}\}$ where  
\[
\frac{0}{1}=a_{k,1}<a_{k,2}<\cdots<a_{k,2^k}<a_{k,2^k+1}=\frac{1}{1}.
\] 
Define 
\[
L_{k+1} =\{a_{k+1,1}, \cdots, a_{k+1,2^{k+1}+1} \}= L_k \cup \{a_{k,i} \boxplus a_{k,i+1} | i=1, 2, \cdots, 2^k \}
\]
as a set and then reorder them to have 
$\frac{0}{1}=a_{k+1,1}<a_{k+1,2}<\cdots < a_{k+1,2^{k+1}+1} = \frac{1}{1}.$

\begin{ex}
\begin{eqnarray*}
L_1 =  \left \{ a_{1,1}, a_{1,2}, a_{1,3} \right\} &=&   \left \{ \frac{0}{1}, \frac{1}{2}, \frac{1}{1}  \right \} \\ 
L_2 =  \left \{ a_{2,1}, \cdots, a_{2,5} \right\} &=&  \left \{ \frac{0}{1}, \frac{1}{3}, \frac{1}{2}, \frac{2}{3}, \frac{1}{1} \right \}\\
L_3 =  \left \{ a_{3,1}, \cdots, a_{3,9} \right\} &=&  \left \{ \frac{0}{1}, \frac{1}{4}, \frac{1}{3}, \frac{2}{5}, \frac{1}{2}, \frac{3}{5}, \frac{2}{3}, \frac{3}{4}, \frac{1}{1} \right \}
\end{eqnarray*}
\end{ex}

\begin{definition}
Let $n=\lfloor \frac{1}{\alpha} \rfloor$. For $k=0, 1, \cdots$ and $i=1, 2, \cdots, 2^k$, define the following open interval:
\[
I_{k,i}:=\left (\frac{1}{n+a_{k,i+1}}, \frac{1}{n+a_{k,i}}\right )
\]
\end{definition}
In particular, we have that $I_{0,1}=\left (\frac{1}{n+\frac{1}{1}}, \frac{1}{n+\frac{0}{1}} \right )=\left (\frac{1}{n+1}, \frac{1}{n} \right )$ and 
\begin{equation}\label{eq:intervals}
\bigsqcup_{i=1}^{2^k} I_{k,i} = \left [ \frac{1}{n+1}, \frac{1}{n} \right ] \setminus \left \{ \frac{1}{n+a} \mid a \in L_k\right \}.  
\end{equation}

\begin{ex} For $k=1$, 
\begin{eqnarray*} 
I_{1,2} \cup I_{1,1} &=&  \left(\frac{1}{n+\frac{1}{1}}, \frac{1}{n+\frac{1}{2}}\right ) \cup \left (\frac{1}{n+\frac{1}{2}}, \frac{1}{n+\frac{0}{1}}\right ) \\ 
&=& \left [ \frac{1}{n+1}, \frac{1}{n} \right ]  \setminus \left \{ \frac{1}{n+\frac{1}{1}},\ \frac{1}{n+\frac{1}{2}},\ \frac{1}{n+\frac{0}{1}} \right\} 
\end{eqnarray*}
For $k=2$, 
\begin{eqnarray*} 
&&I_{2,4} \cup I_{2,3} \cup I_{2,2} \cup I_{2,1} \\
&=& \left(\frac{1}{n+\frac{1}{1}}, \frac{1}{n+\frac{2}{3}}\right ) \cup 
\left(\frac{1}{n+\frac{2}{3}}, \frac{1}{n+\frac{1}{2}}\right ) \cup 
\left(\frac{1}{n+\frac{1}{2}}, \frac{1}{n+\frac{1}{3}}\right ) \cup 
\left(\frac{1}{n+\frac{1}{3}}, \frac{1}{n+\frac{0}{1}}\right ) \\
&=& 
\left [ \frac{1}{n+1}, \frac{1}{n} \right ]  \setminus \left\{ 
\frac{1}{n+\frac{1}{1}}, \frac{1}{n+\frac{2}{3}}, \frac{1}{n+\frac{1}{2}}, \frac{1}{n+\frac{1}{3}}, \frac{1}{n+\frac{0}{1}}
\right\}
\end{eqnarray*}
We observe that 
$I_{1,2}\setminus \left\{\frac{1}{n+\frac{2}{3}}\right\} =I_{2,4} \cup I_{2,3}$ and $I_{1,1} \setminus \left\{\frac{1}{n+\frac{1}{3}}\right\} =I_{2,2} \cup I_{2,1}$.
\end{ex}

In general, we have the following: 
\begin{lem}\label{L1}
For any $k=0, 1, \dots$ and $i=1, 2, \dots, 2^k$, the interval $I_{k,i}=\left (\frac{1}{n+a_{k,i+1}}, \frac{1}{n+a_{k,i}}\right )$ at the $k$th stage splits into two at the $(k+1)$st stage.
\begin{eqnarray}\label{eq:L1}
I_{k,i} \setminus \left\{  \frac{1}{n+a_{k+1,2i}}\right\}  
&=& 
I_{k+1,2i} \sqcup I_{k+1, 2i-1}\\
&=&
\left (\frac{1}{n+a_{k+1,2i+1}}, \frac{1}{n+a_{k+1,2i}}\right ) 
\sqcup
\left (\frac{1}{n+a_{k+1,2i}}, \frac{1}{n+a_{k+1,2i-1}}\right ) 
\nonumber
\end{eqnarray}
\end{lem}

A standard fact about Farey sequence is that 
$L_1 \subset L_2 \subset \cdots \subset \cup_{k=1}^\infty L_k = {\mathbb Q} \cap [0,1]$, which gives the following.

\begin{lem}\label{L2}
Let $\alpha \in \left ( \frac{1}{n+1}, \frac{1}{n} \right )\setminus \Q$. By Equation \ref{eq:intervals}, for every $k=0, 1, \cdots$, there exists an index $i_k \in \{ 1, 2,  3,\cdots, 2^k\}$ such that 
\begin{equation}\label{eq:alpha in interval}
\alpha \in I_{k,i_k}=\left ( \frac{1}{n+a_{k,i_k+1}}, \frac{1}{n+a_{k,i_k}} \right). 
\end{equation}
Moreover, by Lemma~\ref{L1} we obtain a sequence of nested intervals 
\begin{equation}\label{eq:nested_seq}
I_{0, 1} \supset I_{1,i_1} \supset I_{2,i_2} \supset I_{3,i_3} \supset \cdots \supset \bigcap_{k=1}^\infty I_{k,i_k}=\{\alpha\}. 
\end{equation}
with $i_k \in \{ 2 i_{k-1}, 2i_{k-1}-1\}$. 
In particular, the sequence $\{ i_k \}_{k=1}^\infty$ is an invariant of the irrational number $\frac{1}{n+1} < \alpha < \frac{1}{n}.$ 
\end{lem}

\subsection{Computation of 4-tuples}

The goal of this subsection is to prove Theorem~\ref{thm:mapT} which allows us to compute the 4-tuple $(x_k, y_k\, ; \, z_k, w_k)$ of the triple-weight train track $\mathcal T_k$ inductively. 
We do this using $2\times 2$ matrices.

\begin{definition}\label{def:of functionT}
We inductively define a function $T$ which associates to each interval $I_{k,i}$ a $2\times2$ matrix.
Set
\begin{equation}\label{base case}
T(I_{0,1})= \begin{pmatrix} 1 & 0 \\ -n & 1 \\ \end{pmatrix}.
\end{equation}
Suppose that we have defined the function for $I_{k,i}$ as
\begin{equation}\label{eq:def of T}
T(I_{k,i})=\begin{pmatrix} x & z \\ y & w \\ \end{pmatrix}.
\end{equation}
Then we define the function for $I_{k+1, 2i-1}$ and $I_{k+1, 2i}$ as follows:
\begin{equation}\label{eqn:def of T}
T(I_{k+1,2i}) = 
\begin{cases} 
\begin{pmatrix} x & z \\ y & w \\ \end{pmatrix}
\begin{pmatrix} -1 & 1 \\ 1 & 0 \\ \end{pmatrix}=
\begin{pmatrix} z-x & x \\ w-y & y \\ \end{pmatrix} 
& \text{ if } x>0 \\ 
\begin{pmatrix} x & z \\ y & w \\ \end{pmatrix}
\begin{pmatrix} 1 & -1 \\ 0 & 1 \\ \end{pmatrix}=
\begin{pmatrix} x & z-x \\ y & w-y \\ \end{pmatrix} & \text{ if } x<0
\end{cases} 
\end{equation}
\begin{equation}\label{eq:def of T for k+1}
T(I_{k+1,2i-1}) = 
\begin{cases} 
\begin{pmatrix} x & z \\ y & w \\ \end{pmatrix}
\begin{pmatrix} 1 & -1 \\ 0 & 1 \\ \end{pmatrix}=
\begin{pmatrix} x & z-x \\ y & w-y \\ \end{pmatrix} & \text{ if } x>0 \\ 
\begin{pmatrix} x & z \\ y & w \\ \end{pmatrix}
\begin{pmatrix} -1 & 1 \\ 1 & 0 \\ \end{pmatrix}=
\begin{pmatrix} z-x & x \\ w-y & y \\ \end{pmatrix} & \text{ if } x<0
\end{cases}
\end{equation}
\end{definition}

The next lemma shows that 
elements of the matrix $T(I_{k,i})$ are $\mathbb Z$-coefficient polynomials in $n$ of degree one and contains all the information of the intervals $I_{k, i}$. 

\begin{lem}\label{lemABCD}
For each $k, i$, there exist 
$\epsilon=\epsilon_{k,i}\in\{-1, 1\}$,  
$a=a_{k,i}\in \mathbb Z_{>0}$ and 
$b=b_{k,i}, c=c_{k,i}, d=d_{k,i} \in \mathbb Z_{\geq 0}$ 
such that 
\begin{equation}\label{eq:def of abcd}
T(I_{k,i})=\epsilon \begin{pmatrix} a & -b\\ -an-c  & bn+d \\ \end{pmatrix}.
\end{equation}
Moreover, the non-negative integers $a, b, c, d$ satisfy; 
\begin{equation}\label{eq:I}
I_{k,i}= \left\{
\begin{array}{lcc}
\left (\frac{1}{n+\frac{c+d}{a+b}}, \frac{1}{n+\frac{c}{a}}\right) &\mbox{\rm  if } & \epsilon = 1 \\
\left (\frac{1}{n+\frac{c}{a}}, \frac{1}{n+\frac{c+d}{a+b}}\right)
& \mbox{\rm if } & \epsilon = -1
\end{array}
\right.
\end{equation}
and
\begin{equation}\label{eq:II}
\det T(I_{k,i})=
ad-bc =  \sgn(\epsilon).
\end{equation}
We call $\epsilon_{k, i}$ the {\em sign} of the interval $I_{k, i}$. 
\end{lem}

\begin{proof}
(Induction on $k$)
When $k=0$, the initial condition (\ref{base case}) gives $\epsilon=1$, $a=1$, $b=0$, $c=0$, $d=1$, and $I_{0,1}=\left(\frac{1}{n+\frac{1}{1}}, \frac{1}{n+\frac{0}{1}}\right)$. Thus, all the assertions hold.

We assume that the assertions hold for $I_{k, i}$ and will show that assertions (\ref{eq:def of abcd}) and (\ref{eq:I}) hold for both $I_{k+1, 2i}$ and $I_{k+1, 2i-1}$. 

When $\epsilon_{k,i}=1$, by the induction hypothesis we have $I_{k,i}=\left(\frac{1}{n+\frac{c+d}{a+b}}, \frac{1}{n+\frac{c}{a}}\right)$. 
By (\ref{eqn:def of T}) and (\ref{eq:def of T for k+1}), we have
$$
T(I_{k+1, 2i})\stackrel{(\ref{eqn:def of T})}{=} - 
\begin{pmatrix} 
a+b & -a \\ -(a+b)n- (c+d) & an+c 
\end{pmatrix}=:- 
\begin{pmatrix} 
a' & -b'\\ -a'n-c'  & b'n+d'
\end{pmatrix}
$$
and 
$$
T(I_{k+1, 2i-1}) \stackrel{(\ref{eq:def of T for k+1})}{=} 
\begin{pmatrix} 
a & -(a+b) \\ -an- c & (a+b)n+c+d 
\end{pmatrix}=: 
\begin{pmatrix} 
a'' & -b''\\ -a''n-c''  & b''n+d''
\end{pmatrix}.
$$
We see that $a', a'' \in \mathbb Z_{>0}$ and $b', b'', c', c'', d, d'' \in \mathbb Z_{\geq 0}$. 
Thus (\ref{eq:def of abcd}) is satisfied for $k+1$ where $\epsilon_{k+1, 2i}=-1$ and $\epsilon_{k+1, 2i-1}=1.$
Next we check (\ref{eq:I}) for $k+1$.
By Lemma~\ref{L1} and the induction hypothesis (\ref{eq:I}),  
$$
I_{k+1, 2i}=
\left(\frac{1}{n+\frac{c+d}{a+b}}, \frac{1}{n+\frac{c+d}{a+b} \boxplus \frac{c}{a}}\right)=
\left(\frac{1}{n+\frac{c+d}{a+b}}, \frac{1}{n+\frac{2c+d}{2a+b}}\right)=\left(\frac{1}{n+\frac{c'}{a'}}, \frac{1}{n+\frac{c'+d'}{a'+b'}}\right)$$ 
and
$$I_{k+1, 2i-1}= 
\left(\frac{1}{n+\frac{c+d}{a+b} \boxplus \frac{c}{a}}, \frac{1}{n+\frac{c}{a}}\right)=
\left(\frac{1}{n+\frac{2c+d}{2a+b}}, \frac{1}{n+\frac{c}{a}}\right)=
\left(\frac{1}{n+\frac{c''+d''}{a''+b''}}, \frac{1}{n+\frac{c''}{a''}}\right).$$ 
Thus the assertion (\ref{eq:I}) is true for $I_{k+1, 2i}$ and $I_{k+1, 2i-1}$. 

When $\epsilon_{k, i}=-1$, a similar argument works. Thus, (\ref{eq:I}) is proved for $k+1$.  

Finally, for the last assertion (\ref{eq:II}) we observe that $|\det T(I_{k,i})|=|ad-bc|=1$ by the inductive definition of the function $T$. 
When $\epsilon=1$ by (\ref{eq:I}) we get 
$\frac{c}{a} < \frac{c+d}{a+b}$, which gives $0< ad-bc$ and thus $ad-bc=1$. 
Similarly when $\epsilon=-1$ we get $ad-bc=-1$. 
\end{proof}

Recall Lemma \ref{L1} which states that the interval $I_{k, i}$ splits into $I_{k+1, 2i}$ and $I_{k+1, 2i-1}$.

\begin{lem}\label{lem:sign of k+1}
The signs of the intervals  $I_{k+1, 2i}$ and $I_{k+1, 2i-1}$ are 
$$\epsilon_{k+1, 2i}=-1\mbox{ and } \epsilon_{k+1, 2i-1}=1.$$
\end{lem}

\begin{proof}
By (\ref{eq:def of T}) and (\ref{eq:def of abcd}) we have
$x=\epsilon a$ and $z=-\epsilon b$. 
Knowing that $a> 0$ we get $\sgn(x)=\epsilon$.
We also see $\sgn(z-x)=\sgn(-\epsilon(b+a))=-\epsilon=-\sgn(x)$.
The second equation follows since $a>0$ and $b\geq 0$. 
The definition of $T$ in (\ref{eq:def of T for k+1}) gives 
$$
\epsilon_{k+1, 2i-1}=\left\{
\begin{array}{lcc}
\sgn{(x)}=1 & \mbox{ if } & x>0 \\
\sgn{(z-x)}=1 & \mbox{ if } & x<0.
\end{array}\right.
$$
A similar argument with (\ref{eqn:def of T}) yields $\epsilon_{k+1, 2i}=-1$. 
\end{proof}

Now, we can finally relate this new function $T$ and the 4-tuple $(x_k, y_k\, ; \, z_k, w_k)$ of the train track $\mathcal T_k$ defined in Definition~\ref{def:four-tuple}. 

\begin{thm}\label{thm:mapT}
For every $k=0, 1, \cdots,$ the 4-tuple $(x_k, y_k\, ; \, z_k, w_k)$ of the train track $\mathcal T_k$ satisfies 
\begin{equation}\label{eq of T(I)}
\begin{pmatrix} x_k & z_k \\ y_k & w_k \\ \end{pmatrix}
=
T(I_{k,i_k})
\stackrel{{\rm Lem}\ref{lemABCD}}{=}
\epsilon_k \begin{pmatrix} a_k & -b_k\\ -a_kn-c_k  & b_kn+d_k \\ \end{pmatrix}
\end{equation}
for $\epsilon_k\in\{1, -1\}$ and some $a_k \in \mathbb Z_{>0}$ and $b_k, c_k, d_k \in \mathbb Z_{\geq 0}$ with $a_k d_k-b_k c_k = \epsilon_k$. 
\end{thm}

Note that $\epsilon_k=\sgn(\T_k)$, the {\em sign} of the train track $\mathcal T_k$ as in Definition~\ref{def of 4-tuple and ratio}.

\begin{proof}
The proof is done by induction.
When $k=0$, the assertion holds by (\ref{eq:initial 4-tuple}) and (\ref{base case}).  

Assume that (\ref{eq of T(I)}) holds for $k$. 
Recall that $\cap_{l=1}^\infty I_{l,i_l}=\{\alpha\}$.
Since $I_{k,i_k}$ splits into two intervals $I_{k+1, 2 i_k}$ and $I_{k+1, 2 i_k-1},$ 
we have either $I_{k+1, i_{k+1}}=I_{k+1, 2 i_k}$ or $I_{k+1, 2 i_k-1}$, equivalently $i_{k+1}=2i_k$ or $2i_k-1$.
There are four cases to consider depending on the sign of $\epsilon_k$ and the type (see Proposition~\ref{prop:algorithm of 4-tuple}) of the 4-tuple. 

Suppose that $\epsilon_k=1$ and $(x_{k+1}, y_{k+1} \ ; \ z_{k+1}, w_{k+1})$ has Type 1. 
By the induction hypothesis, we get 
$$
\epsilon_k (2(-a_kn-c_k)-(b_kn+d_k))\alpha 
\stackrel{(\ref{eq of T(I)})}{=}
(2y_k-w_k) \alpha 
\stackrel{\rm Type1}{>} z_k -2 x_k 
\stackrel{(\ref{eq of T(I)})}{=} 
\epsilon_k(-b_k-2a_k), 
$$ 
which yields 
$\alpha< \frac{1}{n+\frac{2c_k+d_k}{2a_k+b_k}}$. 
Since $I_{k, i_k}=\left(\frac{1}{n+\frac{c_k+d_k}{a_k+b_k}}, \frac{1}{n+\frac{c_k}{a_k}}\right)$ by (\ref{eq:I}) and the assumption $\epsilon_k=1$, we have 
$$
\alpha \in \left(\frac{1}{n+\frac{c_k+d_k}{a_k+b_k}}, \frac{1}{n+\frac{2c_k+d_k}{2a_k+b_k}}\right) \stackrel{(\ref{eq:L1})}{=}I_{k+1, 2i_k}. 
$$ 
This gives  
\begin{equation}\label{eqn:i=2i}
i_{k+1}=2i_k.
\end{equation}
The assertion can be verified as follows:
$$
\begin{pmatrix} x_{k+1} & z_{k+1} \\ y_{k+1} & w_{k+1} \\ \end{pmatrix}
\stackrel{\rm Prop\ref{prop:algorithm of 4-tuple}}{=}
\begin{pmatrix} z_k-x_k & x_k \\ w_k-y_k & y_k \\ \end{pmatrix}
\stackrel{(\ref{eqn:def of T})}{=} 
T(I_{k+1, 2i_k})=T(I_{k+1, i_{k+1}}).
$$

Similarly, we can verify the assertion for the remaining three cases.  
\end{proof}

The following technical corollary will be useful later. 

\begin{corollary}\label{cor:Type and interval}
In the above proof of Theorem~\ref{thm:mapT}, we 
have observed that:
\begin{enumerate}
\item
If $\epsilon_k=+1$ and $(x_{k+1}, y_{k+1} \ ; \ z_{k+1}, w_{k+1})$ has Type 1 then $i_{k+1}=2i_k$. 
\item
If $\epsilon_k=+1$ and $(x_{k+1}, y_{k+1} \ ; \ z_{k+1}, w_{k+1})$ has Type 2 then $i_{k+1}=2i_k-1$. 
\item
If $\epsilon_k=-1$ and $(x_{k+1}, y_{k+1} \ ; \ z_{k+1}, w_{k+1})$ has Type 1 then $i_{k+1}=2i_k-1$. 
\item
If $\epsilon_k=-1$ and $(x_{k+1}, y_{k+1} \ ; \ z_{k+1}, w_{k+1})$ has Type 2 then $i_{k+1}=2i_k$. 
\end{enumerate}
With Lemma~\ref{lem:sign of k+1}, we obtain:
\begin{itemize}
\item
for Case $(1)$ and $(4)$ $\epsilon_{k+1}=-1,$ and 
\item
for Case $(2)$ and $(3)$ $\epsilon_{k+1}=+1.$
\end{itemize}
In particular, a Type~1 maximal splitting $\stackrel{(\ast)}{\rightharpoonup}$ changes the sign of the train track and a Type~2 maximal splitting $\stackrel{(\ast\ast)}{\rightharpoonup}$ preserves the sign of the train track. Thus, in Figure~\ref{figure-121'2'}, we only see $(\pm)\stackrel{(\ast)}{\rightharpoonup}(\mp)$ and  $(\pm)\stackrel{(\ast\ast)}{\rightharpoonup}(\pm)$. 
\end{corollary}

\section{Nested Farey intervals}

In this section we define nested intervals $\{J_k\}_k$ in the Farey Tessellation and an ${\bf LR}$-sequence. 
 
Let $\beta$ be a pseudo-Anosov 3-braid with the MP-ratio $\alpha$. 
Let $n=\lfloor 1/\alpha \rfloor$. 
In Lemma~\ref{L2}, we have shown that $\alpha$ is the intersection of  nested intervals $I_{1,i_1} \supset I_{2,i_2} \supset I_{3,i_3} \supset \cdots$ where the interval $I_{k,i_k}=\left ( \frac{1}{n+a_{k,i_k+1}}, \frac{1}{n+a_{k,i_k}} \right).$ 
We define $J_k$ as a `reciprocal' of $I_{k, i_k}$:

\begin{definition}\label{def:nested Farey intervals} 
Define an open interval 
$$J_k := (a_{k,i_k} , a_{k,i_k+1})$$ 
for $k=0,1,\cdots$. 
Note that 
$J_0=(\frac{0}{1}, \frac{1}{1})$ and 
$J_k$ is a Farey interval since 
there is an arc in the Farey tessellation connecting the boundary points $a_{k,i_k}$ and $a_{k,i_k+1}.$

The fact $\alpha\in I_{k, i_k}$ is equivalent to $\frac{1}{\alpha}-n \in J_k$ and the nested sequence (\ref{eq:nested_seq}) can be translated into
\begin{equation}\label{eq:J sequence}
J_0 \supset J_1 \supset J_2 \supset J_3 \supset \cdots \supset \bigcap_{k=1}^\infty J_k = \biggl\{ \frac{1}{\alpha}-n\biggr\}, 
\end{equation}
which we call the sequence of {\em nested Farey intervals} for the braid $\beta$. 
\end{definition}

\begin{definition}\label{def:LandR}
The interval 
$J_k= (a_{k,i_k} ,  a_{k,i_k+1})$ splits into two Farey intervals; 
\begin{itemize}
\item[({\bf L})]
the left subinterval  $(a_{k,i_k} , \ a_{k,i_k}\boxplus a_{k,i_k+1})$
that is reciprocal to $I_{k+1, 2i_k-1}$, and
\item[({\bf R})]
the right subinterval 
$(a_{k,i_k}\boxplus a_{k,i_k+1}, \ a_{k,i_k+1})$
that is reciprocal to $I_{k+1, 2i_k}.$
\end{itemize}
For each $k=0, 1, \cdots$, the interval $J_{k+1}$ is exactly the left or right subinterval. 
We associate a letter {\bf L} or {\bf R} to each interval $J_{k+1}$ depending on the left or right subinterval status. 
The nested interval sequence (\ref{eq:J sequence}) can be encoded into a sequence in {\bf L} and {\bf R}. 
We call it the {\em{\bf LR}-sequence} for $\beta$. 
\end{definition}

\begin{example}\label{example:LR-sequence}
Let $\beta = \sigma_1^4 \sigma_2 \sigma_1^3 \sigma_2^4$ which is  a pseudo-Anosov 3-braid.  
The MP-ratio is $\alpha=(19-\sqrt{221})/14 \approx 0.2952 \cdots$ and $n=\lfloor 1/\alpha \rfloor = 3$. 
Here is an estimate of $1/\alpha - n$:  
$$
\frac{0}{1}< \frac{1}{3} < \frac{3}{8} < \frac{5}{13} <  \cdots < \frac{1}{\alpha}-n < \cdots
< \frac{7}{18} < \frac{2}{5} < \frac{1}{2} < \frac{1}{1}
$$
This gives the nested Farey intervals that converges to $1/\alpha - n$ (see Figure~\ref{fig:Farey-beta2}) 
$$J_0=(0/1, 1/1) \supset J_1=(0/1, 1/2) \supset J_2=(1/3, 1/2) \supset J_3=(1/3, 2/5)$$ 
$$\supset J_4=(3/8, 2/5) \supset J_5=(5/13, 2/5) \supset J_6=(5/13, 7/18) \supset \cdots
$$
and its associated {\bf LR}-sequence;  
${\bf L, R, L, R, R, L, \cdots}$.
To see this, we note that $J_1$ is the left subinterval of $J_0$; thus the first letter of the sequence is $\bf L$. Likewise, $J_2$ is the right subsequence of $J_1$; thus the second letter of the sequence is $\bf R$.

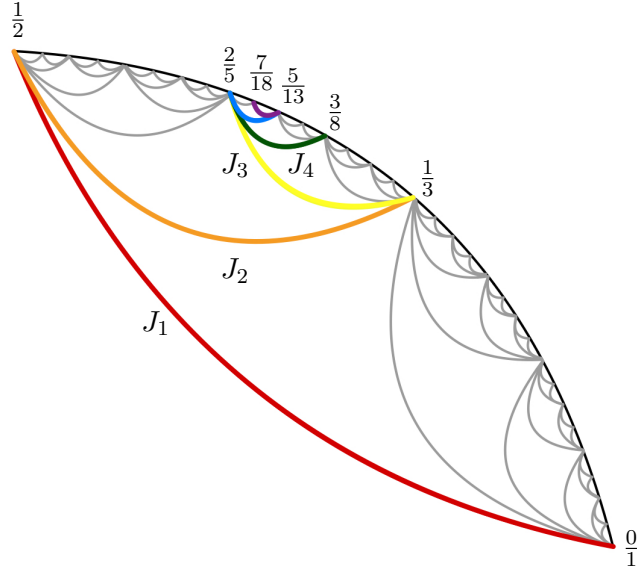
\begin{figure}[h]
\includegraphics[height=8cm]{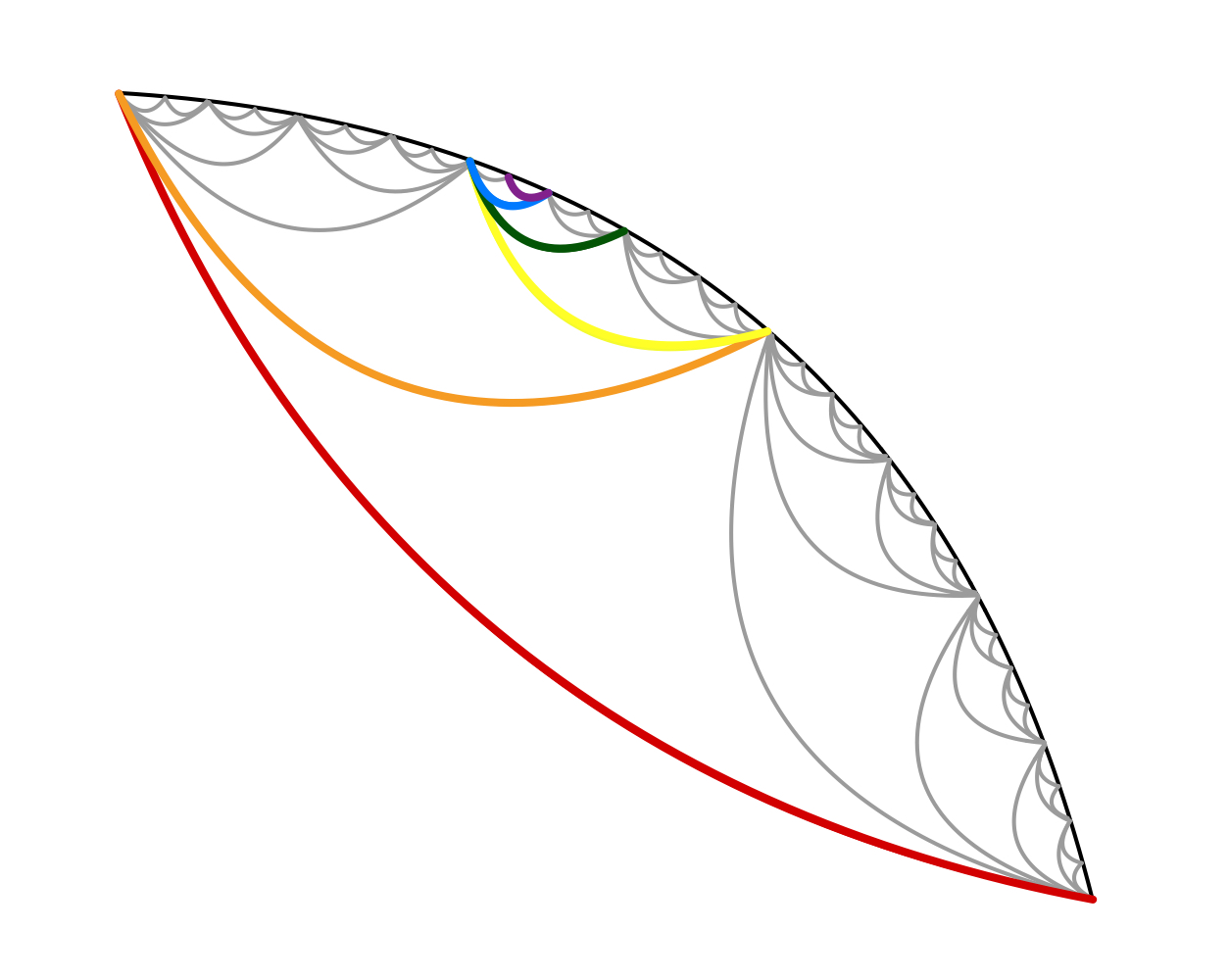}
\put(-28, 15){$\frac{0}{1}$}
\put(-260, 215){$\frac{1}{2}$}
\put(-105, 155){$\frac{1}{3}$}
\put(-180, 198){$\frac{2}{5}$}
\put(-140, 180){$\frac{3}{8}$}
\put(-158, 190){$\frac{5}{13}$}
\put(-170, 195){$\frac{7}{18}$}
\put(-210, 100){$J_1$}
\put(-180, 120){$J_2$}
\put(-180, 160){$J_3$}
\put(-155, 160){$J_4$}
\caption{(Example~\ref{example:LR-sequence}) 
Nested Farey intervals for $\beta=\sigma_1^4 \sigma_2 \sigma_1^3 \sigma_2^4$; 
$J_1=(0/1, 1/2)$ red, $J_2=(1/3, 1/2)$ orange, $J_3=(1/3, 2/5)$ yellow, $J_4=(3/8, 2/5)$ green, $J_5=(5/13, 2/5)$ blue, $J_6=(5/13, 7/18)$ purple. 
}
\label{fig:Farey-beta2}
\end{figure}
\end{example}



\begin{example}\label{examples beta and beta'} 
Let $\beta = \sigma_1^4 \sigma_2 \sigma_1^3 \sigma_2^4$ and $\beta'=\sigma_1^{-4} \sigma_2^{-4} \sigma_1^{-3} \sigma_2^{-1}$ be pseudo-Anosov 3-braids. 
It is interesting that $\beta$ is conjugate to $\beta' \Delta^8$ in $B_3$ and also $\beta$ is the mirror image of the negative flype (Definition~\ref{def:flype}) of $\beta'$, which is ${\tt flype}(\beta')=\sigma_1^{-4} \sigma_2^{-1} \sigma_1^{-3} \sigma_2^{-4}$. 

The I-II-I'-II'-sequences of $\beta$ and $\beta'$ are 
$$
{\rm II', I', I, I', II, I, II', I', I, I', II, I, II', I', I, I', II, I, II', I', I, I', II, I,} \cdots 
$$
and 
$$
{\rm I', I, II', I', I, I', II, I, II', I', I, I', II, I, II', I', I, I', II, I, II', I', I, I', II, I,} \cdots,
$$
respectively. 
The computations were completed with the use of MatLab. 
Removing the beginning ${\rm I', I}$ from the latter sequence, the two I-II-I'-II'-sequences become identical. 

The {\bf LR}-sequences of $\beta$ and $\beta'$ are 
$$
{\bf L, R, L, R, R, L, L, R, L, R, R, L, L, R, L, R, R, L, L, R, L, R, R, L, \cdots}
$$
and
$$
{\bf R, L, L, R, L, R, R, L, L, R, L, R, R, L, L, R, L, R, R, L, L, R, L, R, R, L, \cdots},
$$
respectively. 
Again, removing the beginning ${\bf R, L}$ from the latter sequence, the two {\bf LR}-sequences are identical. 

In fact, $\beta'$ has $\alpha' = (37+ \sqrt{221})/82 \approx 0.6325130335\cdots$ and $n' = \lfloor 1/\alpha' \rfloor = 1$ and 
nested Farey intervals 
$J_0=(\frac{0}{1}, \frac{1}{1})$ $\supset$ 
$J_1=(\frac{1}{2}, \frac{1}{1})$ $\supset$ 
$J_2=(\frac{1}{2}, \frac{2}{3})$ $\supset$ 
$J_3=(\frac{1}{2}, \frac{3}{5})$ $\supset$ 
$J_4=(\frac{4}{7}, \frac{3}{5})$ $\supset$ 
$J_5=(\frac{4}{7}, \frac{7}{12})$ $\supset$ 
$J_6=(\frac{11}{19}, \frac{7}{12})$ $\supset$ 
$J_7=(\frac{18}{31}, \frac{7}{12})$ $\supset$ 
$J_8=(\frac{18}{31}, \frac{25}{43})$ $\supset \cdots$ that are converging to $\frac{1}{\alpha'}-n'$. 
In Figure~\ref{fig:Farey-beta-prime} the intervals $J_3, \cdots, J_8$ are highlighted. 
\begin{figure}[h]
\includegraphics[height=8cm]{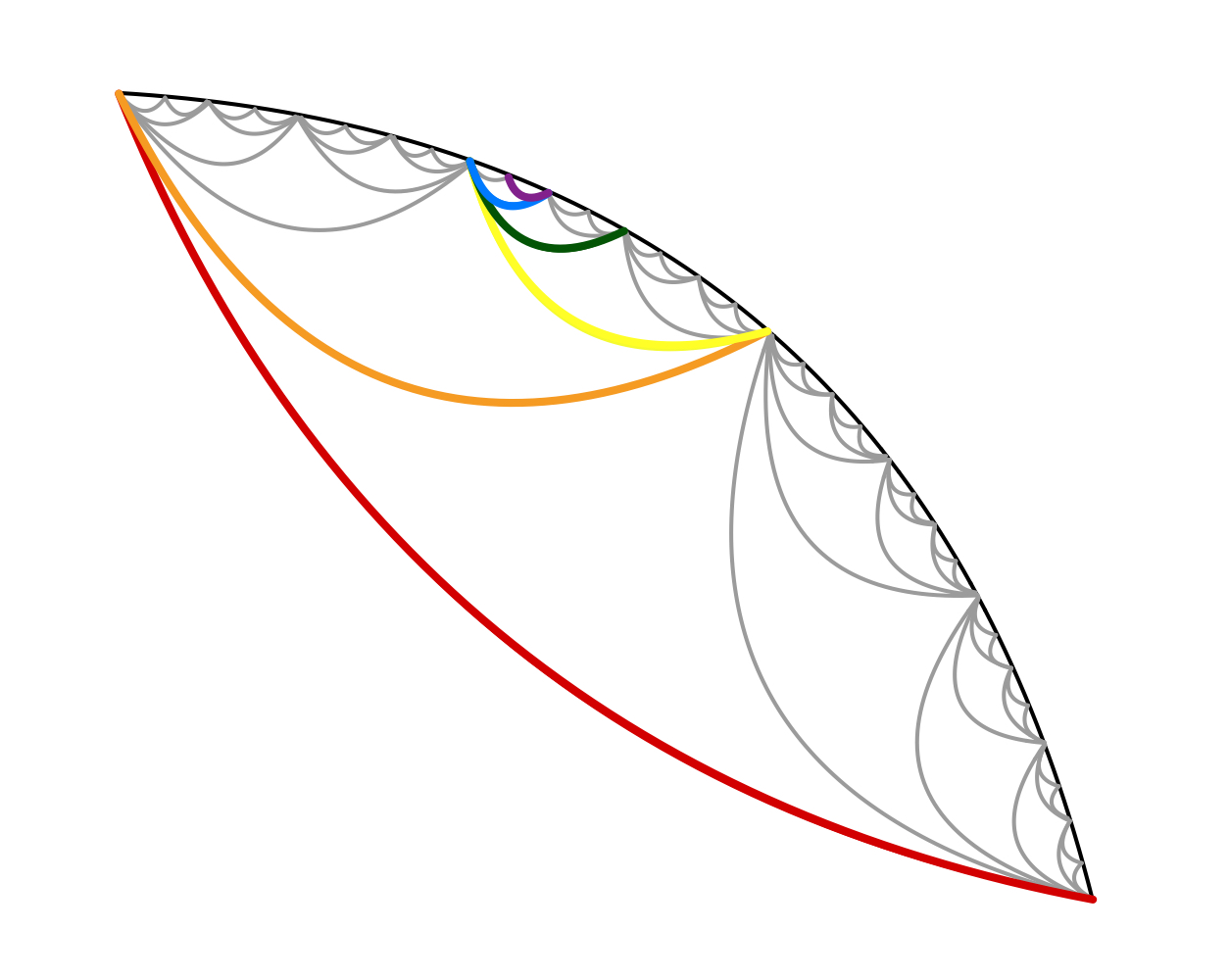}
\put(-28, 15){$\frac{1}{2}$}
\put(-260, 215){$\frac{3}{5}$}
\put(-105, 155){$\frac{4}{7}$}
\put(-180, 198){$\frac{7}{12}$}
\put(-140, 180){$\frac{11}{19}$}
\put(-158, 190){$\frac{18}{31}$}
\put(-170, 195){$\frac{25}{43}$}
\put(-210, 100){$J_3$}
\put(-180, 120){$J_4$}
\put(-180, 160){$J_5$}
\put(-155, 160){$J_6$}
\caption{(Example~\ref{examples beta and beta'}) 
Nested Farey intervals for $\beta'=\sigma_1^{-4} \sigma_2^{-4} \sigma_1^{-3} \sigma_2^{-1}$}
\label{fig:Farey-beta-prime}
\end{figure}
Up to zooming in, the pictures in Figures~\ref{fig:Farey-beta2} and \ref{fig:Farey-beta-prime} are the same.
\end{example}

\section{Necessary conditions for equivalent Agol cycles}
\label{sec:Proof of Main Theorem}

We are ready to state the main theorem that gives a number of necessary conditions for equivalent Agol cycles. Those conditions have different characteristics: topological, number theoretic, and numerical.

\begin{theorem}\label{thm:TFAE}
Let $\beta$ and $\beta'$ be pseudo-Anosov 3-braids. 
The statements (1) and (2) are equivalent (proved by Agol). 
The statement (2) implies (3),  (3) implies (4), and (4) implies (5).
\begin{enumerate}
\item
$\beta$ and $\beta'$ are conjugate in $B_3$ up to a center element. 
\item
$\beta$ and $\beta'$ have equivalent Agol cycles.
\item
{\em (Topological condition)}
There exist $l$ and $m \in \N$ such that $\sgn(\T_l)=\sgn(\T'_{m})$ and the triple-weight train track sequences 
$$
\T_l \msp \T_{l+1} \msp \T_{l+2} \msp \cdots \qquad (\mbox{for } \beta)
$$ 
and
$$
\T'_{m} \msp \T'_{m+1} \msp \T'_{m+2} \msp \cdots \qquad (\mbox{for } \beta')
$$ 
give the same periodic I-II-I'-II'-sequence. 
\item
{\em (Number theoretic condition)}
There exist $l$ and $m \in \N$ such that 
$\sgn(\T_l)=\sgn(\T'_{m})$ and the nested Farey interval sequences 
$J_l \supset J_{l+1} \supset J_{l+2} \supset \cdots$ for $\beta$ and 
$J'_{m} \supset J'_{m+1} \supset J'_{m+2} \supset \cdots$ for $\beta'$ 
give the same periodic ${\bf LR}$-sequences.  
\item
{\em (Numerical condition)}
There exist $l$ and $m \in \N$ such that 
the 4-ratios of $\T_l$ and $\T'_m$ are the same; namely, 
$$\frac{x_l+y_l\alpha}{z_l + w_l\alpha}=
\frac{x'_m+y'_m\alpha'}{z'_m + w'_m\alpha'}.$$
\end{enumerate}
\end{theorem}

The same {\bf LR}-sequence condition in (4) implies the 
existence of integers $A, B, C, D$ with $AD-BC=1$ or $-1$ (if $\sgn(\T_l) =-\sgn(\T'_{m})$ we have $AD-BC=-1$) such that 
$$
\frac{1}{\alpha} - n 
= \frac{A (\frac{1}{\alpha'}-n') + B}{C (\frac{1}{\alpha'}-n') +D}.
$$
In particular, when $l=m=1$ we obtain
$$\frac{1}{\alpha}-n=\frac{1}{\alpha'}-n'.$$
More detail and precise numbers are given in the proof of the theorem.

A feature of condition (5) is that it is not seeing the entire infinite sequences like (3) and (4), but rather focusing on two particular train tracks $\T_l$ and $\T'_m$. 

\begin{proof}[Proof of Theorem~\ref{thm:TFAE}]

The equivalence of (1) and (2) is due to Agol \cite{Agol}. See also Margalit's talk slides \cite{slide}.

The statement (2) implies (3) since Types I, II, I' and II' (introduced in Theorem~\ref{thm:closed-system} and illustrated in Figure~\ref{figure-121'2'}) are defined up to homeomorphism and these are the only homeomorphism types that appear in Agol cycles. 

(3) $\Rightarrow$ (4): 
We need to show for all $t=0, 1, 2, \cdots$, $J_{l+t}$ and $J'_{m+t}$ correspond to the same letter, either ${\bf L}$ or ${\bf R}$.

We first check the base case ($t=0$): 
By Lemma~\ref{lem:sign of k+1} and Definition~\ref{def:LandR}, the condition $\epsilon_l=\sgn(\T_l)=\sgn(\T'_{m})=\epsilon'_{m}$ implies that the Farey intervals $J_l$ and $J'_{m}$ correspond to the same letter, either ${\bf L}$ or ${\bf R}$. 

To see the next stage ($t=1$), we first note that by the assumption (3) and  Corollary~\ref{cor:rel_of_types},
the 4-tuples of $\T_{l+1}$ and $\T'_{m+1}$ have the same type (Type 1 or Type 2).  
Then by Corollary~\ref{cor:Type and interval}, $\epsilon_{l}=\epsilon'_{m}$ implies that  $\epsilon_{l+1}=\epsilon'_{m+1}$. 
The argument for the base case applies here and we can conclude that $J_{l+1}$ and $J'_{m+1}$ correspond to the same letter, either ${\bf L}$ or ${\bf R}$.

Inductively, for all $t=2, 3, \cdots,$ we can show that  $J_{l+t}$ and $J'_{m+t}$ correspond to the same letter, either ${\bf L}$ or ${\bf R}$.

(4) $\Rightarrow$ (5):
We first note that condition (4) is equivalent to the existence of a 2$\times$2 matrix 
$N=
\begin{pmatrix}
A & B\\
C & D
\end{pmatrix}
\in SL(2, \Z)$ that takes the Farey interval 
$J_{l+t}=
(\frac{u_t}{v_t}, \frac{w_t}{s_t})$ to $J'_{m+t}=
(\frac{u'_t}{v'_t}, \frac{w'_t}{s'_t})$ simultaneously; which means 
$N
\begin{pmatrix}
u_t & w_t\\
v_t & s_t
\end{pmatrix} 
=
\begin{pmatrix}
u'_t & w'_t\\
v'_t & s'_t
\end{pmatrix}
$ 
for each $t=0, 1, \cdots$. 
Since the nested sequences have convergences $\cap_{k=1}^\infty \{J_k\} = \{\frac{1}{\alpha}-n\}$ and $\cap_{k=1}^\infty\{J'_k\}= \{\frac{1}{\alpha'}-n'\}$, it is equivalent to 
$$
\frac{1}{\alpha} - n 
= \frac{A (\frac{1}{\alpha'}-n') + B}{C (\frac{1}{\alpha'}-n') +D}.
$$
In other words, the following vectors are parallel (the symbol $\parallel$ stands for parallel):
\begin{equation}\label{eq:parallel-N}
\begin{pmatrix}
\frac{1}{\alpha} - n \\
1
\end{pmatrix}
\parallel
N
\begin{pmatrix}
\frac{1}{\alpha'}-n' \\
1
\end{pmatrix}
\end{equation}

The matrix $N$ can be explicitly computed as follows:
By Lemma~\ref{lemABCD}, the Farey numbers $\frac{u_t}{v_t}, \frac{w_t}{s_t}, \frac{u'_t}{v'_t}, \frac{w'_t}{s'_t}$ for $t=0$ satisfy one of the two cases (here we use the assumption $\epsilon_l=\epsilon'_m$): 
\begin{eqnarray*}
\bullet \quad
\frac{u_0}{v_0}=\frac{c_l}{a_l}, \quad 
\frac{w_0}{s_0}=\frac{c_l+d_l}{a_l+b_l}, \quad
\frac{u'_0}{v'_0}=\frac{c'_m}{a'_m}, \quad
\frac{w'_0}{s'_0}=\frac{c'_m+d'_m}{a'_m+b'_m} && 
\mbox{ if } \epsilon_l=\epsilon_m=1,\\
\bullet \quad
\frac{u_0}{v_0}=\frac{c_l+d_l}{a_l+b_l}, \quad
\frac{w_0}{s_0}= \frac{c_l}{a_l}, \quad
\frac{u'_0}{v'_0}=\frac{c'_m+d'_m}{a'_m+b'_m},   \quad
\frac{w'_0}{s'_0}=\frac{c'_m}{a'_m} && 
\mbox{ if } \epsilon_l=\epsilon_m=-1.
\end{eqnarray*}
In either case, 
$$N =
\begin{pmatrix}
c_l & d_l\\
a_l & b_l
\end{pmatrix}
\begin{pmatrix}
c'_m & d'_m\\
a'_m & b'_m
\end{pmatrix}
^{-1}
$$ 
and $\det N = (-\epsilon_l) (-\epsilon'_m) =1$ since $\epsilon_l=\epsilon_m$.

Next, we note that
$$
\begin{pmatrix}
a_l & -c_l \\
-b_l & d_l
\end{pmatrix}^{-1}
\begin{pmatrix}
a'_m & -c'_m \\
-b'_m & d'_m
\end{pmatrix}
=
\begin{pmatrix}
c_l & d_l\\
a_l & b_l
\end{pmatrix}
\begin{pmatrix}
-1 & 0 \\
0 & -1
\end{pmatrix}
\begin{pmatrix}
c'_m & d'_m\\
a'_m & b'_m
\end{pmatrix}
^{-1}
 =- N.  
$$ 
By  (\ref{eq:parallel-N}) we have
\begin{equation}\label{eq:parallel2}
\begin{pmatrix}
a_l & -c_l \\
-b_l & d_l
\end{pmatrix}
\begin{pmatrix}
\frac{1}{\alpha}-n \\
1
\end{pmatrix}
\parallel
\begin{pmatrix}
a'_m & -c'_m \\
-b'_m & d'_m
\end{pmatrix}
\begin{pmatrix}
\frac{1}{\alpha'}-n' \\
1
\end{pmatrix}.
\end{equation}
Using Theorem~\ref{thm:mapT}, we can rewrite the 4-ratio of $\T_l$ as
$$
\frac{x_l+y_l\alpha}{z_l+w_l\alpha}
=\frac{a_l-(a_ln+c_l)\alpha}{-b_l + (b_ln+d_l)\alpha}
=\frac{a_l(\frac{1}{\alpha}-n)-c_l}{-b_l(\frac{1}{\alpha}-n)+d_l}, 
$$
which is equal to the slope of the left hand side vector of (\ref{eq:parallel2}). 
Likewise the 4-ratio of $\T'_m$ is equal to the slope of the right hand side vector of (\ref{eq:parallel2}). 
Therefore, we conclude the 4-ratios of $\T_l$ and $\T'_m$ are the same:
$\frac{x_l+y_l\alpha}{z_l+w_l\alpha}=\frac{x'_m+y'_m\alpha'}{z'_m+w'_m\alpha'}$.
\end{proof}

\section{Sufficient conditions for equivalent Agol cycles} 

In Section~\ref{sec:Proof of Main Theorem}, we studied necessary conditions of equivalent Agol cycles. 
In this section we explore sufficient conditions, especially related to the converse of (5) $\Rightarrow$ (2) in Theorem~\ref{thm:TFAE}. 
Our goal is to prove: 
if $\mathcal{T}_l$ and $\mathcal{T}_m'$ have the same 4-ratio for some $l$ and $m$, then Agol cycles of $\beta$ and $\beta'$ are equivalent or mirror equivalent.

Let $\T$ be a triple-weight train track and ${\tt Type}(\T) \in \{ {\rm I, II, I', II'}\}$ denote the homeomorphism type (Type I, II, I', and II') as introduced in Figure~\ref{figure-121'2'}. 

\begin{thm}\label{BigThm}
Let $\beta$ and $\beta'$ be pseudo-Anosov 3-braids. 
Suppose that 
there exist $l$ and $m$ such that triple-weight train tracks $\mathcal{T}_l$ and $\mathcal{T}_m'$ have the same 4-ratio; i.e., $\frac{x_l+y_l\alpha}{z_l+w_l\alpha}=\frac{x'_m+y'_m\alpha'}{z'_m+w'_m\alpha'}$. 
Then one of the following cases occur:  
\begin{enumerate}
\item
$\Type (\T_{l+1}) = \Type(\T'_{m+1})$ and
\item
$\Type (\T_{l+1}) = \Type(\T'_{m+1})'$ 
\end{enumerate}
In Case (1), 
$\beta$ and $\beta'$ have equivalent Agol cycles, and in Case (2) they have mirror equivalent Agol cycles. 
\end{thm}

\begin{example}\label{ex:beta''}
Related to Example~\ref{examples beta and beta'}, 
let 
$\beta = \sigma_1^4 \sigma_2 \sigma_1^3 \sigma_2^4$, 
$\beta'=\sigma_1^{-4} \sigma_2^{-4} \sigma_1^{-3} \sigma_2^{-1}$,
and
$\beta''=\sigma_1^{-4} \sigma_2^{-1} \sigma_1^{-3} \sigma_2^{-4}.$  
The pair ($\beta, \beta'$) falls into Case (1) of Theorem~\ref{BigThm}. 
and the pair ($\beta'', \beta'$) falls into Case (2). 
It is interesting to point out that $\beta''$ and $\beta'$ are related by a non-degenerate flype. Thus they belong to distinct conjugacy classes but their Agol cycles are mirror equivalent. 
\end{example}

Below is a lemma needed to prove Theorem~\ref{BigThm}.

\begin{lem}\label{TechLem}
Suppose that $\mathcal{T}_l$ and $\mathcal{T}_m'$ have the same 4-ratio. 
Then for all $t\geq 1$, the subsequent train tracks $\mathcal{T}_{l+t}$ and $\mathcal{T}_{m+t}'$ have the same 4-ratio and their 4-tuples have the same type 
(either Type 1 or Type 2). 
\end{lem}

\begin{proof}[Proof of Lemma~\ref{TechLem}]

Define $\A, \B, \C, \D$ by $$\A+\B \alpha+\C \alpha '+\D \alpha \alpha'=
(x_l+y_l\alpha) (z'_m+w'_m\alpha')-(z_l+w_l\alpha)(x'_m+y'_m\alpha').
$$
Since $\T_l$ and $\T'_m$ have the same 4-ratio $$\A+\B \alpha+\C \alpha '+\D \alpha \alpha'=0.$$ 
Using Theorem~\ref{thm:mapT} we may describe $\A, \B, \C, \D$ as follows. 
\begin{eqnarray*}
\A &=&
\epsilon_l \epsilon_m' (-a_l b_m'+b_l a_m')\\
\B &=&
\epsilon_l \epsilon_m' \left((a_l b_m'-b_l a_m')n + (c_l b_m'-d_l a_m')\right)\\
\C &=&
\epsilon_l \epsilon_m' \left((a_l b_m'-b_l a_m')n' + (a_l d_m' - b_l c_m')\right)\\
\D &=&
\epsilon_l \epsilon_m' \left( (a_l b_m'-b_l a_m') nn' + (a_l d_m' - b_l c_m')n + (c_l b_m'-d_l a_m')n'
+ (c_l d_m' - d_l c_m')\right)
\end{eqnarray*}

We have:
\begin{eqnarray*}
\C \alpha'+ \A &=& -\alpha(\D \alpha' +B) \\ 
\frac{1}{\alpha}  &=& -\frac{\D \alpha' +\B}{\C \alpha'+\A} \\ 
\frac{1}{\alpha} -n &=& - \frac{(\D+n\C) \alpha' +(\B+n\A)}{\C \alpha'+\A} \\ 
&=& 
\frac{-(\B+n\A)(\frac{1}{\alpha'}-n') -(\D+n\C)-n'(\B+n\A)}{\A(\frac{1}{\alpha'}-n') +\C+\A n'} \\
&=&
\frac{\epsilon_l\epsilon_m'(d_la_m'-c_lb_m')(\frac{1}{\alpha'}-n') + \epsilon_l\epsilon_m'(c_ld_m'-d_lc_m')}{\epsilon_l\epsilon_m'(-a_l b_m'+b_l a_m')(\frac{1}{\alpha'}-n')+ \epsilon_l\epsilon_m'(a_l d_m' - b_l c_m')}
\end{eqnarray*}
From the last fraction, we define a matrix
$$
N= \epsilon_l\epsilon_m'
\begin{pmatrix}
d_la_m'-c_lb_m' & -d_lc_m'+c_ld_m' \\
b_l a_m'-a_l b_m' & - b_l c_m'+a_l d_m' 
\end{pmatrix}
=
\epsilon_l
\begin{pmatrix}
c_l & d_l\\
a_l & b_l
\end{pmatrix}
\begin{pmatrix}
c'_m & d'_m\\
a'_m & b'_m
\end{pmatrix}
^{-1}
$$
so that the following vectors are parallel:
\begin{equation}\label{eq:parallel}
\begin{pmatrix}
\frac{1}{\alpha} -n\\
1
\end{pmatrix}
\parallel 
N
\begin{pmatrix}
\frac{1}{\alpha'} - n'\\
1
\end{pmatrix}  
\end{equation}
For later use, we note that: 
\begin{equation}\label{eqn:NN}
N 
\begin{pmatrix}
c_m' & d_m'\\
a'_m & b_m' 
\end{pmatrix} 
= 
\epsilon_l
\begin{pmatrix}
c_l & d_l\\
a_l & b_l
\end{pmatrix},
\end{equation}
and the determinant of $N$ is either 1 or -1. 
\begin{equation}\label{det of N}
\det N =(b_l c_l-a_l d_l)(b_m' c_m'-a_m' d_m')^{-1}
\stackrel{(\ref{eq:II})}{=}
\epsilon_l \epsilon_m'\in \{-1, 1\}
\end{equation}

%

By Lemma~\ref{L2}, 
$\alpha \in \bigcap_{k=0}^\infty I_{k,i_k}$ and $\alpha' \in \bigcap_{k=0}^\infty I'_{k,i_k'}$.
By Lemma~\ref{lemABCD}, 
we have
\begin{equation}\label{eq:I'}
I_{l,i_l}= \left\{
\begin{array}{lcccl}
\left (\frac{1}{n+\frac{c_l+d_l}{a_l+b_l}}, \frac{1}{n+\frac{c_l}{a_l}}\right) &\mbox{ thus }
& \frac{c_l}{a_l} < \frac{1}{\alpha}-n < \frac{c_l+d_l}{a_l+b_l}
&\mbox{\rm  if } & \epsilon_l = 1 
\\
\left (\frac{1}{n+\frac{c_l}{a_l}}, \frac{1}{n+\frac{c_l+d_l}{a_l+b_l}}\right)
&\mbox{ thus }
&  \frac{c_l+d_l}{a_l+b_l} < \frac{1}{\alpha}-n <\frac{c_l}{a_l}
& \mbox{\rm if } & \epsilon_l = -1
\end{array}
\right.
\end{equation}
and
\begin{equation}\label{eqn of I'}
I'_{m,i'_m}= \left\{
\begin{array}{lcccl}
\left (\frac{1}{n'+\frac{c_m'+d_m'}{a_m'+b_m'}}, \frac{1}{n+\frac{c_m'}{a_m'}}\right) 
&\mbox{ thus }
& \frac{c_m'}{a_m'} < \frac{1}{\alpha'}-n' < \frac{c_m'+d_m'}{a_m'+b_m'}
&\mbox{\rm  if } & \epsilon_m' = 1 \\
\left (\frac{1}{n'+\frac{c_m'}{a_m'}}, \frac{1}{n'+\frac{c_m'+d_m'}{a_m'+b_m'}}\right)
&\mbox{ thus }
&  \frac{c_m'+d_m'}{a_m'+b_m'} < \frac{1}{\alpha'}-n' < \frac{c_m'}{a_m}& \mbox{\rm if} & \epsilon_m' = -1.
\end{array}
\right.
\end{equation}
By Lemma~\ref{L1}, 
$I_{l, i_l}=\left(\frac{1}{n+\frac{p_l}{q_l}}, \frac{1}{n+\frac{r_l}{s_l}}\right)$ splits into 
$I_{l+1, 2i_l}=\left(\frac{1}{n+\frac{p_l}{q_l}}, \frac{1}{n+\frac{p_l+r_l}{q_l+s_l}}\right)$ and 
$I_{l+1, 2i_l-1}=\left(\frac{1}{n+\frac{p_l+r_l}{q_l+s_l}}, \frac{1}{n+\frac{r_l}{s_l}}\right)$. 
Thus $i_{l+1}=2i_l$ or $2i_l-1$. 
We observe
\begin{itemize}
\item
If $i_{l+1}=2 i_l$ then $\frac{p_l}{q_l}\boxplus\frac{r_l}{s_l}=\frac{p_l+r_l}{q_l+s_l} < \frac{1}{\alpha} -n <  \frac{p_l}{q_l}$.
\item
If $i_{l+1}=2 i_l-1$ then $\frac{r_l}{s_l}< \frac{1}{\alpha} -n <  \frac{p_l}{q_l}\boxplus\frac{r_l}{s_l}=\frac{p_l+r_l}{q_l+s_l}.$
\end{itemize}
Similarly, $i'_{m+1}=2i'_m$ or $2i'_m-1$. 
\begin{itemize}
\item
If $i'_{m+1}=2 i'_m$ then $\frac{p'_m+r'_m}{q'_m+s'_m} < \frac{1}{\alpha'} -n' <  \frac{p'_m}{q'_m}$.
\item
If $i'_{m+1}=2 i'_m-1$ then $\frac{r'_m}{s'_m}< \frac{1}{\alpha'} -n' <  \frac{p'_m+r'_m}{q'_m+s'_m}.$
\end{itemize}

\begin{claim}\label{claim}
The indices $i_{l+1}$ and $i'_{m+1}$ obey the following rule. 
\begin{itemize}
\item
If $\epsilon_l \epsilon_m'=1$ then $i_{l+1}=2i_l$ if and only if $i'_{m+1}=2i'_m$. 
\item
If $\epsilon_l \epsilon_m'=-1$ then $i_{l+1}=2i_l-1$ if and only if $i'_{m+1}=2i'_m$. 
\end{itemize}
\end{claim}

\begin{proof}[Proof of Claim~\ref{claim}]
We have four cases to consider.

\noindent
({\bf Case 1}: $\epsilon_l=\epsilon_m'=1$) 
By (\ref{eqn of I'}) with $\epsilon_m'=1$, we note that
$\frac{p'_m}{q'_m}=\frac{c_m'+d_m'}{a_m'+b_m'}$ and
$\frac{r'_m}{s'_m}=\frac{c_m'}{a_m'}.$ Thus, we have
$i'_{m+1}=2i'_m$ if and only if 
$$
\frac{2c_m'+d_m'}{2a_m'+b_m'} =
\frac{p'_m}{q'_m}\boxplus\frac{r'_m}{s'_m}
< \frac{1}{\alpha'}-n' 
< 
\frac{p'_m}{q'_m}=
\frac{c_m'+d_m'}{a_m'+b_m'}. 
$$
Since $\det N =1 >0$ by (\ref{det of N}) the slopes of the following three vectors satisfy 
$$
{\tt slope} \left(N
\begin{pmatrix}
2c_m'+d_m'\\
2a_m'+b_m'
\end{pmatrix}\right) 
< 
{\tt slope} \left(N
\begin{pmatrix}
\frac{1}{\alpha'}-n' \\
1
\end{pmatrix}\right)
< 
{\tt slope} \left(N
\begin{pmatrix}
c_m'+d_m'\\
a_m'+b_m'
\end{pmatrix}\right).
$$
By (\ref{eq:parallel}) and (\ref{eqn:NN}), we obtain 
$$
\frac{2c_l+d_l}{2a_l+b_l} < \frac{1}{\alpha}-n < \frac{c_l+d_l}{a_l+b_l}. 
$$ 
By (\ref{eq:I'}), with $\epsilon_l=1$ it is equivalent to $i_{l+1}=2i_l$. 

\noindent
({\bf Case 2}: $\epsilon_l=\epsilon_m'=-1$)
We have $i'_{m+1}=2i'_m$ if and only if 
$
\frac{2c_m'+d_m'}{2a_m'+b_m'} < \frac{1}{\alpha'}-n' < \frac{c_m'}{a_m'}.
$
Since $\det N >0$ we obtain  
$
\frac{2c_l+d_l}{2a_l+b_l} < \frac{1}{\alpha}-n < \frac{c_l}{a_l}.
$
It is equivalent to $i_{l+1}=2i_l$. 

\noindent
({\bf Case 3}: $\epsilon_l=-\epsilon_m'=1$) 
We have $i'_{m+1}=2i'_m$ if and only if 
$
\frac{2c_m'+d_m'}{2a_m'+b_m'} < \frac{1}{\alpha'}-n' < \frac{c_m'}{a_m'}.
$
Since $\det N <0$ we obtain
$
\frac{c_l}{a_l} < \frac{1}{\alpha}-n < \frac{2c_l+d_l}{2a_l+b_l}.
$
It is equivalent to $i_{l+1}=2i_l-1$. 

\noindent
({\bf Case 4}: $-\epsilon_l=\epsilon_m'=1$) 
We have $i'_{m+1}=2i'_m$ if and only if 
$
\frac{2c_m'+d_m'}{2a_m'+b_m'} < \frac{1}{\alpha'}-n' < \frac{c_m'+d_m'}{a_m'+b_m'}. 
$
Since $\det N <0$ we have 
$
\frac{c_l+d_l}{a_l+b_l} < \frac{1}{\alpha}-n < \frac{2c_l+d_l}{2a_l+b_l}.
$
It is equivalent to $i_{l+1}=2i_l-1$. 
\end{proof}

It is interesting to note that Claim~\ref{claim} and Lemma~\ref{lem:sign of k+1} imply that the product of the signs is preserved: 
$$\epsilon_{l+1}\epsilon_{m+1}'=\epsilon_l\epsilon_m'.$$

\begin{claim}\label{claim for Type}
The 4-tuples of $\mathcal T_{l+1}$ and $\mathcal T'_{m+1}$ have the same type (Type 1 or Type 2) as stated in the following table.
\end{claim}

\begin{center}
\begin{tabular}{ |l||c|c||c|c| }
\hline
 & Case 1 & Case 2 & Case 3 & Case 4 \\
$(\sgn(\epsilon_l), \ \sgn(\epsilon_m'))$ & $(+, +)$ & $(-, -)$ & $(+, -)$ & $(-, +)$\\
\hline
\hline
$i_{l+1}=2i_l$ & Type 1 & Type 2 & Type 1 & Type 2 \\
\hline
$i_{l+1}=2i_l-1$ & Type 2 & Type 1 & Type 2 & Type 1\\
\hline
\end{tabular}
\end{center}
\begin{proof}
All the eight cases can be checked similarly. 
For example, we check the claim for Case 3 where $i_{l+1}=2i_l-1$. 
By Corollary~\ref{cor:Type and interval}-(2) the 4-tuple of $\mathcal T_{l+1}$ is Type 2. 
Next, Claim~\ref{claim} states $i_{m+1}'=2i_m'$. 
By Corollary~\ref{cor:Type and interval}-(4), the 4-tuple of $\mathcal T'_{m+1}$ is Type~2.  
\end{proof}

Lastly, we will show that $\mathcal T_{l+1}$ and $\mathcal T'_{m+1}$ have the same 4-ratio, which by induction concludes Lemma~\ref{TechLem}.

Since $\T_l$ and $\T'_m$ have the same 4-ratio 
\begin{equation}\label{eq_of_ratio}
(x_l+y_l\alpha) (z'_m+w'_m\alpha')-(z_l+w_l\alpha)(x'_m+y'_m\alpha')=0.
\end{equation} 
Having Claim~\ref{claim for Type} proved, both $\T_{l+1}$ and $\T'_{m+1}$ have the same 4-tuple type, say Type ~1.  
By Proposition~\ref{prop:algorithm of 4-tuple}, their 4-tuples are: 
\begin{eqnarray}\label{eq:next4tuples}
&&(x_{l+1}, y_{l+1}, z_{l+1}, w_{l+1})=(z_l-x_l, w_l-y_l\ ; \ x_l, y_l) 
\nonumber \\
&\mbox{and}& (x'_{m+1}, y'_{m+1}, z'_{m+1}, w'_{m+1})=(z'_m-x'_m, w'_m-y'_m\ ; \ x'_m, y'_m).
\end{eqnarray}
We obtain 
$\frac{x_{l+1}+y_{l+1}\alpha}{z_{l+1}+w_{l+1}\alpha}=\frac{x'_{m+1}+y'_{m+1}\alpha'}{z'_{m+1}+w'_{m+1}\alpha'}$ since
\begin{eqnarray*}
&&(x_{l+1}+y_{l+1}\alpha) (z'_{m+1}+w'_{m+1}\alpha') - (z_{l+1}+w_{l+1}\alpha) (x'_{m+1}+y'_{m+1}\alpha')\\
&\stackrel{(\ref{eq:next4tuples})}{=}& 
- (x_l+y_l\alpha) (z'_m+w'_m\alpha') + (z_l+w_l\alpha) (x'_m+y'_m\alpha') 
\stackrel{(\ref{eq_of_ratio})}{=}0.
\end{eqnarray*}
Thus, $\mathcal T_{l+1}$ and $\mathcal T'_{m+1}$ have the same 4-ratio.
When both the 4-tuples of $\mathcal T_{l+1}$ and $\mathcal T'_{m+1}$ are Type 2, a similar argument holds. 

This concludes Lemma~\ref{TechLem}. 
\end{proof}

Finally we are ready to prove Theorem~\ref{BigThm}.

\begin{proof}[Proof of Theorem~\ref{BigThm}] 
By Lemma \ref{TechLem}, 
$\mathcal{T}_{l+t}$ and $\mathcal{T}_{m+t}'$ have the same 4-tuple types (Type 1 or Type 2) for all $t\geq 1.$

By Corollary~\ref{cor:rel_of_types} their topological types (I, II, I', II') satisfy 
either 
\begin{enumerate}
\item
$\Type(\T_{l+1})=\Type(\T_{m+1})$ or 
\item
$\Type(\T_{l+1})=\Type(\T_{m+1})'$. 
\end{enumerate}
The system of Types I, II, I', II' is closed under maximal splitting operations and follow the rule as described in Figure~\ref{figure-121'2'}. In particular, the diagram in Figure~\ref{figure-121'2'} is mirror symmetric with respect to a vertical axis. 
Thus, in Case (1) the I-II-I'-II'-sequences corresponding to $\T_{l+1} \msp \T_{l+2} \msp \cdots$ and $\T'_{m+1} \msp \T'_{m+2} \msp \cdots$ are exactly the same and in Case (2) exactly the same up to simultaneously putting $'$. Namely,  
\begin{enumerate}
\item
$\Type(\T_{l+t})=\Type(\T_{m+t})$ for all $t\geq 1$ or 
\item
$\Type(\T_{l+t})=\Type(\T_{m+t})'$ for all $t\geq 1$. 
\end{enumerate}
Lemma \ref{TechLem} also states that $\mathcal{T}_{l+t}$ and $\mathcal{T}_{m+t}'$ have the same 4-ratios. 
Therefore, for each case (Case (1) or (2)), there exists a homeomorphism $\phi \in {\rm Homeo}^+(D_3)$ such that for all $t\geq0,$
\begin{enumerate}
\item
$\phi(\mathcal{T}_{l+t})=\mathcal{T}_{m+t}'$ for all $t\geq 1$, or 
\item
$\phi(\mathcal{T}_{l+t})={\tt mirror}( \mathcal{T}_{m+t}')$ for all $t\geq 1$. 
\end{enumerate}
In Case (1), $\beta$ and $\beta'$ have equivalent Agol-cycles and in Case (2) they have mirror equivalent Agol cycles. 
\end{proof}

\section{Dilatation is preserved under 3-braid flypes}

The goal of this section is to prove Theorem~\ref{thm:dilatation}, which states that the dilatation is preserved under flype moves. 
In addition, as byproducts of this section, we describe how to compute the transition matrix of a given 3-braid (thus the dilatation) in the proof of Lemma~\ref{Matrices}.

\subsection{Birman-Menasco's classification of 3-braids} 

We briefly review well-known facts on 3-braid flypes that are relevant to this paper. 

\begin{defn}\label{def:flype}
Let $\varepsilon = \pm 1$ and $x, y, z \in \mathbb{Z}$. 
If $\beta=\sigma_1^x \sigma_2^{\varepsilon} \sigma_1^{y} \sigma_2^{z}$ and $\beta'=\sigma_1^x \sigma_2^{z} \sigma_1^{y} \sigma_2^{\varepsilon}$,
then we say that $\beta$ and $\beta'$ are related by an {\em $\varepsilon$-flype.}
See Figure \ref{fig3}.
If $\beta$ and $\beta'$ are conjugate (resp. not conjugate), then the flype is called {\em degenerate} (resp. {\em non-degenerate}). 
\end{defn}

\begin{figure}[h]
\includegraphics[height=4.0cm]{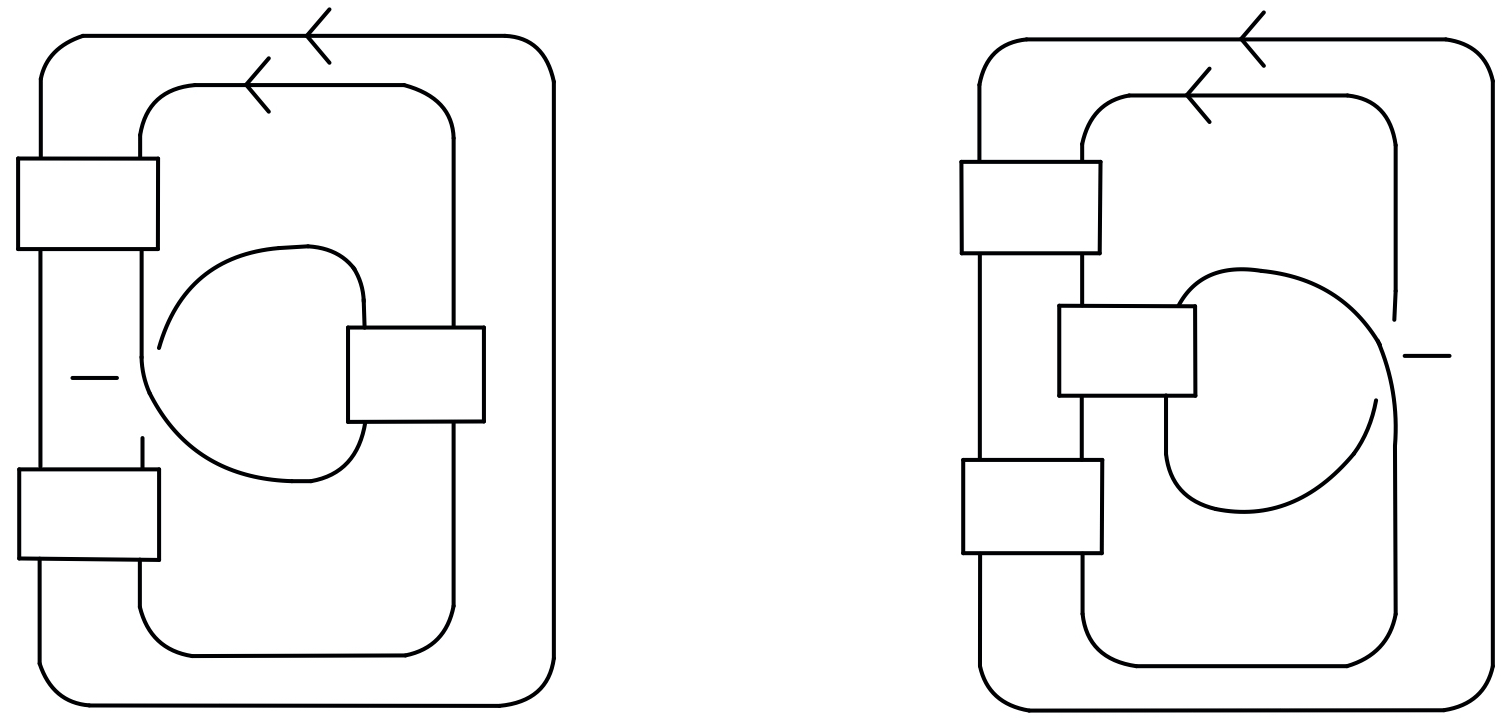} 
\put(-265, 55){\fontsize{13}{11}$\hat{\beta}=$}
\put(-115, 55){\fontsize{13}{11}$\hat{\beta'}=$}
\put(-228, 79){\fontsize{13}{11}$x$}
\put(-80, 78){\fontsize{13}{11}$x$}
\put(-228, 31){\fontsize{13}{11}$y$}
\put(-80, 32){\fontsize{13}{11}$y$}
\put(-178, 52){\fontsize{13}{11}$z$}
\put(-65, 55){\fontsize{13}{11}$z$}
\caption{(Definition~\ref{def:flype}) The braid closures $\hat{\beta}$ and $\hat{\beta'}$ obtained from a negative flype.}
\label{fig3}
\end{figure}
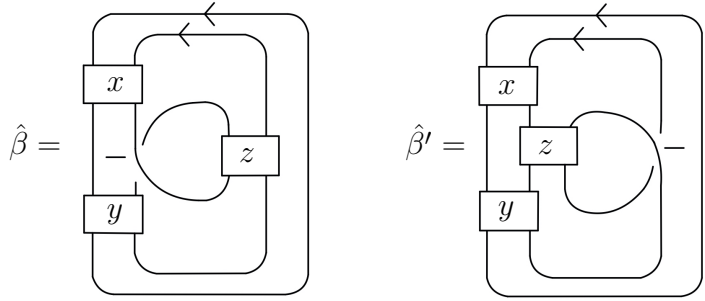

Non-degenerate flypes play a significant role in knot theory. 
Flypes are used in the classification of 3-braids in Birman and Menasco's work \cite{BM}, in Markov's Theorem without Stabilization \cite{BM2}, and the classification of transversally simple knots \cite{BMII, BM08}. 
Many other transversally simple knots admit a negative flype like those found by Etnyre and Honda \cite{EH} (cf. Matsuda and Menasco \cite{MM}) and Ng, Ozsv\'ath, and Thurston \cite{NOT}. 
In the Tait flype conjecture, Thistlethwaite and Menasco proved that two reduced alternating diagrams of an alternating link are related by a sequence of flypes \cite{MT}.

It is easy to see that a flype move preserves the topological link type of the braid closure. 
On the contrary, the following Birman and Menasco's 3-braid classification theorem  states that a flype changes the conjugacy class in general. 
\begin{thm}\cite{BM} \label{BM-3-braid-theorem}
Let $\mathcal L$ be a link type of braid index three. Then one of the following holds:
\begin{enumerate}
\item
There exists a unique conjugacy class of 3-braid representatives of $\mathcal L$. 
\item
There exists two conjugacy classes of 3-braid representatives of $\mathcal L$.  
\end{enumerate}
Case $(2)$ happens if and only if $\mathcal L$ has a 3-braid representative that admits a non-degenerate flype. 
\end{thm}

Moreover, Ko and Lee determine all the non-degenerate flypes. 

\begin{thm}\cite[Theorem 5]{Ko-Lee} \label{theorem:Ko-Lee}
The 3-braids $\beta=\sigma_1^x \sigma_2^{\varepsilon} \sigma_1^{y} \sigma_2^{z}$ and $\beta'=\sigma_1^x \sigma_2^{z} \sigma_1^{y} \sigma_2^{\varepsilon}$ have distinct conjugacy classes if and only if 
\begin{itemize}
\item
Neither $x$ nor $y$ is equal to $0, \, \varepsilon, \, 2\varepsilon$ or $z+\varepsilon$,
\item
$x \neq y$, and
\item
$|z| \geq 2$.
\end{itemize}
\end{thm}


\subsection{Dilatation is preserved under flypes}\label{subsec:dilatation}

Now we state the main result of this section. 
\begin{thm}\label{thm:dilatation}
There are infinitely many integers 
$x, y$ and $z$, such that 
the braids $\beta = \sigma_1^x \sigma_2^{-1} \sigma_1^{y} \sigma_2^z$ and 
$\beta'=\sigma_1^x \sigma_2^z \sigma_1^y \sigma_2^{-1}$ 
belong to distinct conjugacy classes but 
have the same dilatation  
\[\lambda = \frac{1}{2}(\gamma + \sqrt{\gamma^2-4})\]
where 
\begin{equation}\label{eq:trace}
\gamma=\gamma(x,y,z)=\sgn(xyz)(-2 -x-y+xz+yz +xyz).
\end{equation}
\end{thm}

The theorem is an immediate consequence of Theorem~\ref{theorem:Ko-Lee} and  Lemma~\ref{Matrices} below.  

We need three lemmas to prove Theorem~\ref{thm:dilatation}.
We start by analyzing behavior of train tracks under the braid generators $\s_1$ and $\s_2$, and compute transition matrices.

\begin{lem}\label{lemma0}
Measured train tracks are affected by $\sigma_1^{\pm 1}$ and $\sigma_2^{\pm1}$ as shown in Figures \ref{fig1} and \ref{fig2}. 
The matrices between train track types are defined below. They designate how the labels change after the application of $\sigma_i^{\pm1}$. 
\begin{eqnarray*}
&& A=\begin{pmatrix} 1 & 0 \\ 1 & 1 \end{pmatrix}, \quad A^{-1}=\begin{pmatrix} 1 & 0 \\ -1 & 1 \end{pmatrix},\quad  V=\begin{pmatrix} 1 & 0 \\ 1 & -1 \end{pmatrix}, \label{def_of_matrices}\\
&& B=\begin{pmatrix} 1 & 1 \\ 0 & 1 \end{pmatrix}, \quad B^{-1}=\begin{pmatrix} 1 & -1 \\ 0 & 1 \end{pmatrix}, \quad T=\begin{pmatrix} -1 & 1 \\ 0 & 1 \end{pmatrix}. 
\label{def_of_matrices2}
\end{eqnarray*}
Note $T B^{-1} = B T$ and $VA^{-1}=AV$. 
\end{lem} 

\begin{proof}
In this proof, we unzip and zip the edges of the train track, which is an operation where we separate or condense, respectively, the edges according to the assigned weights. More details about this operation can be found in \cite{Primer}.

We consider the case of $\sigma_1$ acting on $\M(a,b)$. 
\[
\includegraphics[height=1.9cm]{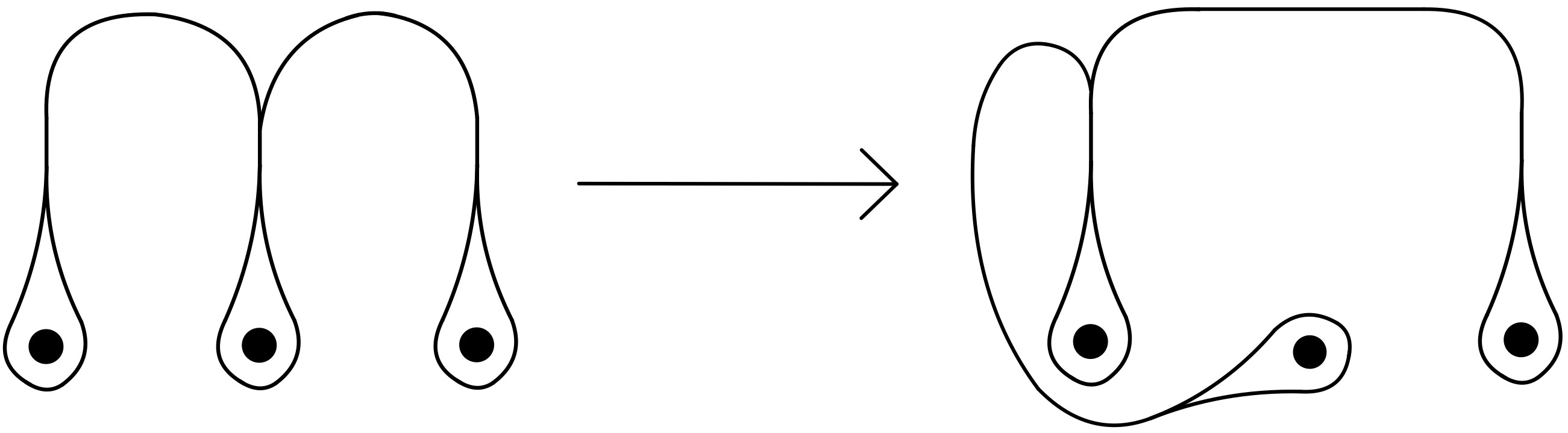}
\put(-200,40){$a$}
\put(-133,40){$b$}
\put(-110,35){$\sigma_1$}
\put(-80,10){$a$}
\put(-33,40){$b$}
\] 

If $a<b$, then $2a < a+b$ and we unzip the train track along $a$.
\[
\includegraphics[height=2.0cm]{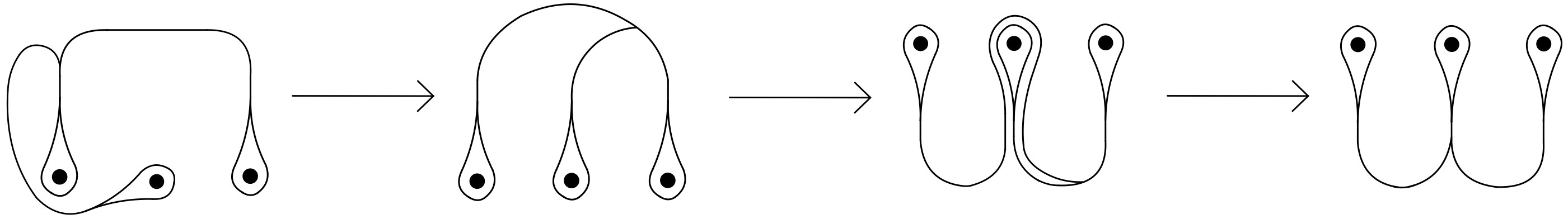}
\put(-413,30){$a$}
\put(-368,38){$b$}
\put(-333,40){Unzip $a$}
\put(-280,60){$-a+b$}
\put(-265,40){$a$}
\put(-220,40){Isotopy}
\put(-195,2){$-a+b$}
\put(-140, 2){$a$}
\put(-95,40){Zip}
\put(-65,-5){$-a+b$}
\put(-20,-5){$b$}
\] 
Notice that the labels have changed by an application of the matrix $T$. 

If $a>b$, then $a+b > 2b$ and we unzip the train track along $b$. 
\[
\includegraphics[height=1.9cm]{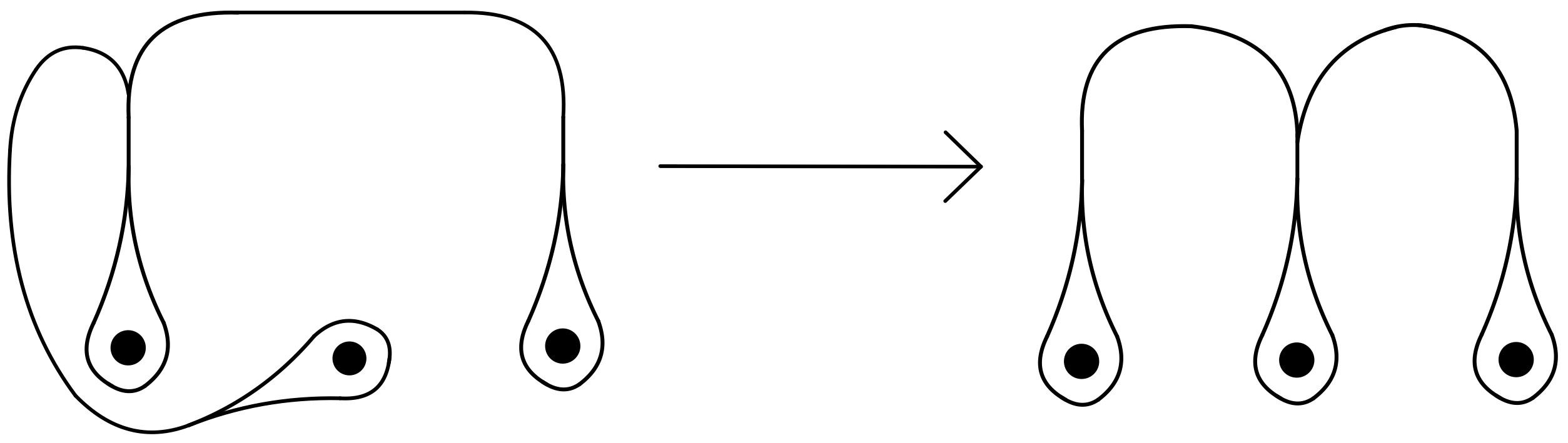}
\put(-200,30){$a$}
\put(-133,40){$b$}
\put(-110,40){Unzip $b$}
\put(-60,55){$a-b$}
\put(-20,55){$b$}
\] 
Notice that the labels have changed by an application of the matrix $B^{-1}$. 

All other cases follow similarly. 
\end{proof}

\begin{figure}[h]
\includegraphics[height=10cm]{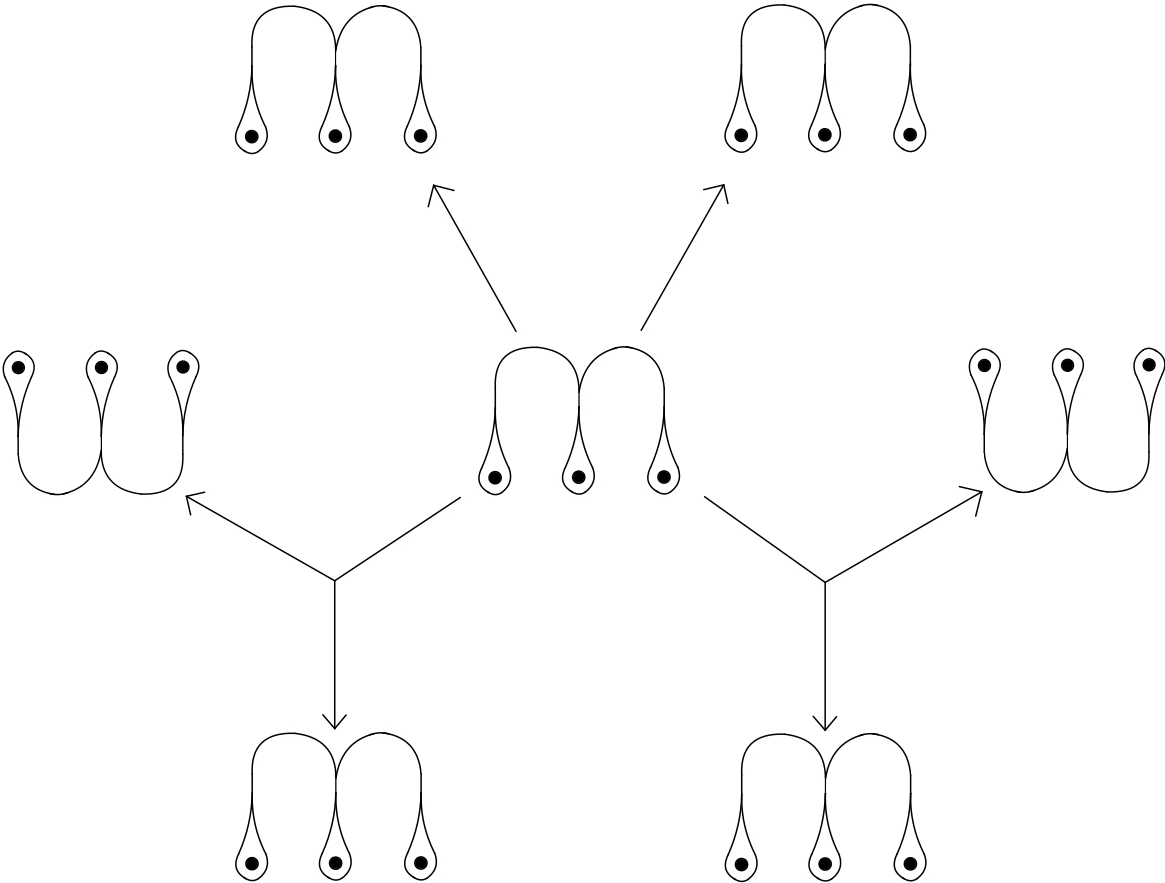}
\put(-145,33){$a$} 
\put(-80,33){$-a+b$}
\put(-320, 33){$a-b$} 
\put(-235, 33){$b$}
\put(-380, 115){$-a+b$} 
\put(-330, 115){$b$}
\put(-50, 115){$a$} 
\put(-25, 115){$a-b$}
\put(-207, 160){$a$} 
\put(-177, 160){$b$}
\put(-145, 270){$a$} 
\put(-80, 270){$a+b$}
\put(-320, 270){$a+b$} 
\put(-235, 270){$b$}
\put(-245, 200){$\sigma_1^{-1}$}
\put(-220, 200){$B$}
\put(-150, 200){$\sigma_2$}
\put(-167, 200){$A$}
\put(-255, 117){$\sigma_1$}
\put(-135, 117){$\sigma_2^{-1}$}
\put(-95, 110){$V$}
\put(-90, 100){$a>b$}
\put(-130, 70){$A^{-1}$}
\put(-105, 70){$a<b$}
\put(-285, 110){$T$}
\put(-312, 100){$a<b$}
\put(-290, 70){$B^{-1}$}
\put(-265, 70){$a>b$}
\caption{(Lemma \ref{lemma0}) Action of $\sigma_1^{\pm 1}$ and $\sigma_2^{\pm1}$ on Type M}
\label{fig1}
\end{figure}

\begin{figure}[h]
\includegraphics[height=10cm]{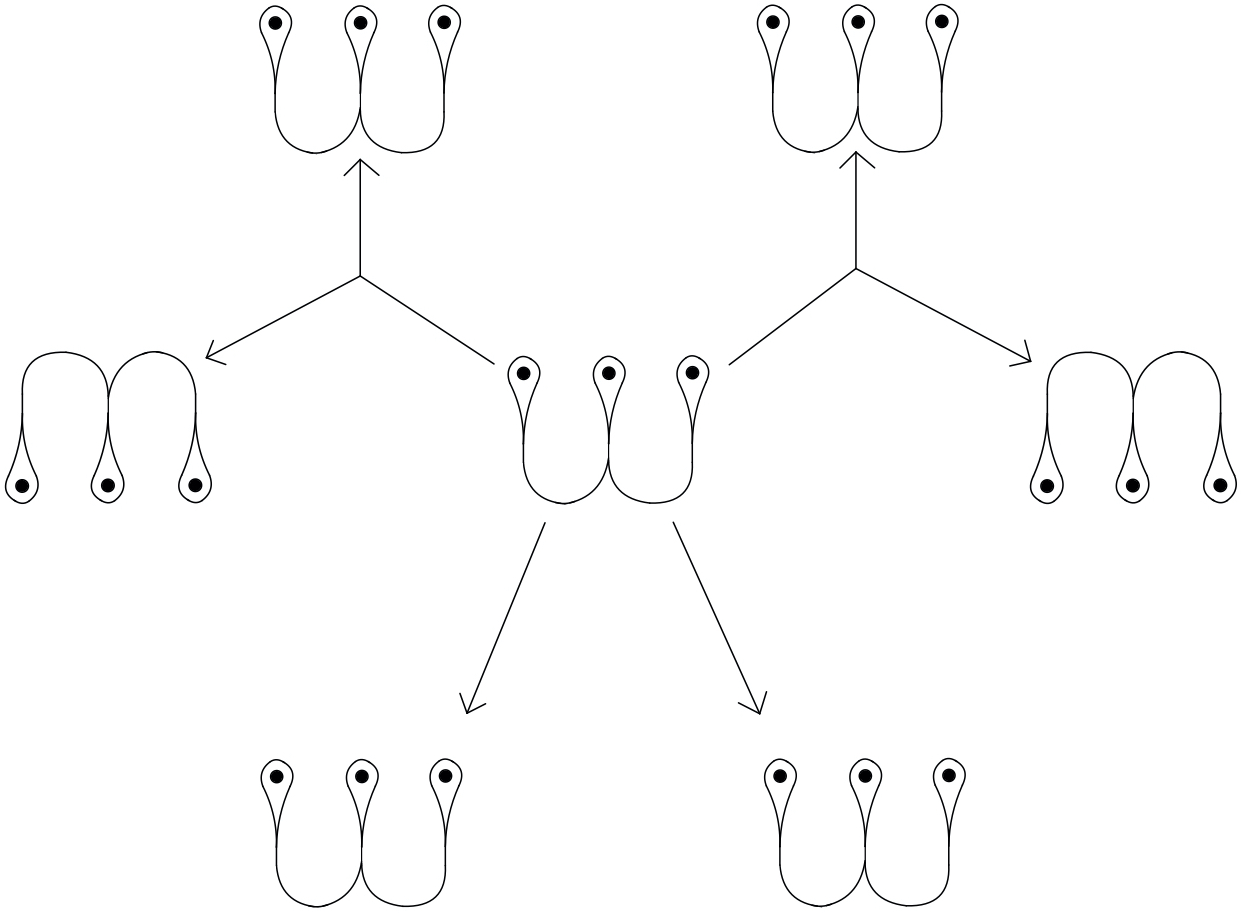}
\put(-155,13){$a$} 
\put(-90,13){$a+b$}
\put(-330, 13){$a+b$} 
\put(-245, 13){$b$}
\put(-390, 180){$-a+b$} 
\put(-340, 180){$b$}
\put(-55, 180){$a$} 
\put(-30, 180){$a-b$}
\put(-215, 133){$a$} 
\put(-185, 133){$b$}
\put(-155, 250){$a$} 
\put(-90, 250){$-a+b$}
\put(-330, 250){$a-b$} 
\put(-245, 250){$b$}
\put(-245, 90){$\sigma_1$}
\put(-225, 90){$B$}
\put(-160, 90){$\sigma_2^{-1}$}
\put(-177, 90){$A$}
\put(-255, 190){$\sigma_1^{-1}$}
\put(-150, 190){$\sigma_2$}
\put(-85, 188){$V$}
\put(-115, 175){$a>b$}
\put(-142, 210){$A^{-1}$}
\put(-117, 210){$a<b$}
\put(-310, 188){$T$}
\put(-305, 175){$a<b$}
\put(-298, 210){$B^{-1}$}
\put(-273, 210){$a>b$}
\caption{(Lemma \ref{lemma0}) Action of $\sigma_1^{\pm 1}$ and $\sigma_2^{\pm1}$ on Type W}
\label{fig2}
\end{figure}

\begin{lem}\label{lemma1}
Based on Lemma~\ref{lemma0}, we obtain four commutative diagrams. 
\begin{itemize}
\item
The action of $\sigma_1$ is shown in the following commutative diagram. 
For $x\gg 0$, Type M converges to Type W under $\sigma_1$ and the change in weights on the train track is represented by $B^{x-1} T$.

\begin{tikzcd}
M(a, b) \arrow[rd, "T", "a<b"'] \arrow[r, "B^{-1}", "a>b"'] 
  & M(a-b, b) \arrow[rd, "T", "a<2b"'] \arrow[r, "B^{-1}", "a>2b"'] 
    & M(a-2b, b) \arrow[rd, "T", "a<3b"'] \arrow[r, "B^{-1}", "a>3b"'] & \cdots\cdots \\ 
& W(-a+b, b) \arrow[r, "B"] 
  & W(-a+2b, b) \arrow[r, "B"]
    & \cdots\cdots
\end{tikzcd}

\item
The action of $\sigma_1^{-1}$ is shown the following commutative diagram. 
For $x\gg 0$, Type W converges to Type M under $\sigma_1^{-x}$ and the change in weights on the train track is represented by $B^{x-1} T$. 

\begin{tikzcd}
W(a, b) \arrow[rd, "T", "a<b"'] \arrow[r, "B^{-1}", "a>b"'] 
  & W(a-b, b) \arrow[rd, "T", "a<2b"'] \arrow[r, "B^{-1}", "a>2b"'] 
    & W(a-2b, b) \arrow[rd, "T", "a<3b"'] \arrow[r, "B^{-1}", "a>3b"'] & \cdots\cdots \\ 
& M(-a+b, b) \arrow[r, "B"] 
  & M(-a+2b, b) \arrow[r, "B"]
    & \cdots\cdots
\end{tikzcd}

\item
The action of $\sigma_2$ is shown in the following commutative diagram. 
For $x\gg 0$, Type W converges to Type M under $\sigma_2$ and the change in weights on the train track is represented by $A^{x-1} V$. 

\begin{tikzcd}
W(a, b) \arrow[rd, "V", "a>b"'] \arrow[r, "A^{-1}", "a<b"'] 
  & W(a, -a+b) \arrow[rd, "V", "2a>b"'] \arrow[r, "A^{-1}", "2a<b"'] 
    & W(a, -2a+b) \arrow[rd, "V", "3a>b"'] \arrow[r, "A^{-1}", "3a<b"'] & \cdots\cdots \\ 
& M(a, a-b) \arrow[r, "A"] 
  & M(a, 2a-b) \arrow[r, "A"]
    & \cdots\cdots
\end{tikzcd}

\item
The action of $\sigma_2^{-1}$ is shown in the following commutative diagram.
For $x\gg 0$, Type M converges to Type W under $\sigma_2^{-x}$ and the change in weights on the train track is represented by $A^{x-1} V$. 

\begin{tikzcd}
M(a, b) \arrow[rd, "V", "a>b"'] \arrow[r, "A^{-1}", "a<b"'] 
  & M(a, -a+b) \arrow[rd, "V", "2a>b"'] \arrow[r, "A^{-1}", "2a<b"'] 
    & M(a, -2a+b) \arrow[rd, "V", "3a>b"'] \arrow[r, "A^{-1}", "3a<b"'] & \cdots\cdots \\ 
& W(a, a-b) \arrow[r, "A"] 
  & W(a, 2a-b) \arrow[r, "A"]
    & \cdots\cdots
\end{tikzcd}

\end{itemize}
\end{lem}


Let $\beta \in B_3 \simeq \MCG(D_3)$ be a pseudo-Anosov 3-braid. 
Suppose that $\beta$ has an invariant train track of Type X, where X = M or W. That is, if we apply $\beta$ to a Type X train track, then we return to the same type of train track. 
A {\em transition matrix} $M$ tells how the weights of the train track edges have changed after applying $\beta$. 
The key idea of the upcoming Lemma~\ref{Matrices} is to compute the type of $\beta$ and the transition matrix by applying Lemmas~\ref{lemma0} and \ref{lemma1}. 

\begin{lemma}\label{Matrices}
Let $\beta = \sigma_1^x \sigma_2^{-1} \sigma_1^{y} \sigma_2^z$ and 
$\beta'=\sigma_1^x \sigma_2^z \sigma_1^y \sigma_2^{-1}$ be pseudo-Anosov 3-braids related by a negative flype.  
For large $x, y$ and $z$, the transition matrices $M$ and $M'$ associated to $\beta$ and $\beta'$ respectively are the following: 
\begin{eqnarray*}
M&=& \sgn(xyz)
\left(  
\begin{matrix} 
      -1-y & \sgn(z) (x+y+xy) \\
      \sgn(z)(1-z-yz) & -1-x+xz+yz+xyz\\
   \end{matrix}
\right)\\
M' &=&  \sgn(xyz)
\left(
\begin{matrix} 
     -1+yz & -x-y+xyz \\
     -1+z+yz & -1-x-y+xz+xyz\\
   \end{matrix}
\right)
\end{eqnarray*}
Furthermore, we have $\det(M)=\det(M')=1$ and 
\begin{equation*}\label{eq:trace}
\tr(M)=\tr(M')=\gamma(x,y,z)=\sgn(xyz)(-2 -x-y+xz+yz +xyz).
\end{equation*}
\end{lemma}

In the proof of Lemma~\ref{Matrices}
we separate all 3-braids into eight cases, depending on the sign of $x,y$, and $z$: 
$$
   \begin{array}{c|ccc} 
& x & y & z \\
\hline   
\mbox{ Case 1} & - & - & - \\
\hline
\mbox{ Case 2} & - & - & + \\
\hline
\mbox{ Case 3} & - & + & - \\
\hline
\mbox{ Case 4} & - & + & + \\
\hline
\mbox{ Case 5} & + & - & - \\
\hline
\mbox{ Case 6} & + & - & + \\
\hline
\mbox{ Case 7} & + & + & - \\
\hline
\mbox{ Case 8} & + & + & + \\
\hline
\end{array}
$$

\begin{remark}\label{rem:Case3and5}
There is a bijection between the set of Case 3 braids and the set of Case 5 braids. 
Suppose $\x, \y , \z>0$. 
Then 
$$
\textrm{(Case 3) } \s_1^{-\x} \s_2^{-1} \s_1^\y \s_2^{-\z} 
\stackrel{\rm conj}{\sim}
\s_2^{-\z}\s_1^{-\x} \s_2^{-1} \s_1^\y 
\stackrel{\rm rev}{\sim}
\s_1^\y \s_2^{-1}\s_1^{-\x}\s_2^{-\z} 
\textrm{ (Case 5)} $$
where $\stackrel{\rm conj}{\sim}$ means conjugation and $\stackrel{\rm rev}{\sim}$ means reverse orientation or read the word backward. 
Thus swapping $\x$ and $\y$ gives a bijection between the two sets. 

Similarly, we can find a bijection between the sets for Case 4 and Case 6 braids by swapping $\x$ and $\y$. 
\end{remark}

\begin{proof}[Proof of Lemma~\ref{Matrices}]
We calculate the matrices $M_i$ and $M_i'$ associated to $\beta$ and $\beta'$ respectively for each Case $i$. 
Below we manipulate the braid word of $\beta$ and $\beta'$ in each case to force $x,y,$ and $z$ to be positive in our calculations. \\

\noindent{\bf Case 1:} ($\beta = \sigma_1^{-x} \sigma_2^{-1} \sigma_1^{-y} \sigma_2^{-z}$ and $\beta'=\sigma_1^{-x} \sigma_2^{-z} \sigma_1^{-y} \sigma_2^{-1}$) 
We use Lemmas~\ref{lemma0} and \ref{lemma1} in the calculations for $\beta$ and $\beta'$. 
\[
\beta: \text{W} \xrightarrow[\sigma_1^{-x}]{B^{x-1}T} \text{M} \xrightarrow[\sigma_2^{-1}]{V} \text{W} 
\xrightarrow[\sigma_1^{-y}]{B^{y-1}T} \text{M} 
\xrightarrow[\sigma_2^{-z}]{A^{z-1}V} \text{W} \qquad \mbox{ (Type W)}
\]
\[
\beta': \text{W} \xrightarrow[\sigma_1^{-x}]{B^{x-1}T} \text{M} \xrightarrow[\sigma_2^{-z}]{A^{z-1}V} \text{W} 
\xrightarrow[\sigma_1^{-y}]{B^{y-1}T} \text{M} 
\xrightarrow[\sigma_2^{-1}]{V} \text{W} \qquad \mbox{ (Type W)}
\]
Both $\beta$ and $\beta'$ have Type W invariant measured train tracks. 
Transition matrices are:
\begin{eqnarray}
M_1 &=& A^{z-1} V B^{y-1} T V B^{x-1}T
\nonumber\\
&=& \begin{pmatrix} 1-y & -x-y+xy \\ 1+z-yz & 1-x-xz-yz+xyz \end{pmatrix} \label{matrixM1} \\ 
M_1' &=& V B^{y-1}T A^{z-1}V B^{x-1}T 
\nonumber \\
&=&\begin{pmatrix} 1-yz & -x-y+xyz \\ 1+z-yz & 1-x-y-xz+xyz\end{pmatrix} 
\label{matrixM1'}
\end{eqnarray}

Below we turn $M_1$ and $M_1'$ into Perron-Frobenius by taking conjugates of the original matrices. 
When $x \geq 3, y\geq 3$, and $z\geq 2$, we can verify that $ A^{-1} M_1 A$ is a non-negative integral (i.e., Perron Frobenius) matrix. 
Here are the computations. 
The $(1,1)$ element of the matrix $ A^{-1} M_1 A$ is 
$$xy-x-2y+1 = (x-2)(y-1)-1 \geq 0.$$  
The $(1,2)$ element is $$(x-1)(y-1)-1\geq 0.$$
The $(2,1)$ element is $$(x-2)(z-1)\left((y-1)-\frac{1}{x-2}-\frac{1}{z-1}\right)\geq 0.$$
The $(2,2)$ element is $$(x-1)(z-1)\left((y-1)-\frac{1}{x-1}-\frac{1}{z-1}\right)\geq 0.$$

If $x\geq 3, y\geq 3$, and $z\geq 2$, the matrix $C^{-1} M_1' C $ where 
$C:=\left( 
\begin{matrix}
      2 & 1 \\
      1 & 1 \\
\end{matrix}\right)$
is also Perron-Frobenius. Indeed, the $(1,1)$ element is $$z(x-2)-1\geq 0.$$ 
The $(1,2)$ element is $$z(x-1)-1\geq 0.$$
The $(2,1)$ element is $$(x-2)(y-2)\left( z- \frac{1}{y-2} - \frac{1}{x-2}\right) \geq 0.$$
The $(2,2)$ element is $$(x-1)(y-2)\left( z- \frac{1}{y-2} - \frac{1}{x-1}\right)\geq 0.$$


\noindent{\bf Case 2:} ($\beta = \sigma_1^{-x} \sigma_2^{-1} \sigma_1^{-y} \sigma_2^{z}$ and $\beta'=\sigma_1^{-x} \sigma_2^{z} \sigma_1^{-y} \sigma_2^{-1}$) 
We have: 
\[
\beta: \text{M} \xrightarrow[\sigma_1^{-x}]{B^{x}} \text{M} \xrightarrow[\sigma_2^{-1}]{V} \text{W} \xrightarrow[\sigma_1^{-y}]{B^{y-1}T} \text{M} \xrightarrow[\sigma_2^{z}]{A^{z}} \text{M} 
\qquad \mbox{ (Type M)}
\]
\[
\beta': \text{W} \xrightarrow[\sigma_1^{-x}]{B^{x-1}T} \text{M} \xrightarrow[\sigma_2^{z}]{A^z} \text{M}  \xrightarrow[\sigma_1^{-y}]{B^y} \text{M} \xrightarrow[\sigma_2^{-1}]{V} \text{W} 
\qquad \mbox{ (Type W)}
\]
\begin{eqnarray*}
M_2 &=& A^{z} B^{y-2} T B^{-1} V B^{x}\\
&=& \begin{pmatrix} -1+y & -x-y+xy \\ 1-z+yz & -1+x-xz-yz+xyz \end{pmatrix} \\ 
M_2' &=& VB^{y}A^{z} B^{x-1}T \\
&=&\begin{pmatrix} -1-yz & x+y+xyz \\ -1+z-yz & -1+x+y-xz+xyz\end{pmatrix} 
\end{eqnarray*}
For $x, y, z \gg 1$, both $M_2$ and $C^{-1}M_2' C$ are Perron Frobenius.\\


\noindent{\bf Case 3:} ($\beta = \sigma_1^{-x} \sigma_2^{-1} \sigma_1^{y} \sigma_2^{-z}$ and $\beta'=\sigma_1^{-x} \sigma_2^{-z} \sigma_1^{y} \sigma_2^{-1}$) 
We have: 
\[
\beta: \text{W} \xrightarrow[\sigma_1^{-x}]{B^{x-1}T} \text{M} \xrightarrow[\sigma_2^{-1}]{V} \text{W} \xrightarrow[\sigma_1^{y}]{B^y} \text{W} \xrightarrow[\sigma_2^{-z}]{A^z} \text{W}
\qquad \mbox{ (Type W)}
\]
\[
\beta': \text{W} \xrightarrow[\sigma_1^{-x}]{B^{x-1}T} \text{M} \xrightarrow[\sigma_2^{-z}]{A^{z-1} V} \text{W}  
\xrightarrow[\sigma_1^{y}]{B^y} \text{W} \xrightarrow[\sigma_2^{-1}]{A} \text{W} 
\qquad \mbox{ (Type W)}
\]
\begin{eqnarray*}
M_3 &=& A^{z} B^{y} V B^{x-1} T\\
&=& \begin{pmatrix} -1-y & x-y+xy \\ -1-z-yz & -1+x+xz-yz+xyz \end{pmatrix} \\ 
M_3' &=& AB^{y}A^{z-1}V B^{x-1}T \\
&=&\begin{pmatrix} -1-yz & x-y+xyz \\ -1-z-yz & -1+x-y+xz+xyz\end{pmatrix} 
\end{eqnarray*}
For $x, y, z \gg 1$, both $A^{-1}M_3 A$ and $A^{-1}M_3' A$ are Perron Frobenius.\\


\noindent{\bf Case 4:} ($\beta = \sigma_1^{-x} \sigma_2^{-1} \sigma_1^{y} \sigma_2^{z}$ and $\beta'=\sigma_1^{-x} \sigma_2^{z} \sigma_1^{y} \sigma_2^{-1}$) 
We have: 
\[
\beta: \text{M} \xrightarrow[\sigma_1^{-x}]{B^{x}} \text{M} \xrightarrow[\sigma_2^{-1}]{V} \text{W} \xrightarrow[\sigma_1^y]{B^y} \text{W} \xrightarrow[\sigma_2]{V} \text{M} \xrightarrow[\sigma_2^{z-1}]{A^{z-1}} \text{M}
\qquad \mbox{ (Type M)}
\]

\[
\beta': \text{W} \xrightarrow[\sigma_1^{-x}]{B^{x-1}T} \text{M} \xrightarrow[\sigma_2^{z}]{A^z} \text{M} \xrightarrow[\sigma_1]{T} \text{W} \xrightarrow[\sigma_1^{y-1}]{B^{y-1}} \text{W} \xrightarrow[\sigma_2^{-1}]{A} \text{W}
\qquad \mbox{ (Type W)}
\]

\begin{eqnarray*}
M_4 &=& A^{z-1}V B^{y} V B^{x}\\
&=& \begin{pmatrix} 1+y & x-y+xy \\ -1+z+yz & 1-x+xz-yz+xyz \end{pmatrix} \\ 
M_4' &=& AB^{y-1}TA^{z}B^{x-1}T \\
&=&\begin{pmatrix} 1-yz & -x+y+xyz \\ 1-z-yz & 1-x+y+xz+xyz\end{pmatrix} 
\end{eqnarray*}
For $x, y, z \gg 1$, both $M_4$ and $A^{-1}M_4' A$ are Perron Frobenius.\\


\noindent{\bf Case 5:} ($\beta = \sigma_1^{x} \sigma_2^{-1} \sigma_1^{-y} \sigma_2^{-z}$ and $\beta'=\sigma_1^{x} \sigma_2^{-z} \sigma_1^{-y} \sigma_2^{-1}$) We have: 
\[
\beta: \text{W} \xrightarrow[\sigma_1^{x}]{B^{x}} \text{W} \xrightarrow[\sigma_2^{-1}]{A} \text{W} \xrightarrow[\sigma_1^{-1}]{T} \text{M} \xrightarrow[\sigma_1^{-y+1}]{B^{y-1}} \text{M} \xrightarrow[\sigma_2^{-1}]{V} \text{W} \xrightarrow[\sigma_2^{-z+1}]{A^{z-1}} \text{W}
\qquad \mbox{ (Type W)}
\]
\[
\beta': \text{W} \xrightarrow[\sigma_1^{x}]{B^{x}} \text{W} \xrightarrow[\sigma_2^{-z}]{A^z} \text{W} \xrightarrow[\sigma_1^{-1}]{T} \text{M} \xrightarrow[\sigma_1^{-y+1}]{B^{y-1}} \text{M} \xrightarrow[\sigma_2^{-1}]{V} \text{W}
\qquad \mbox{ (Type W)}
\]
\begin{eqnarray*}
M_5 &=& A^{z-1}V B^{y-1} T A B^{x}\\
&=& \begin{pmatrix} -1+y & -x+y+xy \\ -1-z+yz & -1-x-xz+yz+xyz \end{pmatrix} \\ 
M_5' &=& VB^{y-1}TA^{z}B^{x} \\
&=&\begin{pmatrix} -1+yz & -x+y+xyz \\ -1-z+yz & -1-x+y-xz+xyz\end{pmatrix}  
\end{eqnarray*}
For $x, y, z \gg 1$, both $M_5$ and $M_5'$ are Perron Frobenius.\\


\noindent{\bf Case 6:} ($\beta = \sigma_1^{x} \sigma_2^{-1} \sigma_1^{-y} \sigma_2^{z}$ and $\beta' = \sigma_1^x \sigma_2^{z} \sigma_1^{-y} \sigma_2^{-1}$) We have: 
\[
\beta: \text{M} \xrightarrow[\sigma_1^{x}]{B^{x-1}T} \text{W} \xrightarrow[\sigma_2^{-1}]{A} \text{W} \xrightarrow[\sigma_1^{-y}]{B^{y-1}T} \text{M} \xrightarrow[\sigma_2^{z}]{A^{z}} \text{M} 
\qquad \mbox{ (Type M)}
\]
\[
\beta': \text{W} \xrightarrow[\sigma_1^{x}]{B^x} \text{W} \xrightarrow[\sigma_2]{V} \text{M} \xrightarrow[\sigma_2^{z-1}]{A^{z-1}} \text{M}  \xrightarrow[\sigma_1^{-y}]{B^y} \text{M}  \xrightarrow[\sigma_2^{-1}]{V} \text{W}
\qquad \mbox{ (Type W)}
\]
\begin{eqnarray*}
M_6 &=& A^z B^{y-1}T A B^{x-1}T \\
&=& \begin{pmatrix} 1-y & -x+y+xy \\ -1+z-yz & 1+x-xz+yz+xyz \end{pmatrix} \\ 
M_6' &=& VB^y A^{z-1}V B^x\\
&=&\begin{pmatrix} 1+ yz & x-y+xyz \\ 1-z+yz & 1+x-y-xz+xyz\end{pmatrix} 
\end{eqnarray*}
For $x, y, z \gg 1$, both $A^{-1}M_6A$ and $M_6'$ are Perron Frobenius.\\

\noindent{\bf Case 7:} ($\beta = \sigma_1^{x} \sigma_2^{-1} \sigma_1^{y} \sigma_2^{-z}$ and $\beta' = \sigma_1^x \sigma_2^{-z} \sigma_1^{y} \sigma_2^{-1}$) 
We have: 
\[
\beta: \text{W} \xrightarrow[\sigma_1^{x}]{B^x} \text{W} \xrightarrow[\sigma_2^{-1}]{A} \text{W} \xrightarrow[\sigma_1^{y}]{B^y} \text{W} \xrightarrow[\sigma_2^{-z}]{A^z} \text{W} 
\qquad \mbox{ (Type W)}
\]
\[
\beta': \text{W} \xrightarrow[\sigma_1^{x}]{B^x} \text{W} \xrightarrow[\sigma_2^{-z}]{A^z} \text{W} \xrightarrow[\sigma_1^{y}]{B^y} \text{W}  \xrightarrow[\sigma_2^{-1}]{A} \text{W}
\qquad \mbox{ (Type W)}
\]
\begin{eqnarray*}
M_7 &=& A^z B^y A B^x= \begin{pmatrix} 1+ y & x+y+xy \\ 1+z+yz & 1+x+xz+yz+xyz \end{pmatrix} \\ 
M_7' &=& AB^y A^z B^x=\begin{pmatrix} 1+ yz & x+y+xyz \\ 1+z+yz & 1+x+y+xz+xyz\end{pmatrix} 
\end{eqnarray*}
Both $M_7$ and $M_7'$ are Perron Frobenius.


\noindent{\bf Case 8:} ($\beta = \sigma_1^{x} \sigma_2^{-1} \sigma_1^{y} \sigma_2^{z}$ and $\beta' = \sigma_1^x \sigma_2^{z} \sigma_1^{y} \sigma_2^{-1}$) We have:
\[
\beta: \text{M} \xrightarrow[\sigma_1^{x}]{B^{x-1}T} \text{W} \xrightarrow[\sigma_2^{-1}]{A} \text{W} \xrightarrow[\sigma_1^{y}]{B^y} \text{W} \xrightarrow[\sigma_2]{V} \text{M} \xrightarrow[\sigma_2^{z-1}]{A^{z-1}} \text{M} 
\qquad \mbox{ (Type M)}
\]
\[
\beta': \text{W} \xrightarrow[\sigma_1^{x}]{B^x} \text{W} \xrightarrow[\sigma_2]{V} \text{M} \xrightarrow[\sigma_2^{z-1}]{A^{z-1}} \text{M}  \xrightarrow[\sigma_1]{T} \text{W}  \xrightarrow[\sigma_1^{y-1}]{B^{y-1}} \text{W} \xrightarrow[\sigma_2^{-1}]{A} \text{W}
\qquad \mbox{ (Type W)}
\]
\begin{eqnarray*}
M_8 &=& A^{z-1} V B^{y} A B^{x-1}T \\
&=& \begin{pmatrix} -1-y & x+y+xy \\ 1-z-yz & -1-x+xz+yz+xyz \end{pmatrix} \\ 
M_8' &=& A B^{y-1} T A^{z-1}V B^x\\
&=&\begin{pmatrix} -1+ yz & -x-y+xyz \\ -1+z+yz & -1-x-y+xz+xyz\end{pmatrix}  
\end{eqnarray*}
For $x, y, z \gg 1$, both $A^{-1}M_8A$ and $M_8'$ are Perron Frobenius.
\end{proof}

\begin{remark}\label{rem_on_Case1}

In Proposition~\ref{Matrices}, we assume that $x, y, z \gg 1$. 
More precisely in Case 1, for example, we require $(x-1)-\frac{1}{(y-1)-\frac{1}{z}} > \alpha$. 

Let $a $ and $b$ be positive integers. 
After scaling, we can identify W$(a, b)$ with W$(\alpha, 1)$ for some $\alpha=\frac{a}{b}$. 
We start with W$(\alpha, 1)$ and apply $\sigma_1^{-x}$. 
If $x>\alpha$, we obtain 
$$W(\alpha, 1) \stackrel{\sigma_1^{-x}}{\longrightarrow} M(x-\alpha, 1).$$ 
If $x-1>\alpha$, we then obtain 
$$M(x-\alpha, 1)\stackrel{\sigma_2^{-1}}{\longrightarrow} W(x-\alpha, x-1-\alpha).$$
If $(x-1) - \frac{1}{y-1}>\alpha$, we obtain 
$$W(x-\alpha, x-1-\alpha)\stackrel{\sigma_1^{-y}}{\longrightarrow}
M( (y-1)(x-1)-1-\alpha(y-1), (x-1)-\alpha).$$
If $(x-1)-\frac{1}{(y-1)-\frac{1}{z}}>\alpha$, we further obtain 
\begin{eqnarray}
&&M( (y-1)(x-1)-1-\alpha(y-1), (x-1)-\alpha) 
\nonumber 
\\
& \stackrel{\sigma_2^{-z}}{\longrightarrow} &
W((y-1)(x-1)-1-\alpha(y-1), (x-1) (zy-z-1)-z - \alpha(zy-z-1)) \nonumber 
\\
&=& 
W((1-y)\alpha -x-y+xy, (1+z-yz)\alpha + (1-x-xz-yz+xyz)) \label{expression}
\end{eqnarray}
The expression in (\ref{expression}) matches with the matrix $M_1$ in (\ref{matrixM1}). 
Therefore, as long as 
\begin{equation}\label{validxyz}
(x-1)-\frac{1}{(y-1)-\frac{1}{z}} > \alpha
\end{equation}
holds, the transition matrix $M_1$ is valid. 
\end{remark}

\bibliographystyle{alpha}

\bibliography{references}

\end{document}